\documentclass[a4paper]{article}
\usepackage [a4paper,left=3 cm,bottom=2.5cm,right=3cm,top=2.5cm]{geometry}
\usepackage[english]{babel}
\usepackage{amssymb}
\usepackage{amsmath,amsthm}
\usepackage{subcaption}
\usepackage{tabularx,array}
\usepackage{graphicx, url}
\usepackage{enumitem} 
\usepackage{xcolor}
\DeclareMathOperator{\sgn}{sgn}

\newcommand{\E}{\mathbb{E}}

\newcommand{\R}{\mathbb{R}}

\makeatletter
\newcommand*{\barfix}[2][.175ex]{%
  \mathpalette{\@barfix{#1}}{#2}%
}
\newcommand*{\@barfix}[3]{%
  \vbox{%
    \kern#1\relax
    \hbox{$#2#3\m@th$}%
  }%
}
\makeatother

\newcommand{\blue}{\color{black}}
\newcommand{\modar}{\color{black}}
\newcommand{\modarn}{\color{black}}
\newcommand{\abs}[1]{\left|#1\right|}
\newcommand{\norm}[1]{\left\lVert#1\right\rVert}

\newtheorem{remark}{Remark}
\newtheorem{theorem}{Theorem}
\newtheorem{lemma}{Lemma}
\newtheorem{corollary}{Corollary}
\newtheorem{proposition}{Proposition}
\newtheorem{definition}{Definition}

\newcommand{\keywords}[1]{\par\addvspace\baselineskip
\noindent\enspace\ignorespaces#1}
\newcommand{\modch}{\color{black}}
\newcommand{\modchi}{\color{black}}
\newcommand{\change}{\color{black}}
\newcommand{\revar}{\color{black}}
\newcommand{\revarn}{\color{black}}
\newcommand{\rev}{\color{black}}
\begin{document}

\title{Minimax rate of estimation for invariant densities associated to continuous stochastic differential equations over anisotropic {\rev H\"older} classes.}
\author{Chiara Amorino\thanks{ Universit\'e du Luxembourg, L-4364 Esch-Sur-Alzette, Luxembourg. The author gratefully acknowledges financial support of ERC Consolidator Grant 815703 “STAMFORD: Statistical Methods for High Dimensional Diffusions”.} \qquad Arnaud Gloter \thanks{Laboratoire de Math\'ematiques et Mod\'elisation d'Evry, CNRS, Univ Evry, Universit\'e Paris-Saclay, 91037, Evry, France.}} 

\maketitle

\begin{abstract}
We study the problem of the nonparametric estimation for the density $\pi$ of the stationary distribution of a $d$-dimensional stochastic differential equation $(X_t)_{t \in [0, T]}$. From the continuous observation of the sampling path on $[0, T]$, we study {\change the estimation} of $\pi(x)$ as $T$ goes to infinity. For $d\ge2$, we characterize the minimax rate for the $\mathbf{L}^2$-risk in pointwise estimation over a class of  anisotropic H\"older functions $\pi$ with regularity $\beta = (\beta_1, ... , \beta_d)$. 
For $d \ge 3$, our finding is that, having ordered the smoothness such that $\beta_1 \le ... \le \beta_d$, the minimax rate  depends on whether $\beta_2 < \beta_3$ or $\beta_2 = \beta_3$. In the first case, this rate is $(\frac{\log T}{T})^\gamma$, and in the second case, it is $(\frac{1}{T})^\gamma$, where $\gamma$ is an explicit exponent dependent on the dimension and $\bar{\beta}_3$, the harmonic mean of smoothness over the $d$ directions after excluding $\beta_1$ and $\beta_2$, the smallest ones. We also demonstrate that kernel-based estimators achieve the optimal minimax rate. Furthermore, we propose an adaptive procedure {\blue for both $L^2$ integrated and pointwise risk.} In the two-dimensional case, we show that kernel density estimators achieve the rate $\frac{\log T}{T}$, which is optimal in the minimax sense. {\blue Finally we illustrate the validity of our theoretical findings by proposing numerical results.}


\end{abstract}

\section{Introduction}

In this work, we propose kernel density estimators for estimating the invariant density associated with a stochastic differential equation. Specifically, we investigate inference for the following d-dimensional process
\begin{equation}
X_t= X_0 + \int_0^t b( X_s)ds + \int_0^t a(X_s)dW_s, \quad t \in [0,T], 
\label{eq: model intro}
\end{equation}
with $b : \mathbb{R}^d \rightarrow \mathbb{R}^d$, $a : \mathbb{R}^d \rightarrow \mathbb{R}^d \otimes \mathbb{R}^d$ and $W = (W_t, t \ge 0)$ a d-dimensional Brownian motion.

Stochastic differential equations are very attractive for statisticians nowadays, as they model stochastic evolution as time evolves. These models have a variety of applications in many disciplines and emerge in a natural way in the study of many phenomena. {\rev For example, there are} applications in astronomy \cite{202 Iacus}, mechanics \cite{147 Iacus}, physics \cite{176 Iacus}, geology \cite{69 Iacus}, ecology \cite{111 Iacus}, biology \cite{194 Iacus} and epidemiology \cite{17 Iacus}. Some other examples are economics \cite{26 Iacus}, mathematical finance \cite{115 Iacus}, political analysis and social processes \cite{55 Iacus} as well as neurology \cite{109 Iacus}, genetic analysis \cite{155 Iacus}, cognitive psychology \cite{221 Iacus} and biomedical sciences \cite{20 Iacus}.

Given the significance of stochastic differential equations, there has been extensive research into inference for such models. Authors have explored both parametric and non-parametric inference, beginning with continuous or discrete observations, and within various asymptotic frameworks. These frameworks include small diffusions asymptotics over a fixed time interval and long time intervals for ergodic models. Some landmarks in the area are the books 
of Kutoyants \cite{Kut04}, Iacus \cite{Iac10}, Kessler et al \cite{Kes12} and H{\"o}pfner \cite{Hop14}. In the meantime, a significant number of papers have been published on this topic. Among them, we mention Comte et al \cite{Com07}, Dalalyan and Reiss \cite{RD}, Genon-Catalot \cite{Gen90}, Gobet et al \cite{Gob04}, Hoffmann \cite{Hof99}, Lar\'edo \cite{Lar90} and Yoshida \cite{Yos92}.  \\
The in-depth study on stochastic differential equations lead the way for statistical inference for more complicated models such as SDEs driven by L\'evy processes (see for example \cite{Mas19}), diffusions with jumps (\cite{Mas07}, \cite{Sch19}), stochastic partial differential equations (\cite{Alt20}, \cite{Cia17}), diffusions with mixed effects (\cite{Pic10}, \cite{Del18}) and Hawkes processes (\cite{Bac13}, \cite{Dio19}).

In this paper, we propose a kernel density estimator and aim to determine the convergence rates for pointwise $L^2$ estimation of the stationary measure $\pi$ associated with the process $X$ defined by \eqref{eq: model intro}. Additionally, we will assess the optimality of these rates.

 The estimation of the invariant measure is a problem already widely faced in many different contexts by many authors, we quote \cite{Ngu79}, \cite{Del80}, \cite{Bos98}, \cite{Zan01}, \cite{Strauch new}, \cite{Kut04} and \cite{Banon} among them.
Nowadays, this problem attracts the attention of many statisticians due to the multitude of numerical methods associated with it, with the Markov Chain Monte Carlo method being the most prominent. For instance, in \cite{LamPag02} and \cite{Pan08}, one can find approximation algorithms for computing the invariant density, while the analysis of invariant distributions is employed to assess the stability of stochastic differential systems in \cite{Has80} and \cite{Banon}. It can also be used in order to estimate the drift coefficient in a non-parametric way (see \cite{Nic} and \cite{Sch}). {\modchi However, there are some problems involving the estimation of the invariant density for SDEs {\rev in the anisotropic context} which are still open, because of their challenging nature. One above all, the research of minimax convergence rate for the estimation of the invariant density associated to a multidimensional diffusion. The purpose of this work is to fill such a gap. }\\ 
As it is easy to see, the estimation of the invariant measure presents itself as both a long-standing problem and an actively researched area. Kernel estimators, for various reasons, serve as powerful tools in this context. For instance, in \cite{Banon} and \cite{Bosq9}, kernel estimators are employed to estimate the marginal density of a continuous-time process. In \cite{Optimal}, \cite{Chapitre 4}, and \cite{Man10}, they find utility within a jump-diffusion framework, with the latter specifically applied for estimating volatility.

In a context closely related to ours, kernel density estimators have been employed in \cite{RD} and \cite{Strauch} to investigate the convergence rate for estimating the invariant density associated with a reversible diffusion process with unit diffusion. Notably, in \cite{RD}, the authors establish convergence rates for pointwise estimation of the invariant density, as a by-product of their analysis, under isotropic H\"older smoothness constraints. The convergence rates they found are $\frac{\log T}{T}$ when the dimension $d$ is equal to $2$ and $(\frac{1}{T})^{\frac{2 \beta}{2 \beta + d - 2}}$ for $d \ge 3$, where $\beta$ is the common smoothness over the $d$ different directions. Strauch extended the work presented in \cite{RD} by developing adaptive kernel estimators for estimating the invariant density within anisotropic H\"older spaces. These estimators achieve identical convergence rates as those in \cite{RD}, with the replacement of $\beta$ by $\bar{\beta}$—the harmonic mean across the $d$ different directions. Given that the smoothness characteristics of function space elements can vary depending on the chosen direction in $\mathbb{R}^d$, the concept of anisotropy assumes a significant role. 

In this work we aim at estimating the invariant density $\pi$ by means of the kernel estimator $\hat{\pi}_{h,T}$ assuming to have the continuous record of the process $X$ solution to \eqref{eq: model intro} up to time $T$. We first of all prove the following upper bound for the mean squared error, for $d \ge 3$: {\rev there exist $T_0$ and $c > 0$ such that, for any $T \ge T_0$, }
\begin{equation}
\sup_{(a,b)\in \Sigma} \mathbb{E} [|\hat{\pi}_{h,T}(x) - \pi (x)|^2] \le 
\begin{cases}
c(\frac{\log T}{T})^{\frac{2\bar{\beta}_3}{2\bar{\beta}_3 + d - 2}} \qquad \mbox{for } \beta_2 < \beta_3, \\
c(\frac{1}{T})^{ \frac{2\bar{\beta}_3}{2\bar{\beta}_3+ d - 2}} \qquad \mbox{for } \beta_2 = \beta_3,
\end{cases}
\label{eq: conv rate intro}
\end{equation}
where $\Sigma$ is a class of coefficients for which the {\rev stationary} density has some prescribed regularity, $\beta_1 \le \beta_2 \le ... \le \beta_d$ and $\bar{\beta}_3$ is the harmonic mean over the smoothness after having removed the two smallest. In particular, we have $\frac{1}{ \barfix{\bar{\beta}_3}} := \frac{1}{d -2} \sum_{l \ge 3} \frac{1}{\beta_l}$ and so it clearly follows that $\bar{\beta}_3 \ge \bar{\beta}$. {\change It is surprising that the presence of the logarithm in the convergence rate depends on whether $\beta_2=\beta_3$ or not. This variation arises from different methods of bounding the variance, resulting in varying degrees of freedom for the optimal bandwidth $h = (h_1, \ldots, h_d)$. In some specific cases, as detailed in Remark 1, we can eliminate the logarithmic term.} \\
{\blue We also introduce an adaptive procedure based on the Goldenshluger-Lepski method, which considers both the integrated and pointwise $L^2$ risk. This procedure selects a bandwidth that strikes an optimal balance between bias and variance based on the available data. The motivation behind studying the data-selection procedure with respect to the integrated $L^2$ risk is to prevent the potential loss of a logarithm in our estimation, as discussed in \cite{BroLow} for example. This data-driven procedure is capable of automatically detecting the presence of a logarithmic term in the convergence rate.} \\
{\modchi It is important to note that while the convergence rate determined earlier, dependent on $\bar{\beta}_3$, may not appear immediately intuitive, it aligns with the findings in \cite{Lower bound} for estimating the invariant density associated with a L\'evy-driven SDE. In the mentioned paper, numerical results underscore the distinct role played by the two smallest levels of smoothness compared to the others. In our work, we delve into the specifics of this role, and with the result presented in \eqref{eq: conv rate intro}, we can elucidate their impact on the convergence rate. Furthermore, we remove the boundedness condition on the drift coefficient.}
We also prove the following upper bound, in the bidimensional context: {\rev there exist $T_0$ and $c > 0$ such that, for any $T \ge T_0$, }
$$\sup_{(a,b)\in \Sigma} \mathbb{E} [|\hat{\pi}_{h,T}(x) - \pi (x)|^2] \le c \frac{\log T}{T}.$$
{\change The convergence rate described above is consistent with that in \cite{RD} and \cite{Strauch}. However, the question of whether this rate is optimal had remained unanswered until now. To our knowledge, the only existing results in this context are some lower bounds for the sup-norm risk, as demonstrated in \cite{Strauch} and analogous findings in the jump-framework discussed in \cite{Optimal}.\\
In this paper we obtain a minimax lower bound on the $L^2$-risk for the pointwise estimation with the same rate $\frac{\log T}{T}$. }
 To do that, we introduce the minimax risk 
$$ 	\mathcal{R}_T:= \inf_{\tilde{\pi}_T} \sup_{(a,b) \in \Sigma} \mathbb{E}[(\tilde{\pi}_T (x_0) - \pi_{(a,b)} (x_0))^2],$$
where the infimum is taken over all the possible estimators of the invariant density. Then, we show that {\rev there exist $c> 0$ and $ T_0 > 0$ such that, for $T \ge T_0$,}
$$\mathcal{R}_T \ge {\rev c} \frac{\log T}{T}.$$
{\change For $d=2$, it is not possible to propose an estimator with a better convergence rate than the one we have found for the kernel density estimator we introduced.} \\
Regarding the case $d \ge 3$ we prove that, in general, the following lower bound holds true: {\rev there exists $c > 0$ and $T_0 > 0$ such that, for any $T \ge T_0$,}
$$\mathcal{R}_T \ge c(\frac{1}{T})^{ \frac{2\bar{\beta}_3}{2\bar{\beta}_3+ d - 2}}.$$
{\modchi This result is not very surprising, as an analogous lower bound has been showed in \cite{Lower bound} for the estimation of the invariant density associated to a L\'evy driven SDE. What is instead quite surprising is that,} when $\beta_2 < \beta_3$, it is possible to improve the result here above by showing {\rev there exist $T_0$ and $c > 0$ such that, for any $T \ge T_0$, }
$$\mathcal{R}_T \ge c(\frac{\log T}{T})^{ \frac{2\bar{\beta}_3}{2\bar{\beta}_3+ d - 2}}.$$
These results highlight how crucial the condition $\beta_2 < \beta_3$ is. They imply that, on a class of diffusions $X$ whose invariant density has the prescribed regularity, it is not possible to find an estimator with rates of estimation faster than in \eqref{eq: conv rate intro}. It follows that the kernel density estimator achieves the best possible convergence rates.

It is worth remarking that the upper bound for the mean squared error relies on an upper bound for the transition density and the mixing properties of the process. Due to the latter, the constants involved are dependent on the coefficients, and it can be quite challenging, in general, to establish uniform control over these constants. However, relying on the theory of Lyapounov-Poincar\'e as introduced in \cite{Bak_et_al08}, we prove a mixing inequality uniform on the class of coefficient $(a,b) \in \Sigma$ (see Proposition \ref{P:mixing uniform}).
{\revar {\change We remark that the Lyapounov-Poincar\'e method can be extended to get uniform mixing inequality also in more general frameworks, such as the jump context.}} \\
We therefore get a minimax upper bound on the risk of the estimator $\hat{\pi}_{h,T}$ which is uniform over the class of coefficient $\Sigma$ and we complement it by obtaining a minimax lower bound on the risk of the estimator $\hat{\pi}_{h,T}$.
{\blue Finally, we present numerical results to demonstrate the successful validation of our theoretical findings.}

The outline of the paper is the following. In Section \ref{s: model}, we introduce the model and its assumptions. In Section \ref{S: Estimator}, we present the kernel estimator and provide upper bounds for the mean squared error. {\revarn The Section \ref{Ss:Adaptive} is devoted to the adaptive version of the estimator and associated upper bounds.}
Section \ref{S: Lower bounds} is dedicated to stating complementary lower bounds, complementing the results in the previous section. {\revarn In Section \ref{S:Numerical}, we present numerical results on the adaptive estimators.} In Section \ref{s: proofs upper} we prove the upper bounds stated in Section \ref{S: Estimator} {\revarn while Section \ref{S: proof adaptive} contains proofs related to the adaptive version of the estimators. The Section \ref{s: proof lower} is devoted to the proof of the lower bounds. Some technical results are moreover showed in the appendix.} 

\section{Model Assumptions}{\label{s: model}}
We want to estimate in a non parametric way the invariant density associated to a d-dimensional diffusion process $X$. In the sequel we assume that a continuous record of the process $X^T = \left \{  X_t, 0 \le t \le T \right \}$ up to time $T$ is available. The diffusion is a solution of the following stochastic differential equation:
\begin{equation}
X_t= X_0 + \int_0^t b( X_s)ds + \int_0^t a(X_s)dW_s, \quad t \in [0,T], 
\label{eq: model}
\end{equation}
where $b : \mathbb{R}^d \rightarrow \mathbb{R}^d$, $a : \mathbb{R}^d \rightarrow \mathbb{R}^d \times \mathbb{R}^d$ and $W = (W_t, t \ge 0)$ is a d-dimensional Brownian motion. The initial condition $X_0$ and $W$ are independent. We denote $\tilde{a}:= a \cdot a^T$. \\
 We denote with $|.|$ and $<., . >$ respectively the Euclidian norm and the scalar product in $\mathbb{R}^d$, and for a matrix in $\mathbb{R}^d \otimes \mathbb{R}^d$ we denote its operator norm by  $|\cdot|$. \\
 \\
\textbf{A1}: \textit{The functions $b(x)$ and $a(x)$ are globally Lipschitz functions of class $\mathcal{C}^1$, such that for all $x\in\mathbb{R}^d$,
	$$|b(0)|\le b_0, ~ |a(x)|\le a_0,~|\frac{\partial}{\partial x_i} b(x)|\le b_1, ~ |\frac{\partial}{\partial x_i} a(x)|\le a_1 ,\text{ for $i \in \{1,\dots,d\}$},$$
	where $b_0> 0,~a_0>0,~a_1>0,~b_1>0$ are some constants.
	Moreover, for some $a_{\text{min}} > 0$,
	$$a_{\text{min}}^2 \mathbb{I}_{d \times d} \le \tilde{a}(x) $$
	where $\mathbb{I}_{d \times d}$ denotes the $d \times d$ identity matrix.} \\
	\\
	\textbf{A2 (Drift condition) }: \textit{ \\
 There exist $\tilde{C} > 0$ and $\tilde{\rho} > 0$ such that $<x, b(x)>\, \le -\tilde{C} |x|$, $\forall x : |x| \ge \tilde{\rho}$.
 } \\
\\
The assumptions here above involve the ergodicity of the process and they are needed in order to show the existence of a Lyapounov function {\modchi (see for example \cite{Bak_et_al08})}. Hence, the process $X$ admits a unique invariant distribution $\mu$ and the ergodic theorem holds (see Lemma \ref{lemma: ergodicity} below). We suppose that the invariant probability measure $\mu$ of $X$ is absolutely continuous with respect to the Lebesgue measure and from now on we will denote its density as $\pi$: $ d\mu = \pi dx$. \\
\\
As in several cases the regularity of some function $g: \mathbb{R}^d \rightarrow \mathbb{R}$ depends on the direction in $\mathbb{R}^d$ chosen, we decide to work under anisotropic smoothness constraints. In particular, we assume that the density $\pi$ we want to estimate belongs to the anisotropic H{\"o}lder class $\mathcal{H}_d (\beta, \mathcal{L})$ defined below.
\begin{definition}
Let $\beta = (\beta_1, ... , \beta_d)$, $\beta_i \ge 0$, $\mathcal{L} =(\mathcal{L}_1, ... , \mathcal{L}_d)$, $\mathcal{L}_i > 0$. A function $g : \mathbb{R}^d \rightarrow \mathbb{R}$ is said to belong to the anisotropic H{\"o}lder class $\mathcal{H}_d (\beta, \mathcal{L})$ of functions if, for all $i \in \left \{ 1, ... , d \right \}$,
$$\left \| D_i^k g \right \|_\infty \le \mathcal{L}_i \qquad \forall k = 0,1, ... , \lfloor \beta_i \rfloor, $$
$$\left \| D_i^{\lfloor \beta_i \rfloor} g(. + t e_i) - D_i^{\lfloor \beta_i \rfloor} g(.) \right \|_\infty \le \mathcal{L}_i |t|^{\beta_i - \lfloor \beta_i \rfloor} \qquad \forall t \in \mathbb{R},$$
for $D_i^k g$ denoting the $k$-th order partial derivative of $g$ with respect to the $i$-th component, $\lfloor \beta_i \rfloor$ denoting the largest integer strictly smaller than $\beta_i$ and $e_1, ... , e_d$ denoting the canonical basis in $\mathbb{R}^d$.
\end{definition}

	This leads us to consider a class of coefficients $(a,b)$ for which the stationary density $\pi=\pi_{(a,b)}$ has some prescribed H\"older regularity.
 	\begin{definition}
		Let $\beta = (\beta_1, ... , \beta_d)$, $\beta_i \ge 0$ and $\mathcal{L} =(\mathcal{L}_1, ... , \mathcal{L}_d)$, $\mathcal{L}_i > 0$, $0<a_\text{min}\le a_0$ and $b_0>0, ~a_1>0, ~b_1>0, ~\tilde{C}>0, ~\tilde{\rho}>0$.
		
		We define $\Sigma (\beta, \mathcal{L},a_{\text{min}},b_0,a_0,a_1,b_1,\tilde{C},\tilde{\rho})$ the set of couple of functions $(a,b)$ where  $a: \mathbb{R}^d \rightarrow \mathbb{R}^d\otimes\mathbb{R}^d$ and $b: \mathbb{R}^d \rightarrow \mathbb{R}^d$ are such that 
		\begin{itemize}
			\item 		$a$ and $b$ satisfy Assumption A1 with the constants $(a_\text{min},b_0,a_0,a_1,b_1)$,
			\item 		$b$ satisfies Assumption A2 with the constants $(\tilde{C},\tilde{\rho})$,
			\item  the density $\pi_{(a,b)}$ of the invariant measure associated to the stochastic differential equation \eqref{eq: model} belongs to $\mathcal{H}_d (\beta, 2 \mathcal{L})$.
		\end{itemize}
		\label{def: insieme sigma v2}
	\end{definition}

In the sequel we will use repeatedly some known results about the transition density and the ergodicity of diffusion processes.
{\modchi The bound on the transition density in case of unbounded drift follows from \cite{MenPes}. As we will see, comparing it with the results known for the bound of the transition density in case of bounded drift, the main difference is that the departed point is now replaced by the flow of the departed point. For fixed $(s, x) \in \mathbb{R}^+ \times \mathbb{R}^d$, we denote by $\theta_{t,s}(x)$ the deterministic flow solving
$$\dot{\theta}_{t,s}(x) = b(\theta_{t,s}(x)), \quad t \ge 0, \qquad \theta_{s,s}(x) = x.$$
Then, the first point of Theorem 1.2 of \cite{MenPes} ensures the following bounds: 
\begin{lemma}{[Theorem 1.2 of \cite{MenPes}]}
Under A1, for any $T > 0$ $(s, t) \in [0, \infty)^2$, $0 < t-s < T$ and $x \in \mathbb{R}^d$, the unique weak solution of \eqref{eq: model} admits a density $p_{t - s}(x,y)$ which is continuous in $x, y \in \mathbb{R}^d$. Moreover, there exist $\lambda_0 \in (0, 1]$, $C_0 \ge 1$ depending on $T, a_0, a_{min}, d,$ and the Lipschitz constants in A1 such that, for any $(s,t)$ such that $0< t-s < T$ and $x,y \in \mathbb{R}^d$ it is 
\begin{align}{\label{eq: bound Pesce}}
C_0^{-1} (t-s)^{- \frac{d}{2}} \exp{\big( - \lambda_0^{-1} \frac{|\theta_{t,s}(x) - y|^2}{t-s} \big)} \le p_{t-s}(x,y) \le C_0 (t-s)^{- \frac{d}{2}} \exp{\big( - \lambda_0 \frac{|\theta_{t,s}(x) - y|^2}{t-s}\big)}{\rev.}
\end{align}
\label{l: bound with flow}
\end{lemma}
The lemma here above implies the following corollary, whose proof can be found in the Appendix. 
\begin{corollary}
Under A1-A2, there exist $\tilde{\lambda_0} \in (0, 1]$, $\tilde{C_0}, \Tilde{C_1} \ge 1$ and $c \ge 1$ such that, for any $(s,t)$ such that $0< t-s < T$ and $x,y \in \mathbb{R}^d$ it is 
\begin{align}
	\label{E:equation borne sup inf gaussienne}
\tilde{C_0} (t-s)^{- \frac{d}{2}} e^{ - \tilde{\lambda_0}^{-1} \frac{|x - y|^2}{t-s} } e^{- c (|x|^2 + 1) |t -s|} \le p_{t-s}(x,y)\le \tilde{C_1} (t-s)^{- \frac{d}{2}} e^{ - \tilde{\lambda_0} \frac{|x - y|^2}{t-s} } e^{ c (|x|^2 + 1) |t -s|}.
\end{align}
Moreover, the constants depend only on $T ,~a_\text{min},~b_0,~a_0,~ a_1$, $b_1$ and $d$.
\label{lemma: bound transition density}
\end{corollary}

}

\begin{lemma}
Suppose that A1-A2 hold true. Then, {\rev the} process $X$ admits a unique invariant measure $\pi$ and it is exponentially $\beta$-mixing.
\label{lemma: ergodicity}
\end{lemma}

{\modchi We will provide the proof of Corollary \ref{lemma: bound transition density} in the appendix,} while the ergodic property stated in Lemma \ref{lemma: ergodicity} can be found in \cite{Par_Ver}.

\section{Estimator and upper bounds}\label{S: Estimator}

Given the observation $X^T$ of a diffusion $X$, solution of \eqref{eq: model}, we propose to estimate the invariant density $\pi \in \mathcal{H}_d (\beta, \mathcal{L})$ by means of a kernel estimator.
We therefore introduce some kernel function $K: \mathbb{R} \rightarrow \mathbb{R}$ satisfying 
\begin{equation}
\int_\mathbb{R} K(x) dx = 1, \quad \left \| K \right \|_\infty < \infty, \quad \mbox{supp}(K) \subset [-1, 1], \quad \int_\mathbb{R} K(x) x^l dx = 0,
\label{eq: properties K}
\end{equation}
for all $l \in \left \{ 0, ... , M \right \}$ with $M \ge \max_i \beta_i$. \\
For $j \in \left \{ 1, ... , d \right \}$, we denote by $X_t^j$ the $j$-th component of $X_t$. A natural estimator of $\pi$ at $x= (x_1, ... , x_d)^T \in \mathbb{R}^d$ in the anisotropic context is given by 
\begin{equation}
\hat{\pi}_{h,T}(x) = \frac{1}{T \prod_{l = 1}^d h_l} \int_0^T \prod_{m = 1}^d K(\frac{x_m - X_u^m}{h_m}) du = : \frac{1}{T} \int_0^T \mathbb{K}_h(x - X_u) du.
\label{eq: def estimator}
\end{equation}
The multi-index $h=(h_1, ... , h_d) $ is small. In particular, we assume $h_i < 1$ for any $i \in \{1, ... , d \} $. \\
{\modchi A simple reason why this estimator results appropriate for the estimation of the invariant density is that, thanks to Birkhoff's ergodic theorem, 
$\hat{\pi}_{h,T}(x) = \frac{1}{T} \int_0^T \mathbb{K}_h(x - X_u) du$ converges almost surely, for $T \rightarrow \infty$, to $\E[\mathbb{K}_h(x - X_0)] = (\mathbb{K}_h \ast \pi) (x)$, which converges to $\pi(x)$ for $h \rightarrow 0$.} \\

	In order to get minimax upper bound on the risk of the estimator $\hat{\pi}_{h,T}$ we will need to apply a mixing
inequality uniform on the class of coefficients $(a,b) \in \Sigma (\beta, \mathcal{L},a_{\text{min}},b_0,a_0,a_1,b_1,\tilde{C},\tilde{\rho})$. The following lemma will be useful in this sense. {\modch In the sequel, we will denote as $\mathbb{E}$ the expectation 
	with respect to {\modar the stationary process $(X_t)_{t\ge0}$, whose invariant probability {\rev density} is $\pi = \pi_{(a,b)}$.}} \\
{\rev Let us introduce the notation $\pi(f)$ for the expectation of the function $f$ with respect to the invariant density $\pi$, i.e.
$\pi(f):= \int_{\R^d} f(x) \pi(x) dx$. Then, the following bounds hold true:}

\begin{lemma} \label{L:hyp sigma stronger}
{ \modch Let $\beta=(\beta_1,\dots,\beta_d)$, $\beta_i\ge 0$, $\mathcal{L}=(\mathcal{L}_1,\dots,\mathcal{L}_d)$, $\mathcal{L}_i>0$, $0<a_\text{min}\le a_0$, $b_0>0, ~a_1>0, ~ b_1>0$, $\tilde{C}>$, $\tilde{\rho}>0$.
 There exists a constant $C>0$ independent of $(a,b) \in \Sigma := \Sigma (\beta, \mathcal{L},a_{\text{min}},b_0,a_0,a_1,b_1,\tilde{C},\tilde{\rho})$ such that}, for all measurable function $\varphi : \mathbb{R}^d \to \mathbb{R}$ bounded, we have
	\begin{equation}\label{E: resu mixing reg non reg}
		\abs{	{\modch \mathbb{E}}[\varphi(X_t)\varphi(X_0)] - \pi(\varphi)^2} \le C e^{-t/C} \norm{\varphi}_\infty^2.
	\end{equation}
\end{lemma}

As immediate consequence of the fact that $\mathbb{K}$ is compactly supported we get the proposition below, which ensures the uniformity of the constants involved in the mixing properties of the process.  

\begin{proposition} \label{P:mixing uniform}
{\modch Let us consider the same notation as in Lemma \ref{L:hyp sigma stronger}. Then, there exist constants $\rho_{\text{mix}}>0$ and $C_\text{mix}>0$} such that $\forall h=(h_1,\dots,h_d)$ with $h_i<1$,
	\begin{equation}\label{E:mix_unif_K}
	\abs{\E \left[\mathbb{K}_h(x - X_t)\mathbb{K}_h(x - X_0)\right]
	- \pi(\mathbb{K}_h)^2 }\le C_\text{mix} \, e^{-\rho_\text{mix} t} \norm{\mathbb{K}_h}_\infty^2.
	\end{equation}
Moreover, the constants $C_\text{mix}$ and $\rho_\text{mix}$ are uniform over the set of coefficients $(a,b)\in\Sigma$.
\end{proposition}

\subsection{Upper bounds}

In the sequel, we will denote as $\bar{\beta}_3$ the harmonic mean over the $d-2$ largest components: 
\begin{equation}
\frac{1}{\bar{\beta_3}} := \frac{1}{d -2} \sum_{l \ge 3} \frac{1}{\beta_l}.
\label{eq: beta3}
\end{equation}
Next theorem provides the convergence rate for the pointwise estimation of the invariant density for $d \ge 3$. Its proof can be found in Section \ref{s: proofs upper}.
\begin{theorem} {\modch 
Suppose that $d \ge 3$. Consider $0<a_\text{min}\le a_0$ and $b_0 > 0, ~a_1>0, ~ b_1>0$, $\tilde{C} > 0$, $\tilde{\rho} > 0$ and denote $\Sigma := \Sigma (\beta, \mathcal{L},a_{\text{min}},b_0,a_0,a_1,b_1,\tilde{C},\tilde{\rho})$. If $\beta_1 \le \beta_2 \le ... \le \beta_d$, then there exist $c > 0$ and $T_0 > 0$} such that, for $T \ge T_0$, the optimal choice for the multidimensional bandwidth $h$ yields the following convergence rates.
\begin{itemize}
    \item[$\bullet$] If $\beta_2 < \beta_3$, then
$$  \sup_{(a, b) \in \Sigma} \mathbb{E}[|\hat{\pi}_{h,T}(x) - \pi (x)|^2] \le c (\frac{\log T}{ T})^ {\frac{2\bar{\beta}_3}{2\bar{\beta}_3 + d - 2}}.$$
\item[$\bullet$] If otherwise $\beta_2 = \beta_3$, then 
$$ \sup_{(a, b) \in \Sigma} \mathbb{E}[|\hat{\pi}_{h,T}(x) - \pi (x)|^2] \le c (\frac{1}{T})^ {\frac{2\bar{\beta}_3}{2\bar{\beta}_3 + d - 2}}.$$
\end{itemize}
\label{th: upper bound d ge 3}
\end{theorem}
In the isotropic context we have, in particular, $\bar{\beta}_3 =\beta$, where $\beta$ is the common smoothness over the $d$ directions. The convergence rate we derive is therefore the same as in \cite{RD}, being equal to $(\frac{1}{T})^ {\frac{2\beta}{2 \beta + d - 2}}$.\\
{\modchi One may argue that the fact that the result depends on the harmonic mean of the smoothness only on $d-2$ direction is not very intuitive. However, the exponent found in the convergence rate here above is not completely surprising, as it appears also in \cite{Lower bound} for the estimation of the invariant density of a diffusion with jumps and it is proven to be optimal in this context, up to a logarithmic gap. What is surprising in this result is the fact that the role played by the logarithm is now completely explained. As we will see in the proof, when $\beta_2 = \beta_3$ a particular choice of the bandwidth allows us to get rid of the logarithm.\\
Moreover, compared with \cite{Lower bound}, the bounds are now uniform over the class of coefficients $\Sigma$ and the drift $b$ is no longer assumed to be bounded. } \\
\\
The asymptotic behavior of the estimator and so the proof of Theorem \ref{th: upper bound d ge 3} is based on the standard bias -variance decomposition. Hence, we need an upper bound on the variance, as in next proposition.

\begin{proposition}
{ \modch Suppose that A1 - A2 hold for some $a_\text{min}$, $b_0$, $a_0$, $a_1$, $b_1$, $\tilde{C}$ and $\tilde{\rho}$; that $d \ge 3$ and that $\pi \in \mathcal{H}_d (\beta,{\rev 2} \mathcal{L})$. } {\modar Suppose moreover that 
$\beta_1 \le ... \le \beta_d$, 
 and let $k_0=\max\{i \in \{1,\dots,d\}\mid \beta_1=\dots=\beta_i\}$}.
If $\hat{\pi}_{h,T}$ is the estimator given in \eqref{eq: def estimator}, then there exist $c > 0$ and $T_0 > 0$ such that, for $T \ge T_0$, the following bounds hold true.
\begin{itemize}
 \item[$\bullet$] If $k_0 = 1$ and $\beta_2 < \beta_3$ or $k_0 = 2$, then
\begin{equation}
Var(\hat{\pi}_{h,T}(x)) \le \frac{c}{T} \frac{\sum_{j = 1}^d |\log(h_j)|}{\prod_{l \ge 3} h_l}.
\label{eq: estim variance with log}
\end{equation}
\item[$\bullet$] If $k_0 \ge 3$, then 
\begin{equation}
Var(\hat{\pi}_{h,T}(x)) \le \frac{c}{T} \frac{1}{(\prod_{l =1}^{k_0} h_l)^{1 - \frac{2}{k_0}}(\prod_{l \ge k_0 + 1} h_l)}.
\label{eq: estim variance without log}
\end{equation}
\item[$\bullet$] If otherwise $k_0 = 1$ and $\beta_2 = \beta_3$, then 
$$Var(\hat{\pi}_{h,T}(x)) \le \frac{c}{T} \frac{1}{\sqrt{h_2 h_3}\prod_{l \ge 4} h_l}.$$
\end{itemize}
Moreover, {\modch the constants $c$ and $T_0$ are} uniform over the set of coefficients $$(a,b)\in\Sigma (\beta, \mathcal{L},a_{\text{min}},b_0, a_0,a_1,b_1,\tilde{C},\tilde{\rho}).$$
\label{prop: bound variance}
\end{proposition}
{\modchi We remark that, in the right hand side of \eqref{eq: estim variance with log}, it is possible to remove the contribution of two arbitrary bandwidths. We choose to remove the contribution of the first two as $h_1$ and $h_2$ are associated, in the bias term, to $\beta_1$ and $\beta_2$. Indeed,
the optimal bandwidth is chosen in order to get the balance between the bias and the variance term and,
as we order the smoothness, $\beta_1$ and $\beta_2$ provide the strongest constraints (see the proof of Theorem \ref{th: upper bound d ge 3}). \\}
\\
One can remark that the bound on the variance for $\beta_2 < \beta_3$ in the proposition here above is the same as the one for jump diffusion processes (see Proposition 1 in \cite{Lower bound}). The reason why it happens is that both propositions rely on the exponential $\beta$-mixing of the considered process and on a bound on the transition density. Comparing Lemma \ref{lemma: bound transition density} with Lemma 1 of \cite{Chapitre 4} it is possible to see that the upper bound for the transition density associated to a jump-diffusion consists in two terms: one derives from the gaussian component while the other is due to the presence of jumps. It is worth noting that in our computations the contribution of the second is always negligible compared to the one coming from the first. \\
\\
We now study the behaviour of our estimator for $d=2$. 
\begin{theorem}
{ \modch Suppose that $d=2$. Consider $0<a_\text{min}\le a_0$, $b_0>0, ~ a_1>0, ~ b_1>0$, $\tilde{C}$, $\tilde{\rho}$ and $\Sigma := \Sigma (\beta, \mathcal{L},a_{\text{min}},b_0, a_0,a_1,b_1,\tilde{C},\tilde{\rho})$. Then, there exist $c > 0$ and $T_0 > 0$ such that, for $T \ge T_0$,} the optimal choice for the multidimensional bandwidth $h$ yields the following convergence rates.
$$ \sup_{(a, b) \in \Sigma} \mathbb{E}[|\hat{\pi}_{h,T}(x) - \pi (x)|^2] { \rev \le c} \frac{\log T}{ T}.$$
\label{th: upper bound d=2}
\end{theorem}
To conclude the part regarding the upper bounds on the mean squared error associated to our estimator \eqref{eq: def estimator} when a continuous record of the process is available, we are left to discuss the mono-dimensional case. It is however known that, under our hypothesis, the proposed kernel estimator achieves the parametric rate $\frac{1}{T}$ and such a rate is optimal (see for example \cite{Kut_2004} or Theorem 1 in \cite{Kut_1998}).

{\rev \subsection{Adaptive procedure} \label{Ss:Adaptive}
 We observe that, in practice, the bandwidths $h_1$, ... , $h_d$ need to be selected from data. For this reason, it can be interesting to propose an adaptive procedure in the same spirit as the one firstly introduced by Goldenslugher and Lepski, in \cite{GL}. \\
As explained in previous subsection, for $d=1$ and $d=2$ the bound on the variance does not depend on the unknown smoothness, and the same holds for the optimal choice for the bandwidth. Hence, there is no gain in implementing a data-driven adaptive procedure for $d < 3$.
For $d \ge 3$, instead, the bandwidth choice depends on the smoothness $\beta$. We emphasize that, as we are in the anisotropic context, the smoothnesses over the different directions are different and, as a consequence, the bandwidth selection procedure has to be able to provide different choices for the bandwidth $h_1$, $h_2$, ... , $h_d$. To do that, the idea consists in providing a set of candidate bandwidths, associate to a set of potential estimators, and then in choosing the bandwidths such that the $L^2$ error is as small as possible. With this purpose in mind, we introduce a quantity that heuristically stands for the bias term and a penalty term whose size is the one found in the bound of the variance in previous subsection. The selected bandwidths will be the ones for which the sum of these two quantities achieves the minimum. \\
\\
{\blue The adaptive procedure we propose encompasses both integrated and pointwise $L^2$ risk.}\\
In the sequel, for $A \subset \mathbb{R}^d$ compact and for $g \in L^2(A)$, $\left \| g \right \|^2_A := \int_A |g(x)|^2 dx$ denotes the $L^2$ norm with respect to Lebesgue on $A$.
{\blue  One might question why, in the adaptive procedure, we also analyze the integrated $L^2(A)$ norm, whereas in the previous section, we focused on the pointwise norm. It becomes evident from Theorem \ref{th: upper bound d ge 3} and the constant $c$ being independent of $x$ that we can extend the upper bounds from pointwise estimation to estimation on $L^2(A)$, resulting in 
\begin{equation*}
\mathbb{E}[\left \| \hat{\pi}_{h,T} - \pi \right \|^2_A] \le
\begin{cases}
c(\frac{\log T}{ T})^ {\frac{2\bar{\beta}_3}{2\bar{\beta}_3 + d - 2}}\qquad \mbox{if } \beta_2 < \beta_3 \\
c(\frac{1}{ T})^ {\frac{2\bar{\beta}_3}{2\bar{\beta}_3 + d - 2}}  \qquad \mbox{if } \beta_2 = \beta_3.
\end{cases}
\end{equation*}
However, when transitioning from the data-driven procedure in $L^2(A)$ to pointwise estimation, a logarithmic factor is typically lost, as observed in prior works such as \cite{BroLow} for isotropic adaptive studies or \cite{Klu} for anisotropic ones. This phenomenon also holds true in our context, as demonstrated in Theorems \ref{th: adaptive} and \ref{th: adaptive point} below.

Furthermore, the analysis becomes more challenging when considering pointwise estimation. Indeed in this context, to claim that the rate-optimal choice for the bandwidth belongs to the set of candidate bandwidth, an additional condition involving $\beta$ emerges.

One of the main objectives of this work is to investigate the optimality of the logarithmic term in the convergence rate. To achieve this, we needed a procedure capable of automatically detecting the presence or absence of the logarithmic term. Consequently, the introduction of an additional logarithm due to the choice of a pointwise norm becomes a significant concern in this context. This is why we have opted to include the integrated $L^2$ risk in this section.} \\
\\

In order to introduce the quantities heuristically discussed at the beginning of this section we start by defining some auxiliary indexes: 
$$k_1:= arg \min_{l=1, ... , d} h_l, \qquad k_2:= arg\min_{l \neq k_1} h_l$$
and, in an iterative way, we can introduce for any $j$, 
$$k_j := arg\min_{l \neq k_1, ... , k_{j - 1}} h_l.$$
Then, we define the set of potential bandwidths {\blue for the $L^2(A)$ estimation} $\mathcal{H}_T$ as 
{\revar 
\begin{multline}
	\label{eq: def mathcal H}
\mathcal{H}_T \subset  \{ h \in (0,1]^d : \, \forall l=1, ..., d \quad    
(\frac{1}{T})^b \le h_l \le (\frac{1}{\log T})^{\frac{1}{d - 2} + a}; \, 
\\
\min((\sum_{j=1}^d |\log h_j| \, h_{k_1} h_{k_2})^{\frac{1}{2}}, (h_{k_1})^{\frac{1}{2}} (h_{k_2} h_{k_3})^\frac{1}{4}) \ge \frac{c (\log T)^{2 + a}}{\sqrt{T}} \},  
\end{multline}
where $a>1$, $b>0$. {\blue For the pointwise estimation, instead, we introduce the analogous set $\mathcal{H}_T^p$ which is such that
\begin{multline}
	\label{eq: def mathcal Hp}
\mathcal{H}_T^p \subset  \{ h \in (0,1]^d : \, \forall l=1, ..., d \quad    
(\frac{1}{T})^b \le h_l \le (\frac{1}{\log T})^{\frac{1}{d - 2} + a}; \, 
\\
(\prod_{l = 1}^d h_l)^{\frac{1}{2}} \min((\sum_{j=1}^d |\log h_j| \, h_{k_1} h_{k_2})^{\frac{1}{2}}, (h_{k_1})^{\frac{1}{2}} (h_{k_2} h_{k_3})^\frac{1}{4}) \ge \frac{c (\log T)^{2 + a}}{\sqrt{T}} \},  
\end{multline}
where again $a > 1$ and $b > 0$. It is easy to check that the right hand side of \eqref{eq: def mathcal Hp} is included in the right hand side of \eqref{eq: def mathcal H}.
}

Let us stress that it is possible to choose $a$ arbitrarily close to $1$, and $b$ arbitrarily large.
Moreover, we assume that there exists $c >0$ for which {\blue $|\mathcal{H}_T | + |\mathcal{H}_T^p | \le c T^c$, i.e. the growth of the sets $\mathcal{H}_T$ and $\mathcal{H}_T^p$ is at most polynomial.} \\
}
According to the set of candidate bandwidths, we can introduce the set of candidate estimators:
$$\mathcal{F}(\mathcal{H}_T) := \left \{ \hat{\pi}_{h, T} (x) = \frac{1}{T} \int_0^T \mathbb{K}_h (X_u - x) du: \quad x \in \mathbb{R}^d, \quad h \in \mathcal{H}_T \right \},$$
{\blue 
\begin{equation}
    \label{eq: def F H p}
    \mathcal{F}(\mathcal{H}_T^p) := \left \{ \hat{\pi}_{h, T} (x) = \frac{1}{T} \int_0^T \mathbb{K}_h (X_u - x) du: \quad x \in \mathbb{R}^d, \quad h \in \mathcal{H}_T^p \right \}.
\end{equation}
The goal of this section is to select two estimators from the families $\mathcal{F}(\mathcal{H}_T)$ and $\mathcal{F}(\mathcal{H}_T^p)$, respectively, in a completely data-driven way, based only on the observation of the continuous trajectory of the process X.} \\
Following the idea in \cite{GL}, our selection procedure relies on the introduction of auxiliary convolution estimators. According to our records, it was introduced in \cite{Lep99} as a device to circumvent the lack of ordering among a set of estimators in anisotropic case, where the increase of the variance of an estimator does not imply a decrease of its bias. \\
For any bandwidths $h = (h_1, ... , h_d)^T$, $\eta = (\eta_1, ... , \eta_d)^T$ $\in \mathcal{H}_T{\blue \cup \mathcal{H}_T^p}$ and $x \in \mathbb{R}^d$, we define
$$\mathbb{K}_h * \mathbb{K}_\eta (x) := \prod_{j = 1}^d (K_{h_j} * K_{\eta_j}) (x_j) = \prod_{j = 1}^d \int_{\mathbb{R}} K_{h_j} (u - x_j) K_{\eta_j } (u) du.$$
We moreover define the kernel estimators 
$$\hat{\pi}_{{\revar (h, \eta),T}} (x) := \frac{1}{T} \int_0^T (\mathbb{K}_h * \mathbb{K}_{\eta}) (X_u - x) du, \quad x \in \mathbb{R}^d.$$
As the convolution is commutative, we clearly have $\hat{\pi}_{{\revar (h, \eta),T}} = \hat{\pi}_{ {\revar (\eta,h),T}}$. Then, the selection procedure we propose is based on a comparison of the differences $\hat{\pi}_{{\revar (h, \eta),T}} - \hat{\pi}_{\eta, T}$.
We remark that the bound on the variance in Proposition \ref{prop: bound variance} when one does not know the order the regularities $\beta$ consists in
\begin{align*}
& \frac{k}{T} \, \min \Big(\frac{\sum_{j = 1}^d |\log h_j|}{\prod_{l \neq k_1, k_2} h_l}, \, \frac{1}{\sqrt{h_{k_2} h_{k_3}}\prod_{l \neq k_1, k_2, k_3} h_l}, \, \min_{j \ge 3} \frac{1}{(\prod_{l=1}^j  h_{k_l})^{1 - \frac{2}{j}}} \frac{1}{\prod_{l \neq k_1, ..., k_j} h_l} \Big) \\
& = \frac{k}{T} \, \min \Big(\frac{\sum_{j = 1}^d |\log h_j|}{\prod_{l \neq k_1, k_2} h_l}, \, \frac{1}{\sqrt{h_{k_2} h_{k_3}}\prod_{l \neq k_1, k_2, k_3} h_l} \Big),
\end{align*}
the equivalence is a consequence of the definition of $h_{k_j}$ for $j=1, ... , d$. With this purpose in mind we introduce the penalty {\blue functions
\begin{align}{\label{eq: penalty}}
V(h) & := \frac{k}{T} \, \min \Big(\frac{\sum_{j = 1}^d |\log h_j|}{\prod_{l \neq k_1, k_2} h_l}, \, \frac{1}{\sqrt{h_{k_2} h_{k_3}}\prod_{l \neq k_1, k_2, k_3} h_l} \Big) =: \frac{k}{T} \tilde{H}^2(h)
\end{align}
and 
\begin{align}{\label{eq: penalty point}}
V^p(h) & := (\log T) V(h) =\frac{k_T}{T} \tilde{H}^2(h),
\end{align}
where $k_T:= k \log T$ and $k$ is a numerical constant which has to be taken large.} Even though it is not explicit, it can be calibrated by simulations as done for example in Section 5 of \cite{Main adapt} through the implementation of a method inspired by Lacour, Massart and Rivoirard in \cite{Lacour et al}. {\blue We present in Section \ref{S:Numerical} some numerical simulations which shows the impact of the choice of $k$ on the quality of estimation. Upon comparing the two penalty functions defined above, it becomes evident that opting for pointwise estimation instead of $L^2$ estimation results in the loss of a logarithmic term.}\\
Then, we compare the differences $\hat{\pi}_{h, \eta} - \hat{\pi}_{\eta, T}$ in $A(h)$ {\blue and $A^p(h,x)$}, defined as below:
\begin{equation}
A(h) := \sup_{\eta \in \mathcal{H}_T} (\left \| \hat{\pi}_{{\revar (h, \eta),T}} - \hat{\pi}_{\eta, T} \right \|^2_A - V(\eta))_+,
\label{eq: def A(h)}
\end{equation}
{\blue \begin{equation}
A^p(h,x) := \sup_{\eta \in \mathcal{H}_T^p} (| \hat{\pi}_{{\revar (h, \eta),T}}(x) - \hat{\pi}_{\eta, T}(x) |^2 - V^p(\eta))_+.
\label{eq: def A(h) point}
\end{equation}}
Heuristically, $A(h)$ {\blue and $A^p(h,x)$ are} estimates of the squared bias and $V(h)$ {\blue and $V^p(h)$} of the variance bound. It is worth underlining that the penalty term which is used here comes from the three bounds obtained in Proposition \ref{prop: bound variance}, remarking that in this case the smoothness is unknown. \\
The selection is done by setting 
\begin{equation}
\tilde{h} := \mbox{arg}\min_{ h \in \mathcal{H}_T} (A (h) + V(h)).
\label{eq: def h tilde}
\end{equation}
{\blue for the $L^2(A)$ estimation and 
\begin{equation}
\tilde{h}^p(x) := \mbox{arg}\min_{ h \in \mathcal{H}^p_T} (A^p (h,x) + V^p(h)).
\label{eq: def h tilde point}
\end{equation}
for the pointwise one.}
Before proceeding with the main results of this subsection, let us introduce some notation. In the sequel, it will be useful to consider $\pi_h := \mathbb{K}_h * \pi$, which is the function that is estimated without bias by $\hat{\pi}_{h, T}$. It is indeed $\mathbb{E}[\hat{\pi}_{h, T}(x)] = \pi_h(x)$. Moreover, we define $\pi_{{\revar (h, \eta)}} := \mathbb{K}_h * \mathbb{K}_\eta * \pi $. To conclude the notation paragraph, we introduce the bias {\blue for the pointwise estimation $B^p(h)(x)$ as $|\pi_h(x) - \pi(x)|$, while for the procedure in the $L^2(A)$ we will employ the bias}
$B(h) := \left \| \pi_h - \pi \right \|^2_{\tilde{A}}$, where we have denoted as $\left \| . \right \|_{\tilde{A}}$ the $L^2$ - norm on $\tilde{A}$, a compact set in $\mathbb{R}^d$ which contains $A$. It is $\tilde{A} := \left \{  \zeta \in \mathbb{R}^d : d(\zeta, A) \le 2 \sqrt{d} \right \}$, where $d(\zeta, A) := \min_{x \in A} |\zeta -x| $. \\
The following result holds {\blue for the adaptive procedure in integrated $L^2$ risk.}
\begin{theorem}
Suppose that assumptions A1 - A2 hold and that $d \ge 3$. Then, we have that there exists $T_0 > 0$ such that, for any $T \ge T_0$, 
$$\mathbb{E}[\left \| \hat{\pi}_{\tilde{h}, T} -  \pi \right \|^2_A] \le c_1 \inf_{h \in \mathcal{H}_T} (B(h) + V(h)) + c_1 e^{ - c_2 (\log T)^{c_3}},$$
for {\revar $c_1$, $c_2$ positive constants and $c_3>1$.}
\label{th: adaptive}
\end{theorem}
As the last term above is negligible, the bound stated above shows that the estimator leads to an automatic trade-off between the bias $\left \| \pi_h - \pi \right \|^2_{\tilde{A}}$ and the variance $V(h)$.
The proof of Theorem \ref{th: adaptive} is postponed to Section \ref{S: proof adaptive}. \\
\\
{\blue An analogous result holds true when one considers the pointwise estimation, as stated in Theorem \ref{th: adaptive point} below.
\begin{theorem}
Suppose that assumptions A1 - A2 hold and that $d \ge 3$. Then, we have that there exists $T_0 > 0$ such that, for any $T \ge T_0$, 
$$\mathbb{E}[| \hat{\pi}_{\tilde{h}^p(x), T}(x) -  \pi(x)|^2] \le c_1 \inf_{h \in \mathcal{H}_T^p} (B^p(h,x) + V^p(h)) + c_1 T^{- c_2},$$
for a positive constant $c_1$ and $c_2\ge 1$.
\label{th: adaptive point}
\end{theorem}
\noindent The proof of Theorem \ref{th: adaptive point} is postponed to Section \ref{S: proof adaptive}.} \\
\\

\noindent 
We want now to replace the rate optimal choice for $h(T)$. With this purpose in mind we introduce the following sets of candidate bandwidths:
\begin{equation}
\mathcal{H}_T := \left \{ h = (h_1, ... , h_d)^T \in (0,1]^d : \,h_l = \frac{1}{z_l} \, { \mbox{ with } z_l \in \{1, ...,  \lfloor T \rfloor \}  \,  \mbox{ satisfying conditions in \eqref{eq: def mathcal H}}} \right \},
\label{eq: example HT}    
\end{equation}
{\blue \begin{equation}
\mathcal{H}_T^p := \left \{ h = (h_1, ... , h_d)^T \in (0,1]^d : \,h_l = \frac{1}{z_l} \, { \mbox{ with } z_l \in \{1, ...,  \lfloor T \rfloor \}  \,  \mbox{ satisfying conditions in \eqref{eq: def mathcal Hp}}} \right \}.
\label{eq: example HT point}    
\end{equation}}
{\revar As the ordering of the $\beta$'s is unknown, we adapt the formal definition of $\bar{\beta}_3$ which now is the harmonic mean of the $\beta$'s after having removed the smallest two:
}
$$ \frac{1}{{\revar{\bar{\beta}_3}}} := \frac{1}{d-2} \sum_{l \neq k_1, k_2} \frac{1}{\beta_l}. $$
\\
Then, we obtain the following theorems:
\begin{theorem}
Suppose that assumptions A1 - A2 hold with $\beta_{k_1} \ge 1$ and let $\mathcal{H}_T$ be defined by \eqref{eq: example HT}. Then, we have
\begin{equation*}
\mathbb{E}[\left \| \hat{\pi}_{\tilde{h},T} - \pi \right \|^2_A] \underset{\sim}{<}
\begin{cases}
(\frac{\log T}{ T})^ {\frac{2{\revar\bar{\beta}_3}}{2{\revar\bar{\beta}_3} + d - 2}} + e^{ - c_1 (\log T)^{c_2}} \qquad \mbox{if } \beta_{k_2} < \beta_{k_3} \\
(\frac{1}{ T})^ {\frac{2{\revar\bar{\beta}_3}}{2{\revar\bar{\beta}_3} + d - 2}} + e^{ - c_1 (\log T)^{c_2}} \qquad \mbox{if } \beta_{k_2} = \beta_{k_3},
\end{cases}
\end{equation*}
{\revar for $c_1$ a positive constant and $c_2>1$.
}\label{th: adaptive optimal}
\end{theorem}
Underlining once again that the exponential term above is negligible, we have that the risk estimates we get using the bandwidth provided by our selection procedure converges to zero at the same rate as in previous section, which we have proven being optimal in a minimax sense. \\
It is important to highlight, then, that our procedure automatically selects in a data-driven way, having or not the logarithm in the convergence rate. {\blue Changes occur with the pointwise adaptive procedure. In fact, as emphasized in Theorem \ref{th: adaptive optimal point} below, this results in the addition of a logarithm in both cases.}
{\blue \begin{theorem}
Suppose that assumptions A1 - A2 hold and let $\mathcal{H}_T^p$ be defined by \eqref{eq: example HT point}. Assume moreover that the smallest smoothness $\beta_{k_1}$ is such that $\beta_{k_1} > 2$. Then, we have
\begin{equation*}
\mathbb{E}[| \hat{\pi}_{\tilde{h}^p(x),T}(x) - \pi(x) |^2] \underset{\sim}{<}
\begin{cases}
(\frac{(\log T)^2}{ T})^ {\frac{2{\bar{\beta}_3}}{2{\bar{\beta}_3} + d - 2}} + T^{- c_1} \qquad \mbox{if } \beta_{k_2} < \beta_{k_3} \\
(\frac{\log T}{ T})^ {\frac{2{\bar{\beta}_3}}{2{\bar{\beta}_3} + d - 2}} + T^{- c_1} \qquad \mbox{if } \beta_{k_2} = \beta_{k_3},
\end{cases}
\end{equation*}
for $c_1\ge 1$.
\label{th: adaptive optimal point}
\end{theorem}}

}

\section{Lower bounds}{\label{S: Lower bounds}}
One may wonder if the convergence rates found by using kernel density estimators are optimal or if it is possible to improve them by using other density estimators. We aim at showing that the convergence rates found in Theorems \ref{th: upper bound d ge 3} and \ref{th: upper bound d=2} are optimal. We will focus first on the case $d \ge 3$. We will start computing a lower bound in a general case. After that, we will show it is possible to improve it, up to ask $\beta_2 < \beta_3$. \\
\\
We can write down the expression of the minimax risk for the estimation, at some point $x_0$, of an invariant density $\pi$ belonging to the anisotropic Holder class $\mathcal{H}_d (\beta, 2 \mathcal{L})$. Let $x_0 \in \mathbb{R}^d$ and $\Sigma (\beta, \mathcal{L},a_{\text{min}},b_0,a_0,a_1,b_1,\tilde{C},\tilde{\rho})$ as in Definition \ref{def: insieme sigma v2} here above. We define the minimax risk 
\begin{equation}
	\mathcal{R}_T (\beta, \mathcal{L},a_{\text{min}},b_0,a_0,a_1,b_1,\tilde{C},\tilde{\rho}) := \inf_{\tilde{\pi}_T} \sup_{(a,b) \in \Sigma (\beta, \mathcal{L},a_{\text{min}},b_0,a_0,a_1,b_1,\tilde{C},\tilde{\rho})} \mathbb{E}[(\tilde{\pi}_T (x_0) - \pi_{(a,b)} (x_0))^2],
	\label{eq: def minimax risk}
\end{equation}
where the infimum is taken on all possible estimators of the invariant density, {\modch based on $X_t$ for $t\in [0, T]$}. The following lower bound will be showed in Section \ref{s: proofs upper}.
\begin{theorem}
	Let $\beta=(\beta_1,\dots,\beta_d)$, $1<\beta_1\le \dots \le \beta_d$, $\beta_2 > 2$, $\mathcal{L}=(\mathcal{L}_1,\dots,\mathcal{L}_d)$, $\mathcal{L}_i>0$.
	{\revar 	Consider $0<a_\text{min}\le a_0$, $a_1>0$, $b_0>0$, $b_1>0$, then, there exist $\tilde{C}$, $\tilde{\rho}$, $c > 0$ and $T_0 > 0$ such that, for $T \ge T_0$,}
	\begin{equation*}
		\mathcal{R}_T (\beta, \mathcal{L},a_{\text{min}},b_0,a_0,a_1,b_1,\tilde{C},\tilde{\rho}) \ge c\,  (\frac{1}{T})^{{ \frac{ 2\bar{\beta}_3}{ 2 \bar{\beta}_3 + d - 2 }}}.
	\end{equation*}
	\label{th: lower bound}
\end{theorem}
Theorem \ref{th: lower bound} implies that, on a class of diffusions $X$ whose invariant density belongs to $\mathcal{H}_d (\beta, 2 \mathcal{L})$, it is not possible to find an estimator with a rate of estimation better than $T^{-{ \frac{ \bar{\beta}_3}{ 2 \bar{\beta}_3 + d - 2 }}}$ .
{\revar Remark that $\tilde{C}$ and $\tilde{\rho}$ can not be chosen freely in the statement of Theorem \ref{th: lower bound}. Indeed, the theorem only ensures that the lower bound holds true for sufficiently small $\tilde{C}$ and large $\tilde{\rho}$. This is in contrast with the upper bound of Theorem \ref{th: upper bound d ge 3} which holds true for all $\tilde{C}$, $\tilde{\rho}$. Actually, it is impossible to get a lower bound for any $\tilde{C}$, $\tilde{\rho}$. Indeed, for fixed values of the $\mathcal{L}_i$'s, $b_0$ and $b_1$, the set $\Sigma (\beta, \mathcal{L},a_{\text{min}},b_0,a_0,a_1,b_1,\tilde{C},\tilde{\rho})$ can be empty if $\tilde{C}$ is too large and $\tilde{\rho}$ too small. In such case, the lower bound can not hold true.} \\
\\
Comparing the result here above with the second point of Theorem \ref{th: upper bound d ge 3} we observe that the convergence rate we found in the lower bound and in the upper bound are the same, when $\beta_2 = \beta_3$. When $\beta_2 < \beta_3$, instead, it is possible to improve the lower bound previously obtained, as gathered in the following theorem. Its proof can be found in Section \ref{s: proofs upper}.
\begin{theorem}
	Let $\beta=(\beta_1,\dots,\beta_d)$, $0<\beta_1 \le \dots \le \beta_d$,  $\beta_2 >2$, and $\beta_2 < \beta_3$. Moreover, $\mathcal{L}=(\mathcal{L}_1,\dots,\mathcal{L}_d)$, $\mathcal{L}_i>0$.
	{\revar 	Consider $0<a_\text{min}\le a_0$, $a_1>0$, $b_0>0$, $b_1>0$, then, there exist $\tilde{C}$, $\tilde{\rho}$, $c > 0$ and $T_0 > 0$ such that, for $T \ge T_0$,}
	\begin{equation}\label{E: lower bound log}
		\mathcal{R}_T (\beta, \mathcal{L},a_{\text{min}},b_0,a_0,a_1,b_1,\tilde{C},\tilde{\rho}) \ge c\, (\frac{\log T}{T})^{{ \frac{ 2\bar{\beta}_3}{ 2 \bar{\beta}_3 + d - 2 }}}.
	\end{equation}
	\label{th: lower bound avec log}
\end{theorem}
We will see that the condition $\beta_2 < \beta_3$ is crucial in order to recover a logarithm in the lower bound. Comparing the lower bounds in Theorems \ref{th: lower bound} and \ref{th: lower bound avec log} with the upper bounds in Theorem \ref{th: upper bound d ge 3} it follows that the kernel density estimator we proposed in \eqref{eq: def estimator} achieves the best possible convergence rate. \\
\\
It is possible to ensure an analogous lower bound in the bi-dimensional case, which ensures the optimality of the convergence rate found in Theorem \ref{th: upper bound d=2}.
\begin{theorem}
	Let $d=2$, $\beta=(\beta_1, \beta_2)$, $0<\beta_1\le \beta_2$, $\beta_2 > 2$, $\mathcal{L}=(\mathcal{L}_1,\mathcal{L}_2)$, $\mathcal{L}_i>0$.
	{\revar 	Consider $0<a_\text{min}\le a_0$, $a_1>0$, $b_0>0$, $b_1>0$, then, there exist $\tilde{C}$, $\tilde{\rho}$, $c > 0$ and $T_0 > 0$ such that, for $T \ge T_0$,}
	\begin{equation*}
		\mathcal{R}_T (\beta, \mathcal{L},a_{\text{min}},b_0, a_0,a_1,b_1,\tilde{C},\tilde{\rho}) \ge c\,  (\frac{\log T}{T}).
	\end{equation*}
	\label{th: lower bound d=2}
\end{theorem}
We conclude that also in the two-dimensional case, the estimator proposed in \eqref{eq: def estimator} is minimax-optimal.

{\revarn 
	\section{Numerical simulations}\label{S:Numerical}
	In this section, we study the numerical performance of the Goldenshluger-Lepski procedure on simulated data. In particular, we discuss the effect of the choice of the constant $k$ appearing in \eqref{eq: penalty point} on the quality of estimation. For simplicity, the diffusion is chosen to follow a stationary Ornstein--Uhlenbeck model, in dimension $d=3$, solution to $dX_t=-\frac{X_t}{2}dt + dB_t$. In this situation, the density $\pi$ is the one of the
	Gaussian distribution $\mathcal{N}(0,\text{Id}_3)$.  
	We focus on the pointwise estimation of $\pi(x)$. We use the family of estimators introduced in
	\eqref{eq: def F H p}, where $\mathbb{K}_{h}(x-y)=\prod_{m=1}^3K(\frac{x_m-y_m}{h_m})\frac{1}{h_m}$ is a product of a one dimensional Gaussian kernel $K$, and $(h_1,h_2,h_3) \in\mathcal{H}_T^p:=\{2\times10^{-0.2k},\quad k\in\{0,\dots11\}\}^3$. 
	The computation of the quantity $A^p(h,x)$ defined in \eqref{eq: def A(h) point} appears difficult in practice. Indeed, the formula involves the computation of the estimators $\hat{\pi}_{(h,\eta),T}$ for $(h,\eta)\in (\mathcal{H}_T^p)^2$ which yields, in our situation, to the computation of $\text{card}(\mathcal{H}_T^p)^2\simeq 10^6$ different estimators.  Thus, we slightly modify the definition \eqref{eq: def A(h) point} and use instead
	\begin{equation}\label{eq: A via sup}
		\widetilde{A}^p(h,x) = \sup_{\eta \in \mathcal{H}_T^p} \{
		 |  \hat{\pi}_{h \vee \eta,T}(x) - \hat{\pi}_{\eta,T}(x)|^2 - V^p(\eta)  
		\}	
	\end{equation}
	 where $h \vee \eta=(h_1 \vee \eta_1,h_2 \vee \eta_2,h_3 \vee \eta_3)$. Since  $h \vee \eta \in \mathcal{H}_T^p$ for $(h,\eta) \in (\mathcal{H}_T^p)^2$, this modification reduces the computational task to the computation of the family of estimators $\hat{\pi}_{h,T}$ for $h \in \mathcal{H}_T^p$. Remark,  that since we are using a Gaussian kernel, we have $\mathbb{K}_{h}*\mathbb{K}_{h}=\mathbb{K}_{h+\eta}$, for which $\mathbb{K}_{h \vee \eta}$ can be used as proxy. This justifies the approximation of	${A}^p(h,x)$  by $	\widetilde{A}^p(h,x) $. Also, alternative definitions of the Goldenshluger-Lespki method directly relies on quantities analogous to \eqref{eq: A via sup} (e.g. see \cite{Lepski15}).	 
	 
	 The total time of observation is $T=10^5$, the process $X$ is simulated using an Euler scheme with step $10^{-5}$, while the continuous time estimators \eqref{eq: def estimator} are approximated by discrete versions using the same grid as the Euler scheme. We use $m=200$ replications and obtain a set of estimators $(\hat{\pi}_{h,T}^{(i)})_{h \in \mathcal{H}^p_T}$ for each iteration $i\in \{1,\dots,m\}$. From these $m$ independent realisations of the estimator, it is possible to compute the empirical variance and mean 
	 of $\hat{\pi}_{h,T}$ for each $h \in \mathcal{H}_T^p$. Using the true value of the density $\pi(x)$, we can determine the bandwidth $h\in \mathcal{H}^p_T$ minimizing the empirical $L^2$ risk of $\hat{\pi}_{h,T}$ over the $m$ Monte-Carlo trials. Let us denote by $h_\text{oracle}$ the value of the
	 bandwidth minimizing this empirical risk and by $\hat{\pi}_{\text{oracle},T}=\hat{\pi}_ {h_\text{oracle},T}$ the associated estimator. As the determination of this optimal $h_\text{oracle}$ uses the knowledge of $\pi(x)$, it is unfeasible in practice. However, this oracle estimator will serve as a basis of comparison for the Goldenshluger-Lepski estimator. In Table \ref{T:Oracle}, we show the empirical mean and standard deviation of the oracle estimator, together with the value $h_\text{oracle}$ written as $h_\text{oracle}=(2\times10^{-0.2k_1},2\times10^{-0.2k_2},2\times10^{-0.2k_3})$. The estimator works remarkably well. It is noticeable that when estimating $\pi(x)$ at the point $x=(0,0,0)$, around which the function $\pi$ is isotropic, the shape of the optimal kernel is isotropic as well, since $k_1=k_2=k_3$. On the other hand, for the estimation of $\pi(x)$ at $x=(0,0,1)$ the optimal $h_\text{oracle}$ is anisotropic.	 
	 
	 In the definition of the Goldenshluger-Lepski estimator, we need to specify the value of the constant $k$ appearing in \eqref{eq: penalty point}. Theoretically, it is sufficient to choose this constant above some threshold, which is related to the variance of the estimator. As a consequence, the exact value of this threshold is explicitly unknown. Since we are conducting a Monte Carlo experiment, it is possible to estimate the variance of the estimators $\hat{\pi}_{h,T}$ and thus to deduce useful values for $k$. Indeed, we choose some arbitrary $h^* = (2\times10^{-0.2k_1},2\times10^{-0.2k_2},2\times10^{-0.2k_3})$ with $k_1=k_2=k_3=4$ and compute the empirical variance of the estimator $\hat{\pi}_{h^*,T}$. Then, we define $k^*$ as the constant such that the penality term $V(h^*)$ defined in \eqref{eq: penalty} matches the empirical variance. This value $k^*$ gives a hint on the magnitude of the constant $k$ that could be chosen.	 
	  To determine precisely the dependence of the Goldenshluger-Lepski estimator $\hat{\pi}_{\hat{h}^p(x),T}$ on the constant $k$, we now let $k$ range in $(0.01\times k^*, 1.5\times k^*)$ and compute the empirical $L^2$ risk of the corresponding estimators. It appears that the  Goldenshluger-Lepski estimator works correctly on a large range of choices of $k$ in this interval. In Figures \ref{F: ratio}, we plot the ratio between the empirical $L^2$ risk of the Goldenshluger-Lepski estimator and the $L^2$ risk of the oracle estimator for different values of $k \in (0.01\times k^*, 1.5\times k^*)$. As expected, the oracle estimator exhibits a smaller risk than the  Goldenshluger-Lepski's estimator
	  for any values of $k$. However, we see that for $k \in   (0.1\times k^*, 0.75\times k^*)$, the $L^2$ risk of the  Goldenshluger-Lepski estimator remains less than twice that of the oracle estimator. Considering the strong performance of the oracle estimator, this indicates that the  Goldenshluger-Lepski method provides a good estimation of the parameter for a wide range of constants $k$.
	 
	  The procedure described here above to determine the value of $k^*$ is impossible in practice when only a single path of the process is observed, as the Monte Carlo estimation of the estimator's variance is unfeasible. However, a rough estimation of the variance of the estimator 
	  could be achieved, for instance, by splitting one observed path of $X$ on $[0,T]$ into $N$ disjoint pieces, $(X_{kT/N+s})_{0\le s \le T/N}$ for $k=0,\dots,N-1$, such that $T/N$ remains large. We compute on each piece of the trajectory an estimator $\hat{\pi}_{h^*,T/N}(x)$. Taking the empirical variance of these $N$ estimators gives us a proxy for the variance of $\hat{\pi}_{h^*,T/N}(x)$. We can derive a data-driven value of $k^*$ by matching this empirical variance with the penality term $V(h^*)$ defined in \eqref{eq: penalty}. Finally, we use the Goldenshluger-Lepski estimator with some constant $k$ chosen in $ (0.1\times k^*, 0.75\times k^*)$ as suggested by the Figures \ref{F: ratio}. We have tested this method, with $N=100$ and $k=k^*/4$, on $p=200$ replications of the trajectory. Results can be found in Table \ref{T:GL_Full_Adapt}, and show that this estimation procedure provides a sharp estimation of $\pi(x)$.

	 
	 
	 \begin{table}[h]
	 	\caption{\label{T:Oracle} Oracle estimator} 	\centering
	 	\begin{tabular}{|c|c|c|}
	 		\hline
	 		& $x=(0,0,0)$  & $x=(0,0,1)$  \\
	 		\hline
	 	True value $\pi(x)$	& $6.349 \times 10^{-2}$  & $3.851 \times 10^{-2}$  \\
	 	\hline
	 	Mean $\hat{\pi}_{\text{oracle},T}$	& $6.287 \times 10^{-2}$  &  $3.829 \times 10^{-2}$\\
	 	\hline
	 	Std	$\hat{\pi}_{\text{oracle},T}$		& $1.10 \times 10^{-3}$ & $6.93 \times 10^{-4}$  \\
		\hline
	 	$(k_1,k_2,k_3)$ 		& $(7,7,7)$ & $(7,8,4)$  \\
	 		\hline
	 	\end{tabular}
	 \end{table}

\begin{figure}
	\centering
	\caption{\label{F: ratio}
		Ratio of the $L^2$ risk of $\hat{\pi}_{h^*,T}$ over the $L^2$ risk of $\hat{\pi}_{\text{oracle},T}$ as a function of $k/k^*$.}
	\begin{subfigure}[b]{0.45\textwidth}
		 \includegraphics[scale=0.45]{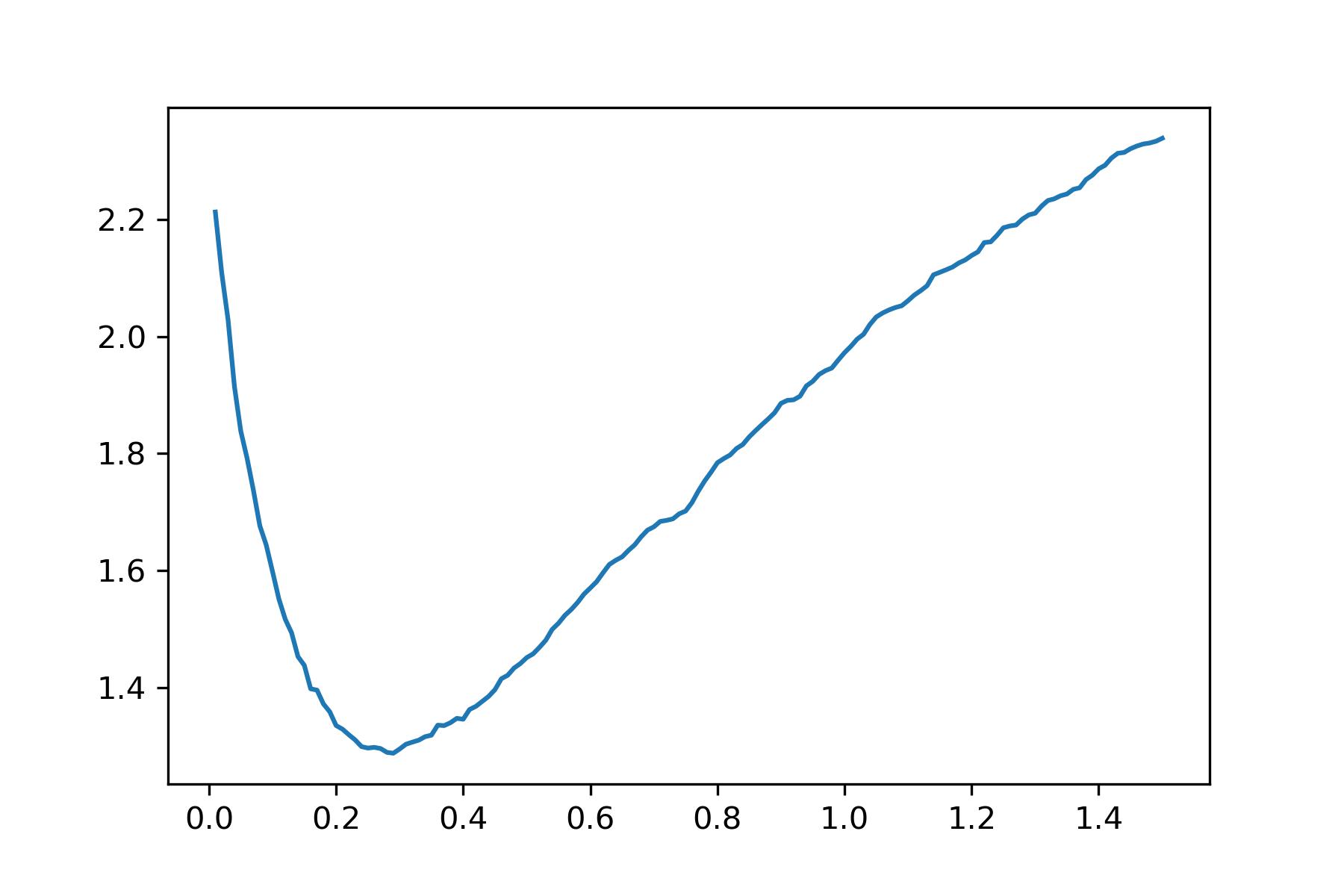}
	\caption{$x=(0,0,0)$}
	\end{subfigure}
	\centering
\begin{subfigure}[b]{0.45\textwidth}
	\includegraphics[scale=0.45]{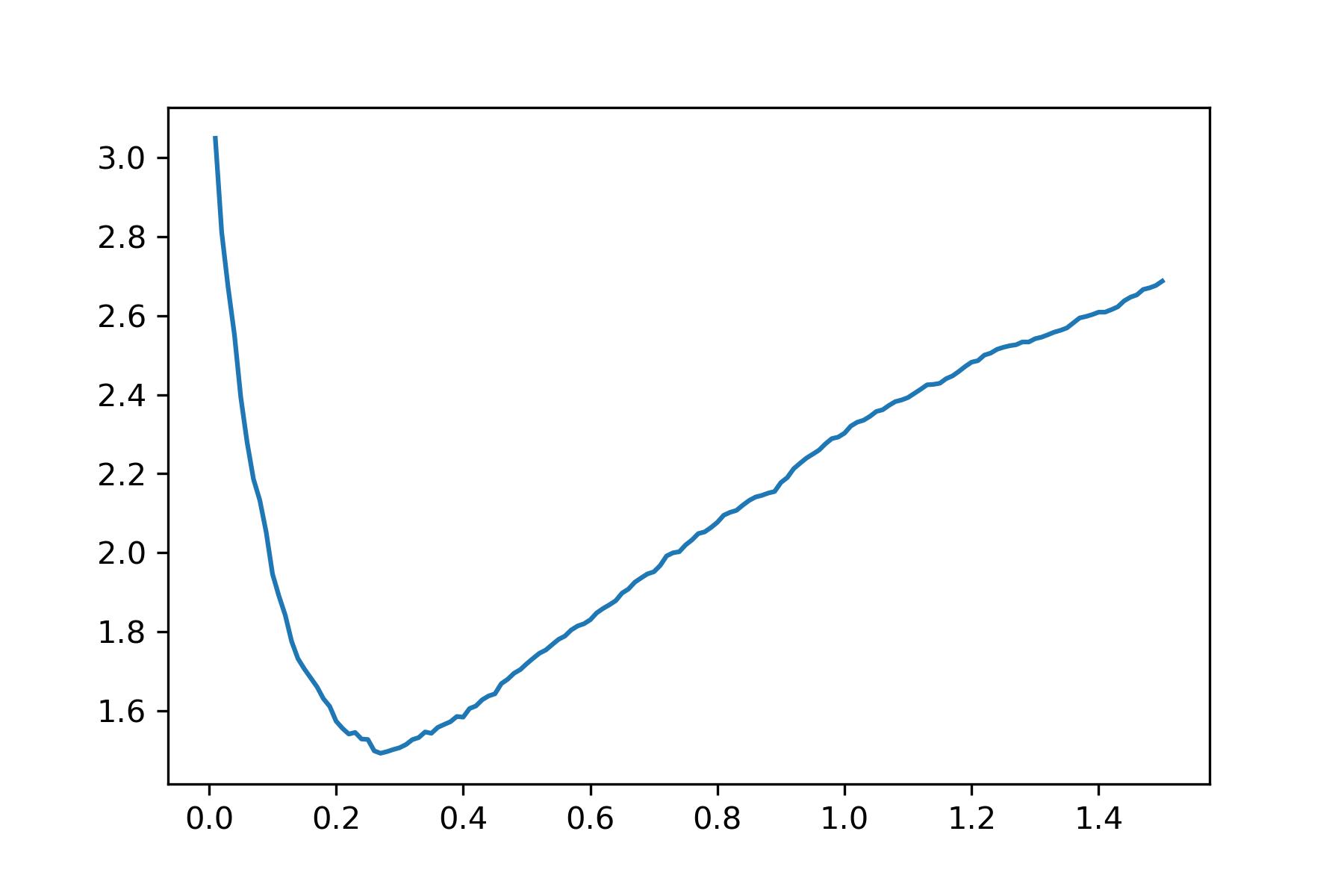}
	\caption{$x=(0,0,1)$}
\end{subfigure}

\end{figure}

\begin{table}[h]
	\caption{\label{T:GL_Full_Adapt} Goldenshluger-Lepski estimator with a data-driven choice of $k$} 	\centering
	\begin{tabular}{|c|c|c|}
		\hline
		& $x=(0,0,0)$  & $x=(0,0,1)$  \\
		\hline
		True value $\pi(x)$	& $6.349 \times 10^{-2}$  & $3.851 \times 10^{-2}$  \\
	\hline
	Mean $\hat{\pi}_{\hat{h}(x),T}$	& $ 6.229 \times 10^{-2}$  &  $3.776 \times 10^{-2}$\\
	\hline
	Std	$\hat{\pi}_{\hat{h}(x),T}$		& $1.25 \times 10^{-3}$ & $8.36 \times 10^{-4}$  \\
	\hline
	\end{tabular}
\end{table}
}

\section{Proofs upper bounds stated in Section \ref{S: Estimator}}{\label{s: proofs upper}}
This section is devoted to the proof of all the results stated in Section \ref{S: Estimator}, about the behaviour of the estimator proposed in \eqref{eq: def estimator}, assuming that a continuous record of the process $X$ is available. As Theorem \ref{th: upper bound d ge 3} is a consequence of Proposition \ref{prop: bound variance}, we start {\rev by} proving Proposition \ref{prop: bound variance}.
\subsection{Proof of Proposition \ref{prop: bound variance}}
\begin{proof}
In the sequel, the constant $c$ may change from line to line and it is independent of $T$. \\
From the definition \eqref{eq: def estimator} and the stationarity of the process we get
$$Var(\hat{\pi}_{h,T}(x)) = \frac{1}{T^2} \int_0^T \int_0^T k(t-s) dt \, ds,$$
where 
$$k(u) := Cov(\mathbb{K}_h(x - X_0), \mathbb{K}_h(x- X_u)).$$
We deduce that
$$Var(\hat{\pi}_{h,T}(x)) \le \frac{1}{T} \int_0^T | k(s)| ds. $$
We want to use different way to upper bound the variance, relying on the bound on the transition density gathered in Lemma \ref{lemma: bound transition density} and on the mixing properties of the process as in Lemma \ref{lemma: ergodicity}. Hence,
we split the time interval $[0,T]$ into 4 pieces: 
$$[0,T]= [0, \delta_1) \cup [\delta_1, \delta_2) \cup[\delta_2, D) \cup[D,T],$$
where $\delta_1$, $\delta_2$ and $D$ will be chosen later, to obtain an upper bound on the variance as sharp as possible. We will see that the bound on the variance will depend on $k_0$ only on the interval $[\delta_1, \delta_2)$. For this reason we start considering what happens on $[0, \delta_1)$ without taking into account the fact that $k_0$ can be larger or smaller than $3$. \\
$\bullet$ For $s \in [0, \delta_1)$, from Cauchy -Schwarz inequality and the stationarity of the process we get
$$| k(s)| \le Var(\mathbb{K}_h(x -X_0))^\frac{1}{2}Var(\mathbb{K}_h(x -X_s))^\frac{1}{2} = Var(\mathbb{K}_h(x -X_0)) \le \int_{\mathbb{R}^d} (\mathbb{K}_h (x -y))^2 \pi(y) dy.$$
As $\pi \in \mathcal{H}_d(\beta, \mathcal{L}) $, its infinitive norm is bounded. Using also the definition of $\mathbb{K}_h$ given in \eqref{eq: def estimator} it follows
\begin{equation}
| k(s)| \le \frac{c}{\prod_{l =1}^d h_l}
\label{eq: first bound k}
\end{equation}
which implies
\begin{equation}
\int_0^{\delta_1} | k(s)| {\rev ds} \le  \frac{c\, \delta_1}{\prod_{l =1}^d h_l}.
\label{eq: stima su 0 delta1}
\end{equation}
$\bullet$ For $s \in [\delta_1, \delta_2)$, taking $\delta_2 < 1$, we use the definition of transition density, for which 
$$| k(s)| \le \int_{\mathbb{R}^d} |\mathbb{K}_h (x -y)| \int_{\mathbb{R}^d} |\mathbb{K}_h (x -y')| p_{s}(y, y') dy'  \pi (y) dy.$$
We now act differently according on the value of $k_0$. If $k_0 =1$ and $\beta_2 < \beta_3$ or $k_0= 2$, then we introduce $q_s(y'_3 ... y'_d| y'_1, y'_2, y)$ as below:
$$q_s(y'_3 ... y'_d| y'_1, y'_2, y) := e^{- \lambda_0 \frac{|y_1 - y'_1|^2}{s}} \times e^{- \lambda_0 \frac{|y_2 - y'_2|^2}{s}} \times \frac{1}{\sqrt{s}} e^{- \lambda_0 \frac{|y_3 - y'_3|^2}{s}} \times ... \times \frac{1}{\sqrt{s}} e^{- \lambda_0 \frac{|y_d - y'_d|^2}{s}}{\modchi e^{c (|y|^2 + 1) s}}. $$
{\modchi We remark that $s < \delta_2 < 1$.}
From Lemma \ref{lemma: bound transition density} we know it is
$$p_{s}(y, y') \le \frac{c_0}{s} q_s(y'_3 ... y'_d| y'_1, y'_2, y).$$
Let us stress that
\begin{equation}
 \sup_{y'_1, y'_2 \in \mathbb{R}^{d}} \int_{\mathbb{R}^{d - 2}} q_s(y'_3 ... y'_d| y'_1, y'_2, y) dy'_3 ... dy'_d \le {\modchi c e^{c (|y|^2 + 1) s}}. 
\label{eq: qs bounded}
\end{equation}
Then,
\begin{equation}
|k(s)| \le \frac{c_0}{s} \int_{\mathbb{R}^d} |\mathbb{K}_h (x -y)| \int_{\mathbb{R}^d} |\mathbb{K}_h (x -y')| q_s(y'_3 ... y'_d| y'_1, y'_2, y) dy' \,  \pi (y) dy.
\label{eq: k1 start}
\end{equation}
Using the definition of $\mathbb{K}_h$ and \eqref{eq: qs bounded} we obtain
\begin{align*}
  & \int_{\mathbb{R}^d} |\mathbb{K}_h (x -y')| q_s(y'_3 ... y'_d| y'_1, y'_2, y) dy' \\
  & \le \frac{c}{\prod_{j \ge 3} h_j} \int_{\mathbb{R}} \frac{1}{h_1} K(\frac{y'_1 - x_1}{h_1}) \int_{\mathbb{R}} \frac{1}{h_2} K(\frac{y'_2 - x_2}{h_2}) (\int_{\mathbb{R}^{d - 2}} q_s(y'_3 ... y'_d| y'_1, y'_2, y) dy'_3 ... dy'_d )  dy'_2 \, dy'_1 \\
  & \le \frac{c}{\prod_{j \ge 3} h_j} \int_{\mathbb{R}} \frac{1}{h_1} K(\frac{y'_1 - x_1}{h_1}) \int_{\mathbb{R}} \frac{1}{h_2} K(\frac{y'_2 - x_2}{h_2}) {\modchi e^{c (|y|^2 + 1) s}} dy'_2 \, dy'_1 \\
  & \le \frac{c}{\prod_{j \ge 3} h_j}{\modchi e^{c (|y|^2 + 1) s}}.
\end{align*}
{\modchi Replacing this result in \eqref{eq: k1 start} we obtain 
$$|k(s)| \le \frac{c_0}{s} \int_{\mathbb{R}^d} |\mathbb{K}_h (x -y)|  e^{c (|y|^2 + 1) s} \pi (y) dy. $$
As the support of the kernel function is a compact set, the integral is bounded.}

It implies that, when $k_0 =1$ and $\beta_2 < \beta_3$ or $k_0 =2$, we get 
\begin{equation}
|k(s)| \le \frac{c}{\prod_{j \ge 3} h_j} \frac{1}{s}
\label{eq: bound k 1}
\end{equation}
and so
\begin{equation}
\int_{\delta_1}^{\delta_2} |k(s)| ds \le \int_{\delta_1}^{\delta_2} \frac{c}{\prod_{j \ge 3} h_j} \frac{1}{s} ds = c \frac{\log (\delta_2) - \log (\delta_1)}{\prod_{j \ge 3} h_j},
\label{eq: fine delta1 delta2}
\end{equation}
where the constant $c$ does not depend on the coefficient $(a, b ) \in \Sigma$. 
We now consider what happens on $[\delta_1, \delta_2)$ when $k_0 \ge 3$ or $k_0 =1$ and $\beta_2 = \beta_3$. In analogy to what done before we introduce $q_s(y'_{k_0 + 1} ... y'_d| y'_1, ... , y'_{k_0}, y)$ which is such that
$$q_s(y'_{k_0 + 1} ... y'_d| y'_1, ... , y'_{k_0}, y) := e^{- \lambda_0 \frac{|y_1 - y'_1|^2}{s}} \times ... \times  e^{- \lambda_0 \frac{|y_{k_0} - y'_{k_0}|^2}{s}} \times \frac{1}{\sqrt{s}} e^{- \lambda_0 \frac{|y_{k_0 + 1} - y'_{k_0 + 1}|^2}{s}} \times ... \times \frac{1}{\sqrt{s}} e^{- \lambda_0 \frac{|y_d - y'_d|^2}{s}}{\modchi e^{c (|y|^2 + 1) s}}. $$
Using again Lemma \ref{lemma: bound transition density} we can write 
$$p_{s}(y, y') \le c_0 s^{- \frac{k_0}{2}} q_s(y'_{k_0 + 1} ... y'_d| y'_1, ... , y'_{k_0}, y)$$
and as before we have
$$ \sup_{y'_1, ... , y'_{k_0} \in \mathbb{R}^{ {k_0}}} \int_{\mathbb{R}^{d - {k_0}}} q_s(y'_{k_0 + 1} ... y'_d| y'_1, ... , y'_{k_0}, y) dy'_{k_0 + 1} ... dy'_d \le c {\modchi e^{c (|y|^2 + 1) s}}. $$
Acting as in \eqref{eq: k1 start} and below it follows
\begin{align}{\label{eq: bound k 2}}
|k(s)| & \le c_0 s^{- \frac{k_0}{2}} \int_{\mathbb{R}^d} |\mathbb{K}_h (x -y)| \int_{\mathbb{R}^d} |\mathbb{K}_h (x -y')| q_s(y'_{k_0 + 1} ... y'_d| y'_1, ... , y'_{k_0}, y) dy' \,  \pi (y) dy \\
& \le c s^{- \frac{k_0}{2}} \frac{1}{\prod_{l \ge k_0 + 1} h_l}. \nonumber
\end{align}
Hence, when $k_0 \ge 3$, we have
\begin{equation}
\int_{\delta_1}^{\delta_2} |k(s)| ds \le \frac{c}{\prod_{l \ge k_0 + 1} h_l} \int_{\delta_1}^{\delta_2} s^{- \frac{k_0}{2}} ds \le  \frac{c}{\prod_{l \ge k_0 + 1} h_l} \delta_1^{1 - \frac{k_0}{2}}.
\label{eq: fine delta1 delta2 k0 large}
\end{equation}
We remark that, for $k_0 = 1$, the reasoning here above still applies. However, as 
$1 - \frac{k_0}{2} = \frac{1}{2}$ is now positive, we obtain
\begin{equation}
\int_{\delta_1}^{\delta_2} |k(s)| ds \le \frac{c}{\prod_{l \ge 2} h_l} \delta_2^{\frac{1}{2}}.
\label{eq: 19.5 in scan}
\end{equation}
As the bound no longer depends on $\delta_1$, we can choose to take $\delta_1 = 0$. \\
$\bullet$ For $s \in [\delta_2, D)$ we use the same estimation for any value of $k_0$, but for $k_0 = 1$, $\beta_2 = \beta_3$. We still consider the bound on the transition density gathered in Lemma \ref{lemma: bound transition density}. Such a bound is not uniform in $t$ big. However, for $t \ge 1$, it is {\modchi
\begin{align*}
p_t(y, y') & = \int_{\R^d} p_{\frac{1}{2}} (y,z) p_{t - \frac{1}{2}} (z, y') dz \\
& \le c \int_{\R^d} e^{- 2 \lambda_0 |y - z|^2}  e^{c (|y|^2 + 1)} p_{t - \frac{1}{2}} (z, y')  dz \\
& \le c  e^{c (|y|^2 + 1)} \int_{\R^d} p_{t - \frac{1}{2}} (z, y') dz \le c  e^{c (|y|^2 + 1)}.
\end{align*}}
We deduce, for all $t$, 
\begin{equation}
p_t(y, y') \le c_0 t^{- \frac{d}{2}} e^{- \lambda_0 \frac{|y-y'|^2}{t}} + {\modchi c e^{c (|y|^2 + 1)} }.
\label{eq: bound transition, t grande}
\end{equation}
It follows
\begin{align}{\label{eq: bound k delta1 delta2}}
| k(s)| & \le c \int_{\mathbb{R}^d} |\mathbb{K}_h (x -y)| \int_{\mathbb{R}^d} |\mathbb{K}_h (x -y')| (s^{- \frac{d}{2}} + 1){\modchi e^{c (|y|^2 + 1)} } dy'  \pi (y) dy \\
& \le c (s^{- \frac{d}{2}} + 1), \nonumber
\end{align}
{\modchi where we have also used that the support of the kernel function is a compact set}. 
We therefore get
\begin{equation}
\int_{\delta_2}^D |k(s)| ds \le c (\delta_2^{1 - \frac{d}{2}} + D).
\label{eq: fine delta2 D}
\end{equation}
When $k_0 = 1$ and $\beta_2 = \beta_3$, instead, we proceed analogously to how we did in the previous interval in order to obtain \eqref{eq: bound k 2}. We choose to remove, in particular, the contribution of the first three bandwidths. We obtain, $\forall s \in [\delta_2, D)$, 
\begin{equation}
|k(s)| \le c (s^{- \frac{3}{2}} \frac{1}{\prod_{l \ge 4} h_l} + 1),
\label{eq: bound delta2 D k0 =1 18.5}
\end{equation}
where the final constant comes from \eqref{eq: bound transition, t grande}, as now $s$ can be larger than $1$. It follows
\begin{equation}
\int_{\delta_2}^D |k(s)| ds \le c (\frac{1}{\prod_{l \ge 4} h_l} \frac{1}{\delta_2^{\frac{1}{2}}} + D).
\label{eq: 23.5 scan}
\end{equation}
$\bullet$ For $s \in [D,T]$ we exploit the mixing properties of the process, as stated in Proposition \ref{P:mixing uniform}. The following control on the covariance holds true:
$$|k(s)| \le c \left \| \mathbb{K}_h (x - \cdot ) \right \|_\infty^2 e^{- \rho s} \le c (\frac{1}{\prod_{j = 1}^d h_j})^2 e^{- \rho s}, $$
for $\rho$ and $c$ positive constants which are uniform over the set of coefficients $(a,b) \in \Sigma$. It entails 
\begin{equation}
\int_{D}^T |k(s)| ds \le c (\frac{1}{\prod_{j = 1}^d h_j})^2 e^{- \rho D}. 
\label{eq: fine D T}
\end{equation}
We now put all the pieces together. For $k_0 = 1$ and $\beta_2 < \beta_3$ or $k_0 =2$ we collect together \eqref{eq: stima su 0 delta1}, \eqref{eq: fine delta1 delta2}, \eqref{eq: fine delta2 D} and \eqref{eq: fine D T}. We deduce
\begin{equation}
Var(\hat{\pi}_{h,T}(x)) \le \frac{c}{T}( \frac{ \delta_1}{\prod_{l =1}^d h_l} + \frac{1}{\prod_{j \ge 3} h_j} (|\log(\delta_1)| + |\log (\delta_2)|)+ \delta_2^{1 - \frac{d}{2}} +  D + (\frac{1}{\prod_{j = 1}^d h_j})^2 e^{- \rho D}).
\label{eq: variance con parametri}
\end{equation}
We now want to choose $\delta_1$, $\delta_2$ and $D$ for which the estimation here above is as sharp as possible. To do that, we take 
$\delta_1 := h_1 h_2$, $\delta_2 := (\prod_{j \ge 3} h_j)^{\frac{2}{d-2}}$ and $D:= [\max (- \frac{2}{\rho} \log (\prod_{j = 1}^d h_j), 1) \land T]$. Replacing them in \eqref{eq: variance con parametri} we obtain
\begin{align*}
Var(\hat{\pi}_{h,T}(x)) & \le \frac{c}{T}(\frac{1}{\prod_{j \ge 3} h_j} + \frac{\sum_{j = 1}^d |\log(h_j)|}{\prod_{j \ge 3} h_j}  + \frac{1}{\prod_{j \ge 3} h_j} + \sum_{j = 1}^d |\log(h_j)| + 1 ) \\
& \le \frac{c}{T} \frac{\sum_{j = 1}^d |\log(h_j)|}{\prod_{j \ge 3} h_j}.
\end{align*}
If otherwise $k_0 \ge 3$, we consider \eqref{eq: fine delta1 delta2 k0 large} instead of \eqref{eq: fine delta1 delta2} which, together with \eqref{eq: stima su 0 delta1}, \eqref{eq: fine delta2 D} and \eqref{eq: fine D T} provides 
\begin{equation}
 Var(\hat{\pi}_{h,T}(x)) \le \frac{c}{T}( \frac{ \delta_1}{\prod_{l =1}^d h_l} + \frac{\delta_1^{1 - \frac{k_0}{2}}}{\prod_{j \ge k_0 + 1} h_j}+ \delta_2^{1 - \frac{d}{2}} +  D + (\frac{1}{\prod_{j = 1}^d h_j})^2 e^{- \rho D}).
\label{eq: variance con parametri k0 large}
\end{equation}
We observe that the balance between the first two terms is achieved for $\delta_1 := (\prod_{l = 1}^{k_0} h_l)^{\frac{2}{k_0}}$. Then, we can choose $\delta_2 =1$. We moreover take 
$D:= [\max (- \frac{2}{\rho} \log (\prod_{j = 1}^d h_j), 1) \land T]$, as before. It yields
$$Var(\hat{\pi}_{h,T}(x)) \le \frac{c}{T} \frac{(\prod_{l = 1}^{k_0} h_l)^{\frac{2}{k_0}}}{\prod_{l \ge 1} h_l} = \frac{1}{T} \frac{1}{(\prod_{l = 1}^{k_0} h_l)^{1 - \frac{2}{k_0}}(\prod_{l \ge k_0 +1} h_l)},$$
as we wanted. \\
To conclude, by \eqref{eq: 19.5 in scan} with $\delta_1 = 0$, \eqref{eq: 23.5 scan} and \eqref{eq: fine D T} we have proven that when $k_0 = 1$ and $\beta_2 = \beta_3$ we have 
\begin{align*}
Var(\hat{\pi}_{h,T}(x)) & \le \frac{c}{T}( \frac{ \delta_2^{\frac{1}{2}}}{\prod_{l \ge 2} h_l} + \frac{1}{\delta_2^{\frac{1}{2}} \, \prod_{l \ge 4} h_l}+  D + (\frac{1}{\prod_{j = 1}^d h_j})^2 e^{- \rho D}) \\
& \le \frac{c}{T} (\frac{\sqrt{h_2 h_3}}{\prod_{l \ge 2} h_l} + \frac{1}{\sqrt{h_2 h_3}\prod_{l \ge 4} h_l}) = \frac{c}{T}\frac{1}{\sqrt{h_2 h_3}\prod_{l \ge 4} h_l} , 
\end{align*}
where the last estimation follows by having chosen $\delta_2 := h_2 h_3$. The parameter $D$ is moreover chosen as above. Remarking that all the constants are uniform over the set of coefficient $(a,b) \in \Sigma$, the result is proven.
\end{proof}

\subsection{Proof of Theorem \ref{th: upper bound d ge 3}}
\begin{proof}
From the usual bias-variance decomposition it is
$$\mathbb{E}[|\hat{\pi}_{h,T}(x) - \pi (x)|^2]\, {\rev =} |\mathbb{E}[\hat{\pi}_{h,T}(x)] - \pi (x)|^2 + Var(\hat{\pi}_{h,T}(x)).$$
An upper bound for the variance is gathered in Proposition \ref{prop: bound variance}. Regarding the bias, a standard computation (see for example the proof of Proposition 1.2 of \cite{Ts} or Proposition 1 of \cite{Decomp}) provides 
\begin{equation}
|\mathbb{E}[\hat{\pi}_{h,T}(x)] - \pi (x)|^2 \le c  \sum_{j = 1}^d h_j^{\beta_j},
\label{eq: stima bias fine}
\end{equation}
with the constant $c$ that does not depend on $x$, nor on $(a,b) \in \Sigma$. \\
The result is then obtained by looking for the balance between the two terms. \\
We start considering the case where $k_0 =1$ and $\beta_2 < \beta_3$ or $k_0 = 2$. Thanks to Proposition \ref{prop: bound variance} we know the bound \eqref{eq: estim variance with log} holds true.
After simple computations (as in the proof of Theorem 1 of \cite{Lower bound}) it is easy to see that the balance is achieved by choosing the rate optimal bandwidth $h^*:= ((\frac{\log T}{T})^{a_1}, ..., (\frac{\log T}{T})^{a_d})$, with 
\begin{equation}
a_1 \ge \frac{ \barfix{\bar{\beta}_3}}{\beta_1 (2  \barfix{\bar{\beta}_3} + d -2 ) },  \qquad a_2 \ge \frac{\barfix{\bar{\beta}_3}}{\beta_2 (2  \barfix{\bar{\beta}_3} + d -2 ) }
\label{eq: scelta a1 a2}
\end{equation}
and
\begin{equation}
a_l =\frac{ \barfix{\bar{\beta}_3}}{\beta_l (2  \barfix{\bar{\beta}_3} + d -2 ) } \qquad \forall l \in \left \{ 3, ... , d \right \}.
\label{eq: scelta al}
\end{equation}
Replacing the value of $h^*$ in the upper bounds of the variance and of the bias we get 
$$\mathbb{E}[|\hat{\pi}_{h,T}(x) - \pi (x)|^2] \le c (\frac{\log T}{T})^{\frac{ 2 \barfix{\bar{\beta}_3}}{2  \barfix{\bar{\beta}_3} + d -2  }},$$
with the same constant $c$ for all $(a,b) \in \Sigma$. \\
When $k_0 = 1$ and $\beta_2 = \beta_3$ the bias-variance decomposition consists in 
$$\mathbb{E}[|\hat{\pi}_{h,T}(x) - \pi (x)|^2] \le c \sum_{j = 1}^d h_j^{\beta_j} + \frac{c}{T}  \frac{1}{\sqrt{h_2 h_3}\prod_{l \ge 4} h_l}. $$
As $\beta_2 = \beta_3$ we choose $h_2 = h_3$, such that the upper bound on the variance gathered in Proposition \ref{prop: bound variance} becomes simply $\frac{c}{T} \frac{1}{\prod_{l \ge 3} h_l}$.
Then, similar computations as above implies that the balance is achieved by choosing $h^*(T) := ((\frac{1}{T})^{a_1}, ... , (\frac{1}{T})^{a_d})$, with $a_1$ as in \eqref{eq: scelta a1 a2} and $a_2 = a_3$ with $a_j$ as in \eqref{eq: scelta al}, for $j \ge 3$. It yields
$$ \sup_{(a,b) \in \Sigma} \mathbb{E}[|\hat{\pi}_{h,T}(x) - \pi (x)|^2] \le c (\frac{1}{T})^{\frac{ 2 \barfix{\bar{\beta}_3}}{2  \barfix{\bar{\beta}_3} + d -2  }}.$$
Assuming now that $k_0 \ge 3$, we have 
$$\mathbb{E}[|\hat{\pi}_{h,T}(x) - \pi (x)|^2] \le c  \sum_{j = 1}^d h_j^{\beta_j} + \frac{c}{T} \frac{1}{h_1^{k_0 - 2}(\prod_{l \ge k_0 + 1}h_l)},$$
where we have already chosen to take $h_1 = h_2 = ... = h_{k_0}$. We now look for the rate optimal bandwidth by setting $h^*_l := (\frac{1}{T})^{a_l}$ and searching $a_1$, $a_{k_0 + 1}$, ... , $a_d$ such that the bias and variance are balanced. It consists in solving the following system:
\begin{equation}
\begin{cases}
a_1 \beta_1 = a_l \beta_l \qquad \forall l \in \{ k_0 + 1, ... , d \}\\
1 - a_1(k_0 -2) - \sum_{l \ge k_0 + 1} a_l = 2 a_1 \beta_1.
\end{cases}    
\end{equation}
As a consequence of the first $d-k_0$ equations, we can write 
\begin{equation}
a_l =\frac{\beta_1}{\beta_l} a_1, \qquad \forall l \in \left \{ k_0 + 1, ... , d \right \}.
\label{eq: al tramite a_d}
\end{equation}
Hence, the last equation becomes
\begin{align}{\label{eq: eq for choice a}}
2 \beta_1 a_1 &= 1 - a_1(k_0 -2) - \beta_1 a_1 \sum_{l \ge k_0 + 1} \frac{1}{\beta_l} \\
& = 1 - a_1(k_0 -2) - \beta_1 a_1 \frac{d-k_0}{ \barfix{\bar{\beta}_k}}, \nonumber
\end{align}
where $\bar{\beta}_k$ is the mean smoothness over $\beta_{k_0 + 1}$, ... , $\beta_d$ and it is such that $\frac{1}{ \barfix{\bar{\beta}_k}} = \frac{1}{d-k_0} \sum_{l \ge k_0 + 1} \frac{1}{\beta_l}$. We also observe that, as $\frac{1}{\barfix{\bar{\beta}_3}} = \frac{1}{d-2} \sum_{l \ge 3} \frac{1}{\beta_l}$, it is 
\begin{equation}
\frac{k_0 - 2}{\beta_1} + \frac{d-k_0}{\barfix{\bar{\beta}_k}} = \frac{d-2}{\barfix{\bar{\beta}_3}}.
\label{eq: equivalence beta}
\end{equation}
Hence, \eqref{eq: eq for choice a} can be seen as 
$$2 \beta_1 a_1 = 1 - a_1 \beta_1 \frac{d-2}{\barfix{\bar{\beta}_3}}, $$
which leads to the choice 
$$a_1 =\frac{ \barfix{\bar{\beta}_3}}{\beta_1 (2  \barfix{\bar{\beta}_3} + d -2 ) }.$$
Thanks to the first $d-k_0$ equations in the system it follows
$$a_l =\frac{ \barfix{\bar{\beta}_3}}{\beta_l (2  \barfix{\bar{\beta}_3} + d -2 ) } \qquad \forall l \in \left \{ k_0 + 1, ... , d \right \}.$$
Plugging the rate optimal bandwidth $h^*$ in the bound of the mean squared error and recalling that the constant $c$ does not depend on $(a,b) \in \Sigma$, we obtain 
$$\sup_{(a,b) \in \Sigma} \mathbb{E}[|\hat{\pi}_{h,T}(x) - \pi (x)|^2] \le c (\frac{1}{T})^{\frac{ 2 \barfix{\bar{\beta}_3}}{2  \barfix{\bar{\beta}_3} + d -2  }},$$
as we wanted. 
\end{proof}

\begin{remark}
As the situation is complicated, it is worth highlighting how the optimal bandwidths depend on the smoothness.
\begin{itemize}
    \item[$\bullet$] When $\beta_2 < \beta_3$, $h_1$ and $h_2$ are arbitrarily small while $h_3$, ... , $h_d$ are fixed. In this case there is the logarithm in the rate.
    \item[$\bullet$] When $\beta_2 = \beta_3$ we have two possibilities for the bandwidths, depending on $\beta_1$. In both cases we remove the logarithm in the rate. 
    \begin{itemize}
        \item If $\beta_1 < \beta_2$, then $h_1$ can be arbitrarily small while $h_2$, ... , $h_d$ are fixed.
        \item If $\beta_1 = \beta_2$, then all the bandwidths $h_1$, ... , $h_d$ are fixed.
    \end{itemize}
\end{itemize}
It is interesting to see that when we have two degrees of freedom on the bandwidths it is impossible to remove the logarithm in the convergence rate, while having one degree of freedom on the bandwidths is enough to modify the convergence rate: it is the same it would have been without any degree of freedom on the bandwidths.

\end{remark}

\subsection{Proof of Theorem \ref{th: upper bound d=2}}
\begin{proof}
The proof relies, as before, on the bias-variance decomposition. Concerning the bias, \eqref{eq: stima bias fine} still holds true. We need to provide an upper bound on the variance, based on Lemma \ref{lemma: bound transition density} and Proposition \ref{P:mixing uniform}, as before. The main difference compared to the proof of Proposition \ref{prop: bound variance} is that we now split the integral over $[0, T]$ in three pieces: 
$$[0, T] = [0, \delta) \cup [\delta, D) \cup [D, T),$$
where $\delta$ and $D$ will be chosen later. On the first and on the last piece we act as in the proof of Proposition \ref{prop: bound variance} and so \eqref{eq: stima su 0 delta1} and \eqref{eq: fine D T} keep holding. Regarding the interval $[\delta, D)$, we act on it as we did on $[\delta_2, D)$. From \eqref{eq: bound k delta1 delta2}, recalling that now $d=2$, we obtain 
$$\int_{\delta}^D |k(s)| ds \le c \int_{\delta}^D (s^{-1} + 1) ds \le c(|\log D| + |\log \delta| + D). $$
Putting all the pieces together we get
\begin{align*}
Var(\hat{\pi}_{h,T}(x)) & \le \frac{c}{T}( \frac{ \delta}{h_1 \, h_2} + |\log D| + |\log \delta| + D + (\frac{1}{h_1\, h_2})^2 e^{- \rho D}) \\
& \le \frac{c}{T}(1 + \sum_{j = 1}^2 |\log h_j| + \sum_{j = 1}^2 |\log (| \log h_j|)| ) \\
& \le \frac{c \sum_{j = 1}^2 |\log h_j|}{T},
\end{align*}
where we have opportunely chosen $\delta= h_1 h_2$ and $D:= [\max (- \frac{2}{\rho} \log (h_1 \, h_2), 1) \land T]$. It follows
$$\mathbb{E}[|\hat{\pi}_{h,T}(x) - \pi (x)|^2] \le c h_1^{2 \beta_1} + c h_2^{2 \beta_2} + \frac{c \sum_{j = 1}^2 |\log h_j|}{T}.$$
To conclude it is enough to observe that the optimal choice for the bandwidth consists in taking $h_l^* = (\frac{\log T}{T})^{a_l}$, with $a_l \ge \frac{1}{2 \beta_l}$ for $l=1, 2$, that provides the wanted convergence rate.
\end{proof}
{\rev
\section{Proof of the adaptive procedure stated in Section \ref{S: Estimator}}\label{S: proof adaptive}

\subsection{Proof of Theorem \ref{th: adaptive}}
The adaptive procedure in Theorem \ref{th: adaptive} heavily relies on a bound on the expectation of $A(h)$, as stated in following proposition.
\begin{proposition}
Suppose that assumptions A1 - A2 hold and that $d \ge 3$. Then, $\forall h \in \mathcal{H}_T$, we have the following bound
$$\mathbb{E}[A (h)] \le  c_1 B(h) + c_1 e^{ - c_2 (\log T)^2}.$$
\label{prop: A(h)}
\end{proposition}
The proof of the fact that Proposition \ref{prop: A(h)} implies Theorem \ref{th: adaptive} is classical  and it is therefore here omitted (see for example the proof of Theorem 1 in \cite{Chapitre 4}, relying on Proposition 5 therein). \\
The proof of Proposition \ref{prop: A(h)}, similarly as Proposition 5 in \cite{Chapitre 4}, relies on the use of Berbee's coupling method as in Viennet \cite{Viennet} and on a version of Talagrand inequality given in Klein and Rio \cite{KR}.
However, a different penalty function is here chosen, which results in some challenges when one wants to use Talagrand inequality. 

\subsubsection{Proof of Proposition \ref{prop: A(h)}}
\begin{proof}
The proof of Proposition \ref{prop: A(h)} is close in spirit to the proof of Proposition 5 of \cite{Chapitre 4}. We want to highlight the main differences appearing in our context. \\

\noindent As in the proof of Proposition 5 of \cite{Chapitre 4}, the key point consists in Talagrand inequality, which is stated on independent random variables. Therefore, we start by introducing some blocks which are mutually independent through Berbee's coupling method as done in Viennet \cite{Viennet}, Proposition 5.1 (p. 484). \\
Hence, we assume that $T = 2 p_T q_T$, with $p_T$ integer and $q_T$ a real to be chosen, and we split the initial process $X = (X_t)_{t \in [0,T]}$ in $2 p_T$ processes of length $q_T$. For each $j \in \left \{ 1, ... , p_T \right \} $ we introduce 
$X^{j,1} := (X_t)_{t \in [2(j - 1)q_T, (2j - 1)q_T]}$ and $X^{j,2} := (X_t)_{t \in [(2j-1) q_T, 2 j q_T]}$ as in Proposition 5 of \cite{Chapitre 4} and we construct the process $(X^*_t)_{t \in [0,T]}$ having the same distribution of $X^{j,1}$ and $X^{j,2}$ on the blocks but such that, for each $k \in \left \{ 1,2 \right \}$, $X^{* \,1,k}, ... , X^{* \,p_T,k}$ are independent (see points 1,2,3 in Proposition 5 of \cite{Chapitre 4} for details). \\
All the quantities computed using $X_t$ can be replaced by the same quantities computed through $X^*_t$, we add in this case an $*$ in the notation. In particular, we introduce in this way $\hat{\pi}^*_{\revar{h,T}}$ and $\hat{\pi}^*_{\revar{h,T}} = \frac{1}{2}(\hat{\pi}^{* (1)}_{\revar{h,T}} + \hat{\pi}^{*(2)}_{\revar{h,T}})$, to separate the part coming from $X^{* \,.,1}$ (super -index $(1)$) and those coming from $X^{* \,.,2}$ (super -index $(2)$), having $\hat{\pi}^{* (1)}_{\revar{h,T}} := \frac{1}{ p_T q_T} \sum_{j = 1}^{p_T} \int_{2(j - 1) q_T}^{(2 j - 1) q_T}\mathbb{K}_h (X^*_u - x) du $. \\
In a natural way we define moreover $\hat{\pi}^{*}_{{\revar (h, \eta),T}} := \mathbb{K}_\eta * \hat{\pi}^*_{\revar{h,T}}$, that can be written again as $\frac{1}{2}(\hat{\pi}^{* (1)}_{{\revar (h, \eta),T}} + \hat{\pi}^{*(2)}_{{\revar(h, \eta),T}})$, to separate the contribution of $X^{* \,.,1}$ and $X^{* \,.,2}$. \\

\noindent We recall that this notation is analogous to the one introduced in Proposition 5 of \cite{Chapitre 4}, where triangular inequality is used to obtain
\begin{align}{\label{eq: berbee}}
  A(h) &\le \sup_{\eta \in \mathcal{H}_T} [\left \| \hat{\pi}_{{\revar (h, \eta),T}} - \hat{\pi}_{{\revar (h, \eta),T}}^* \right \|_A^2  + (\left \| \hat{\pi}_{{\revar (h, \eta),T}}^* - \pi_{{\revar (h, \eta)}} \right \|_A^2 - \frac{V(\eta)}{2})_+ + \left \| \pi_{{\revar (h, \eta)}} - \pi_\eta \right \|_A^2 \nonumber \\
  & + (\left \| \pi_\eta - \hat{\pi}_{{\revar \eta,T}}^* \right \|_A^2 - \frac{V(\eta)}{2})_+ + \left \| \hat{\pi}_{{\revar \eta,T}}^* - \hat{\pi}_{{\revar \eta,T}} \right \|_A^2] \\  
  & = \sup_{\eta \in \mathcal{H}_T}[ \sum_{j = 1}^5 I_j^{h,\eta}]. \nonumber 
\end{align}
Then, the analysis of $I_1$, $I_3$ and $I_5$ is the same as in Proposition 5 of \cite{Chapitre 4} and provides, respectively
\begin{equation}
\sup_{\eta \in \mathcal{H}_T} I_3^{h, \eta} \le c B(h)
\label{eq: I3}
\end{equation}
and 
\begin{align*}
\mathbb{E}[ |\sup_{\eta \in \mathcal{H}_T} I_1^{h, \eta} |]+ \mathbb{E}[ |\sup_{\eta \in \mathcal{H}_T} I_5^{h, \eta} |] & \le \frac{c}{(\prod_{l=1}^d \eta_l)^2} \frac{T}{(\log T)^{2}} e^{- \gamma (\log T)^2},
\end{align*}
where we have also chosen $q_T$ as $(\log T)^2$. Then, the definition of the set of candidate bandwidths $\mathcal{H}_T$ as in \eqref{eq: def mathcal H} implies
\begin{equation}{\label{eq: I1 I5}}
\mathbb{E}[ |\sup_{\eta \in \mathcal{H}_T} I_1^{h, \eta} |]+ \mathbb{E}[ |\sup_{\eta \in \mathcal{H}_T} I_5^{h, \eta} |] \le c \frac{T^{1 + 2db}}{(\log T)^{2}} e^{- \gamma (\log T)^2}.
\end{equation}
On $I_2^{h, \eta}$ and $I_4^{h, \eta}$ we want to use Talagrand inequality as formulated in Lemma 2 in \cite{Main adapt}. It is a 
 consequence of the Talagarand inequality given in Klein and Rio \cite{KR} (see also Lemma 5 in \cite{Chapitre 4}):
\begin{lemma}
Let $T_1, ... , T_{\revar p}$ be independent random variables with values in some Polish space $\mathcal{X}$, $\mathcal{R}$ a countable class of measurable functions from $\mathcal{X}$ {\revar into $\mathbb{R}$} 
	and 
$v_p (r) := \frac{1}{p} \sum_{j = 1}^p [r (T_j) - \mathbb{E}[r(T_j)]]. $
Then,
\begin{equation}
\mathbb{E}[(\sup_{r \in \mathcal{R}} |v_p(r)|^2 - 2H^2)_+] \le c (\frac{v}{p} e^{- c \frac{p H^2}{v}} + \frac{M^2}{p^2} e^{- c \frac{p H}{M}}),
\label{eq: Talagrand in Klein Rio}
\end{equation}
with $c$ a universal constant and where
$$\sup_{r \in \mathcal{R}} \left \| r \right \|_\infty \le M, \quad \mathbb{E}[\sup_{r \in \mathcal{R}}|v_p (r)|] \le H, \quad \sup_{r \in \mathcal{R}} \frac{1}{p} \sum_{j = 1}^p Var(r (T_j)) \le v.$$
\label{lemma: Talagrand in Klein Rio}
\end{lemma}
\noindent
{\revar 
In order to use Lemma \ref{lemma: Talagrand in Klein Rio}, we write that
 $\left \| \pi_\eta  - \hat{\pi}^{* (1)}_{\eta,T} \right \|_A^2 = \sup_{r \in \widetilde{\mathcal{B}}(1)}  <   \pi_\eta-\hat{\pi}^{* (1)}_{\eta,T}  ,r>^2  $ where $\widetilde{\mathcal{B}}(1)$ is a countable dense set of the unit ball of $L^2(A)$. For $r \in \tilde{\mathcal{B}}(1)$, we set
 \begin{equation*}
 r_\eta(X^{*,j,1}):=\frac{1}{q_T}\int_{2(j-1)q_T}^{(2j-1)q_T} \mathbb{K}_\eta * r (X_s^{*,j,1}) ds
 \end{equation*}
which is such that $ <  \hat{\pi}^{* (1)}_{\eta,T} ,r> =\frac{1}{p}\sum_{j=1}^p r_\eta(X^{*,j,1})$, and 
$ <   \pi_\eta-\hat{\pi}^{* (1)}_{\eta,T}  ,r>=-v_p(r_\eta)$
where we use the notation $v_{p}$ of Lemma \ref{lemma: Talagrand in Klein Rio}.
The application of Lemma \ref{lemma: Talagrand in Klein Rio}, with the countable class $\mathcal{R}=\{r_\eta, r \in \tilde{\mathcal{B}}(1)\}$, now requires to compute the constants $M$, $H$ and $v$.}
To do so, one can closely follow the proof of Lemma 4 in \cite{Chapitre 4}. Then, it is straightforward to see that 
$$M= c(\prod_{l = 1}^d \eta_l)^{-\frac{1}{2}}.$$
Regarding the computation of $H$, it comes directly from the computation of the variance. It is easy to check that, replacing the bound on the variance given by Proposition 1 in \cite{Chapitre 4} by the one gathered in our Proposition \ref{prop: bound variance}, one obtains
\begin{align*}
	H^2 &= \frac{c}{T} \, \min \Big(\frac{\sum_{j = 1}^d |\log \eta_j|}{\prod_{l \neq k_1, k_2} \eta_l}, \, \frac{1}{\sqrt{\eta_{k_2} \eta_{k_3}}\prod_{l \neq k_1, k_2, k_3} \eta_l} \Big) \\
	& = : \frac{c}{T} \tilde{H}^2 (\eta).
\end{align*}
To use Talagrand inequality, we moreover need to compute $v$. Again, one can check the computations in Lemma 4 in \cite{Chapitre 4} to easily obtain 
$$v := \frac{c}{q_T} (1 + \log (\frac{1}{|\prod_{l = 1}^d \eta_l |})).$$
Then, we use Lemma \ref{lemma: Talagrand in Klein Rio} and it follows that there exists some $k_0 > 0$ such that, for any $\hat{k} \ge k_0$ we have
\begin{align*}
	&\mathbb{E}[\sup_{\eta \in \mathcal{H}_T} (\left \| \pi_\eta  - \hat{\pi}^{* (1)}_{{\revar \eta,T}} \right \|_A^2 - \frac{\hat{k}}{T} \tilde{H}^2)_+]  
	{\revar{} \le \sum_{\eta \in \mathcal{H}_T} \mathbb{E}[ (\left \| \pi_\eta  - \hat{\pi}^{* (1)}_{\eta,T} \right \|_A^2 - \frac{\hat{k}}{T} \tilde{H}^2)_+] 
	} 
	\\
	& \le c \sum_{\eta \in \mathcal{H}_T} \frac{(1 + \log (\frac{1}{|\prod_{l = 1}^d \eta_l|}))}{p_T q_T} e^{- c \frac{\frac{p_T}{T} \tilde{H}^2 (\eta)}{\frac{1}{q_T} (1 + \log (\frac{1}{|\prod_{l = 1}^d \eta_l |})) }} + \frac{(\prod_{l = 1}^d \eta_l)^{-1}}{p_T^2 } e^{- c \frac{p_T \frac{1}{\sqrt{T}} \tilde{H} (\eta)}{(\prod_{l = 1}^d \eta_l)^{- \frac{1}{2}}}} .
\end{align*}
We recall that $2 p_T q_T = T$, where $q_T$ is chosen as $ (\log T)^2 $; we
can therefore upper bound the right hand side of the equation here above with
\begin{align}{\label{eq: Talagrand}}
	& \sum_{\eta \in \mathcal{H}_T} \frac{(1 + \log (\frac{1}{|\prod_{l = 1}^d \eta_l|}))}{T} e^{- \frac{c \tilde{H}^2 (\eta) }{(1 + \log (\frac{1}{|\prod_{l = 1}^d \eta_l|}))}} + \frac{(\log T)^4}{(\prod_{l = 1}^d \eta_l) T^2} e^{- c \frac{\sqrt{T}}{(\log T)^2} (\prod_{l = 1}^d \eta_l)^{\frac{1}{2}} \tilde{H} (\eta)}.
\end{align}
Regarding the first term we remark that, from the definition of the set $\mathcal{H}_T$ given in \eqref{eq: def mathcal H} we have 
\begin{align}{\label{eq: bound Htilde}}
	\tilde{H} (\eta) &\ge \min (\log T (\log T)^{(d-2)(\frac{1}{d - 2} + a)}, (\log T)^{(d-2)(\frac{1}{d - 2} + a)} ) \\
	& = (\log T)^{1 + a(d-2)}. \nonumber
\end{align}
Moreover, 
\begin{align*}
	& (\prod_{l = 1}^d \eta_l)^{\frac{1}{2}} \tilde{H} (\eta) \\
	& = \min((\sum_{j=1}^d |\log \eta_j| \, \eta_{k_1} \eta_{k_2})^{\frac{1}{2}}, (\eta_{k_1})^{\frac{1}{2}} (\eta_{k_2} \eta_{k_3})^\frac{1}{4}).
\end{align*}
Furthermore, it is $(\eta_{k_1} \eta_{k_2})^{\frac{1}{2}} \le (\eta_{k_1})^{\frac{1}{2}} (\eta_{k_2} \eta_{k_3})^\frac{1}{4}$ for $\eta_{k_2} < \eta_{k_3}$, which implies that the minimum is achieved by the first in this context. Otherwise, for $\eta_{k_2} = \eta_{k_3}$,  because of the presence of the logarithm in the first term, the minimum is realised by the second. \\
Then, by the definition we have given of $\mathcal{H}_T$, $\forall h \in \mathcal{H}_T$ we have $(\prod_{l = 1}^d \eta_l)^{\frac{1}{2}} \tilde{H} (\eta) \ge \frac{(\log T)^{2 + a}}{\sqrt{T}}$.
Thus, \eqref{eq: Talagrand} is upper bounded by
\begin{align*}
	& \sum_{\eta \in \mathcal{H}_T} \frac{\log T}{T} e^{- c (\log T)^{a(d-2)}} + T^{d b - 2}(\log T)^{4} e^{- c (\log T)^{a}} \\
	& \le T^{c - 1} \log T e^{- c (\log T)^{a(d-2)}} + T^{d b - 2 + c}(\log T)^{4} e^{- c (\log T)^{a}},
\end{align*}
where in the last inequality we have used that, by the definition \eqref{eq: def mathcal H} of $\mathcal{H}_T$, its cardinality $|\mathcal{H}_T|$ has polynomial growth in $T$ {\revar at most as $T^c$}. Therefore, we have proven
\begin{equation}{\label{eq: Talagrand per I4}}
	\mathbb{E}[\sup_{\eta \in \mathcal{H}_T} (\left \| \pi_\eta  - \hat{\pi}^{* (1)}_{{\revar \eta,T}} \right \|_A^2 - \frac{\hat{k}}{T} \tilde{H}^2)_+]  \le  T^{c - 1} \log T e^{- c (\log T)^{a(d-2)}} + T^{d b - 2 + c}(\log T)^{4} e^{- c (\log T)^{a}};  
\end{equation}
from which it directly follows the wanted bound on $I_4^{h, \eta}$
{\revar since $a>1$ and $d\ge 3$}. Indeed, it is $\pi_\eta  - \hat{\pi}^{*}_\eta = \frac{1}{2}(2 \pi_{\eta} - \hat{\pi}^{* (1)}_{{\revar\eta,T}} - \hat{\pi}^{* (2)}_{{\revar\eta,T}}  )$. Hence, from triangular inequality and the definition of positive part function, we get
\begin{equation}{\label{eq: I4 split}}
I_4^{h, \eta} \le c(\left \| \pi_\eta - \hat{\pi}_{{\revar\eta,T}}^{* (1)} \right \|_A^2 - \frac{V(\eta)}{2})_+ + c (\left \| \pi_\eta - \hat{\pi}_{{\revar\eta,T}}^{* (2)} \right \|_A^2 - \frac{V(\eta)}{2})_+.
\end{equation}
From \eqref{eq: Talagrand per I4}, for a $k$ in the definition of $V(\eta)$ big enough, such that $\frac{k}{2} >  k_0$, we get
{\revar 
\begin{equation}
	\mathbb{E} [\sup_{\eta \in \mathcal{H}_T} I_4^{h, \eta}] \le c_1 T^{c_2} e^{- c_3(\log T)^{c_4} }
	\label{eq: I4}
\end{equation}
with $c_4>1$. To end the proof, it remains to show a similar upper bound with $I_2^{h,\eta}$. From Lemma 3 of \cite{Chapitre 4}} and the definition of kernel function we have 
$$\left \| \pi_{{\revar (h,\eta)}}  - \hat{\pi}^{* (1)}_{{\revar (h,\eta),T}} \right \|_A^2 = \left \| \mathbb{K}_h* ( \pi_\eta  - \hat{\pi}^{* (1)}_{{\revar\eta,T}}) \right \|_A^2 \le \left \| \mathbb{K}_h \right \|^2_{1, \mathbb{R}^d} \left \| \pi_\eta  - \hat{\pi}^{* (1)}_{{\revar\eta,T}} \right \|^2_{2, \tilde{A}} \le c \left \| \pi_\eta  - \hat{\pi}^{* (1)}_{{\revar\eta,T}} \right \|^2_{2, \tilde{A}}.$$
Such a remark, together with \eqref{eq: Talagrand per I4} yields
$$\mathbb{E}[\sup_{\eta \in \mathcal{H}_T} ( \left \| \pi_{{\revar(h,\eta)}}  - \hat{\pi}^{* (1)}_{{\revar(h,\eta),T}} \right \|_A^2 - \frac{\hat{k}}{T} \tilde{H}^2)_+]  \le  T^{c - 1} \log T e^{- c (\log T)^{a(d-2)}} + T^{d b - 2 + c}(\log T)^{4} e^{- c (\log T)^{a}},$$
which clearly implies
\begin{equation*}
	\mathbb{E} [\sup_{\eta \in \mathcal{H}_T} I_2^{h, \eta}] \le c_1 T^{c_2} e^{- c_3(\log T)^{c_4} }.
\end{equation*}
The proof is then concluded.
\end{proof}

{\blue
\subsection{Proof of Theorem \ref{th: adaptive point}}
\begin{proof}
The proof of the pointwise adaptive procedure closely follows the proof of Theorem \ref{th: adaptive} provided above. In particular, analogous to Proposition \ref{prop: A(h)}, the key point is to establish that, for any $h \in \mathcal{H}_T^p$,
\begin{equation}{\label{eq: prop A point}}
 \E[A^p(h,x)] \le c_1 B(h,x) + c_1 T^{- c_2}   
\end{equation}
for some $c_1 > 0$, $c_2 \ge 1$. 
To establish \eqref{eq: prop A point}, we can still employ Berbee's coupling method, which, as seen in Equation \eqref{eq: berbee}, leads us to
\begin{align*}
  A^p(h,x) &\le \sup_{\eta \in \mathcal{H}_T^p} [| \hat{\pi}_{{ (h, \eta),T}}(x) - \hat{\pi}_{{ (h, \eta),T}}^*(x) |^2  + (| \hat{\pi}_{{ (h, \eta),T}}^*(x) - \pi_{{ (h, \eta)}}(x) |^2 - \frac{V^p(\eta)}{2})_+ \nonumber \\
  &+ | \pi_{{ (h, \eta)}}(x) - \pi_\eta(x) |^2 + (| \pi_\eta(x) - \hat{\pi}_{{ \eta,T}}^*(x) |^2 - \frac{V^p(\eta)}{2})_+ + | \hat{\pi}_{{ \eta,T}}^*(x) - \hat{\pi}_{{ \eta,T}}(x) |^2] \\  & =: \sup_{\eta \in \mathcal{H}_T^p}[ \sum_{j = 1}^5 I_j^{p,h,\eta}(x)]. \nonumber 
\end{align*}
The analysis of $I_3^{p,h,\eta}(x)$, $I_1^{p,h,\eta}(x)$ and $I_5^{p,h,\eta}(x)$ is the same as in \eqref{eq: I3} and \eqref{eq: I1 I5}, respectively. It yields
\begin{align*}
\sup_{\eta \in \mathcal{H}_T^p} I_3^{p,h,\eta}(x) \le c B^p(h,x), \qquad \mathbb{E}[ |\sup_{\eta \in \mathcal{H}_T^p} I_1^{p, h, \eta}(x) |]+ \mathbb{E}[ |\sup_{\eta \in \mathcal{H}_T^p} I_5^{p,h, \eta}(x) |] \le c \frac{T^{1 + 2db}}{(\log T)^{2}} e^{- \gamma (\log T)^2}.
\end{align*}
On $I_2^{p,h,\eta}(x)$ and $I_4^{p,h,\eta}(x)$ we want to use Bernstein's inequality. Let us start by considering $I_4^{p,h,\eta}(x)$. Notice that, similarly as in \eqref{eq: I4 split} it is
\begin{equation*}
I_4^{p, h, \eta}(x) \le c(| \pi_\eta(x) - \hat{\pi}_{{\eta,T}}^{* (1)}(x) |^2 - \frac{V^p(\eta)}{2})_+ + c (| \pi_\eta(x) - \hat{\pi}_{{\eta,T}}^{* (2)}(x) |^2 - \frac{V^p(\eta)}{2})_+.
\end{equation*}
Then, 
\begin{align}{\label{eq: start bern}}
\E[\sup_{\eta \in \mathcal{H}_T^p}(| \pi_\eta(x) - \hat{\pi}_{{\eta,T}}^{* (1)}(x) |^2 - \frac{V^p(\eta)}{2})_+] & \le \sum_{\eta \in \mathcal{H}_T^p} \int_0^\infty \mathbb{P}  \Big( (| \pi_\eta(x) - \hat{\pi}_{{\eta,T}}^{* (1)}(x) |^2 - \frac{V^p(\eta)}{2})_+ \ge t  \Big) dt \\
 &= \sum_{\eta \in \mathcal{H}_T^p} \int_0^\infty \mathbb{P} \Big( | \pi_\eta(x) - \hat{\pi}_{{\eta,T}}^{* (1)}(x) | \ge \sqrt{\frac{V^p(\eta)}{2} + t} \Big) dt. \nonumber 
\end{align}
We recall Bernstein's inequality (see for example p.366 \cite{BirgeMassart}): let $T_1$, ... , $T_n$ independent variables and $S_{p}(T) := \sum_{i = 1}^p (T_i - \E[T_i])$ such that $Var(T_i) \le v$, $\left \| T_i \right \|_\infty \le M$. Then, for any $\gamma > 0$, 
\begin{align*}
\mathbb{P} \Big( |S_p(T) - \E[S_p(T)]| \ge p\gamma \Big) & \le 2 \exp(- \frac{p\frac{\gamma^2}{2}}{v + M \gamma}) \\
& \le 2 \max(\exp(\frac{- p \gamma^2}{4 v}), \exp( \frac{-p \gamma}{4 M})).
\end{align*}
Observe it is 
\begin{align*}
\pi_\eta(x) - \hat{\pi}_{{\eta,T}}^{* (1)}(x) &= \frac{1}{p_T q_T} \sum_{i =1}^{p_T} \int_{2(i-1)q_T}^{(2i-1)q_T} (\mathbb{K}_\eta(X_u^* - x) - \E[\mathbb{K}_\eta(X_u^* - x)]) du \\
& = : \frac{1}{p_T} \sum_{i = 1}^{p_T} (T_i - \E[T_i]).
\end{align*}
The computation of $v$, in the same way as $H$ appearing in Theorem \ref{th: adaptive}, follows directly from Proposition \ref{prop: bound variance}.
\begin{align*}
Var(T_i) & = Var(\frac{1}{q_T} \int_{2(i-1)q_T}^{(2i-1)q_T} (\mathbb{K}_\eta(X_u^* - x) - \E[\mathbb{K}_\eta(X_u^* - x)]) du)\\
& \le \frac{c}{q_T} \tilde{H}^2(\eta) =:v,
\end{align*}
where $\tilde{H}^2(\eta)$ is as in \eqref{eq: penalty} and \eqref{eq: penalty point}. \\
Regarding the computation of $M$, as $T_i = \frac{1}{q_T} \int_{2(i-1)q_T}^{(2i-1)q_T} \mathbb{K}_\eta(X_u^* - x) du$, it is $\left \| T_i \right \|_\infty \le c (\prod_{l = 1}^d \eta_l)^{-1} =: M$. It is easy to see that it is worse compared to the integrated $L^2$ norm studied in previous theorem, where $\left \| r \right \|_\infty$ was bounded by $c (\prod_{l = 1}^d \eta_l)^{-\frac{1}{2}}$. Bernstein's inequality gives, for any $t \ge 0$
\begin{align*}
\mathbb{P} \Big( | \pi_\eta(x) - \hat{\pi}_{{\eta,T}}^{* (1)}(x) | \ge \sqrt{\frac{V^p(\eta)}{2} + t} \Big) & \le 2 \max \Big(\exp(- \frac{p_T}{4v}(\frac{V^p(\eta)}{2} + t)),\\
& \exp(- \frac{p_T}{8M} \sqrt{\frac{V^p(\eta)}{2}}) \exp(- \frac{p_T}{8M}\sqrt{t}) \Big),
\end{align*}
as $\sqrt{\cdot}$ is a concave function. Then, using the definition of $V^p(\eta)$, $v$ and $M$ one has, for some $c_1$, $c_2$, $c_3$ and $c_4 > 0$ the following identities: $\frac{p_T}{4v} \frac{V^p(\eta)}{2} = c_1 {k_T}$, $\frac{p_T}{4v} = c_2  \frac{T}{\tilde{H}^2(\eta)}$, $\frac{p_T}{M} = \frac{c_3 T}{q_T} \prod_{l = 1}^d \eta_l$ and $\frac{p_T}{M} \sqrt{\frac{V^p(\eta)}{2}} = c_4 \frac{\sqrt{T}}{q_T} \prod_{l = 1}^d \eta_l \tilde{H}(\eta) \sqrt{k_T}$, 
where we have also used that $T = 2 p_T q_T$ and we recall that $k_T = k \log T$, $q_T = (\log T)^2$. It follows, using also \eqref{eq: start bern}, 
\begin{align*}
& \E[\sup_{\eta \in \mathcal{H}_T^p}(| \pi_\eta(x) - \hat{\pi}_{{\eta,T}}^{* (1)}(x) |^2 - \frac{V^p(\eta)}{2})_+] \le \sum_{\eta \in \mathcal{H}_T^p} \int_0^\infty 2 \max \Big(\exp(- c {k_T})\exp(- c \frac{T}{\tilde{H}^2(\eta)}t), \\
& \exp(- c \sqrt{k_T} \frac{\sqrt{T}}{(\log T)^2} \prod_{l = 1}^d \eta_l \tilde{H}(\eta) \alpha) \exp(- \frac{c T}{(\log T)^2} \prod_{l = 1}^d \eta_l(1 - \alpha) \sqrt{t}) \Big) dt \\
& \le \sum_{\eta \in \mathcal{H}_T^p} \exp(- c {k_T}) \int_0^\infty \exp(- c \frac{T}{\tilde{H}^2(\eta)}t)  dt \\
& + \sum_{\eta \in \mathcal{H}_T^p} \exp(- c \sqrt{k_T} \frac{\sqrt{T}}{(\log T)^2} \prod_{l = 1}^d \eta_l \tilde{H}(\eta) \alpha) \int_0^\infty \exp(- \frac{c T}{(\log T)^2} \prod_{l = 1}^d \eta_l(1 - \alpha) \sqrt{t}) dt \\
& \le c \sum_{\eta \in \mathcal{H}_T^p} \Big[\exp(- c {k_T}) \frac{\tilde{H}^2(\eta)}{T} + \exp(- c \sqrt{k_T} \frac{\sqrt{T}}{(\log T)^2} \prod_{l = 1}^d \eta_l \tilde{H}(\eta) \alpha) \frac{(\log T)^4}{(T \prod_{l = 1}^d \eta_l (1 - \alpha))^2}\Big].
\end{align*}
Recall that, according to the definition of $\mathcal{H}_T^p$ given in \eqref{eq: def mathcal Hp}, it is $\prod_{l = 1}^d \eta_l \tilde{H}(\eta ) \ge \frac{(\log T)^{2 + a}}{\sqrt{T}}$. Moreover, by the definition of $\tilde{H}$ and the fact that for any $\eta \in \mathcal{H}_T^p$ it is 
$(\frac{1}{T})^b \le \frac{1}{\eta_l} \le (\frac{1}{\log T})^{\frac{1}{d - 2} + a}$, one obtains $\tilde{H}(\eta) \le (\frac{1}{\log T})^{a(d - 2)}$. It follows 
\begin{align*}
& \E[\sup_{\eta \in \mathcal{H}_T^p}(| \pi_\eta(x) - \hat{\pi}_{{\eta,T}}^{* (1)}(x) |^2 - \frac{V^p(\eta)}{2})_+] \\ 
& \le c \sum_{\eta \in \mathcal{H}_T^p} \Big[\exp(- c {k_T}) \frac{1}{T(\log T)^{a (d-2)}} + \exp(- c \sqrt{k_T} (\log T)^a \alpha ) \frac{(\log T)^4}{T^2 (1 - \alpha)^2 (\log T)^{(\frac{1}{d - 2} + a)2d}}\Big] \\
& \le c_1 T^{c_2} (e^{- c {k} \log T} + e^{- c (\log T)^{a + \frac{1}{2}}}) \le c_1 T^{- c_3},
\end{align*}
where we have also used that $|\mathcal{H}_T^p|$ has polynomial growth in $T$ at most as $T^{c}$. Remark that the result above is in hold up to choose $k$ large enough to guarantee that $c {k} - c_2 \ge c_3 \ge 1$. It yields the wanted bound on $\E[\sup_{\eta \in \mathcal{H}_T^p} I_4^{p, h, \eta}(x)]$. To conclude, we study $I_2^{p, h, \eta}(x)$. Observe that such term can be treated exactly as the previous one. Indeed, instead of $T_i$ defined as above we have now to deal with $\tilde{T}_i := = \frac{1}{q_T} \int_{2(i-1)q_T}^{(2i-1)q_T} \mathbb{K}_h \ast \mathbb{K}_\eta(X_u^* - x) du $. Notice that 
\begin{align*}
\left \| \tilde{T}_i \right \|_\infty \le c \left \| \mathbb{K}_h \right \|_1 (\prod_{l = 1}^d \eta_l)^{- 1} = M.
\end{align*}
Moreover using Fubini Theorem and the fact that both $\left \| \mathbb{K}_h \right \|_1$ and $\left \| \mathbb{K}_\eta \right \|_1$ are bounded by a constant, it is easy to check the same for  $\left \| \mathbb{K}_h \ast \mathbb{K}_\eta \right \|_1$. Furthermore, using that $\mathbb{K}_h \ast \mathbb{K}_\eta$ has support in $[-(\eta + h), \eta + h]$ and the improved upper bound $\left \| \mathbb{K}_h \ast \mathbb{K}_\eta \right \|_\infty \le c (\prod_{l = 1}^d (h_l + \eta_l))^{-1}$, one can follow step by step the proof given in Proposition \ref{prop: bound variance} to verify that it holds true with $\mathbb{K}_h \ast \mathbb{K}_\eta$ instead of $\mathbb{K}_h$. It implies 
$$Var(\tilde{T}_i) \le \frac{c}{q_T} \tilde{H}^2(\eta)= v.$$
Then, applying Bernstein's inequality in the same manner as demonstrated in the proof concerning $I_4^{p, h, \eta}(x)$ directly yields
$$\E[\sup_{\eta \in \mathcal{H}_T^p} I_2^{p, h, \eta}(x)] \le c_1 T^{- c_2}$$
for some $c_2 \ge 1$, as we wanted. 
  
\end{proof}}

\subsection{Proof of Theorem \ref{th: adaptive optimal}}

\begin{proof}
We recall that, according to Theorem \ref{th: upper bound d ge 3}, the optimal choice for the bandwidth changes depending on how many $\beta$'s are coincident. This is equivalent to understand which term achieves the minimum in the definition of the penalty term. In particular, when the $\beta$'s are not ordered we have for any bandwidth $h$
\begin{align*}
Var(\hat{\pi}_h) & \le \frac{c}{T} \, \min \Big(\frac{\sum_{j = 1}^d |\log h_j|}{\prod_{l \neq k_1, k_2} h_l}, \, \frac{1}{\sqrt{h_{k_2} h_{k_3}}\prod_{l \neq k_1, k_2, k_3} h_l} \Big),
\end{align*}
with $h_{k_1} \le h_{h_2} \le ... \le h_{k_d}$ by construction. Then, if the minimum is realized by the first, it means that $h_{k_3} > h_{k_2}$. In this case the rate optimal choice for the bandwidth consists in taking
\begin{equation}{\label{eq: bandwidth 49.5}}
h_{k_j} := (\frac{\log T}{T})^{\frac{{\revar\bar{\beta}_3}}{\beta_{k_j} (2 {\revar \bar{\beta}_3} + d - 2)}}
\end{equation}
 for any $j \ge 3$ and $h_{k_1}$ and $h_{k_2}$ arbitrarily small. It provides the convergence rate $(\frac{\log T}{T})^{\frac{2{\revar \bar{\beta}_3}}{2{\revar \bar{\beta}_3} + d - 2}}$. \\
We want such a choice of bandwidth to be in \eqref{eq: example HT}. Hence, we start by checking the conditions in \eqref{eq: def mathcal H} are respected. One can observe that, for $h_{k_3} > h_{k_2}$, it is 
$$\min((\sum_{j=1}^d |\log h_j| \, h_{k_1} h_{k_2})^{\frac{1}{2}}, (h_{k_1})^{\frac{1}{2}} (h_{k_2} h_{k_3})^\frac{1}{4}) = (\sum_{j=1}^d |\log h_j| \, h_{k_1} h_{k_2})^{\frac{1}{2}}.$$
Then, we need to choose $h_{k_1}$ and $ h_{k_2}$ large enough to guarantee $(\sum_{j=1}^d |\log h_j| \, h_{k_1} h_{k_2})^{\frac{1}{2}} \ge \frac{c (\log T)^{2 + a}}{\sqrt{T}}$ for some $a > 0$ arbitrarily small. It leads us to the choice 
{\revar $h_{k_1} = h_{k_2} = \frac{(\log T)^{\frac{3}{2} + a}}{\sqrt{T}}$.}
It is easy to check that such a choice still implies the convergence rate to be $(\frac{\log T}{T})^{\frac{2{\revar\bar{\beta}_3}}{2 {\revar\bar{\beta}_3} + d - 2}}$. Indeed, the bias term associated {\revar to $h_{k_1}$ and $h_{k_2}$ is}
\begin{align*}
{\revar h_{k_1}^{2 \beta_{k_1}} + h_{k_2}^{2 \beta_{k_2}}} & = (\frac{(\log T)^{\frac{3}{2} + a}}{\sqrt{T}})^{2 \beta_{k_1}} + (\frac{(\log T)^{\frac{3}{2} + a}}{\sqrt{T}})^{2 \beta_{k_2}} \\
& \le (\frac{\log T}{T})^{\frac{2{\revar\bar{\beta}_3}}{2{\revar\bar{\beta}_3} + d - 2}}, 
\end{align*}
being the last a consequence of the fact that both $\beta_{k_1}$ and $\beta_{k_2}$ are larger than $1$, which is larger than $\frac{2{\revar\bar{\beta}_3}}{2 {\revar\bar{\beta}_3} + d - 2}$. \\
Clearly the optimal choice of the bandwidth with components $h_{k_l}$ as above satisfies the conditions in \eqref{eq: def mathcal H}. 
\\
\\
When the minimum of the variance is achieved by $\frac{1}{\sqrt{h_{k_2} h_{k_3}}\prod_{l \neq k_1, k_2, k_3} h_l}$, instead, it implies that $h_{k_2} = h_{k_3}$. According to Theorem \ref{th: upper bound d ge 3},
the rate optimal choice for the bandwidth is similar as in previous case, consisting in 
\begin{equation}{\label{eq: choice h adaptive}}
h_{k_j} := (\frac{1}{T})^{\frac{{\revar\bar{\beta}_3}}{\beta_{k_j} (2 {\revar\bar{\beta}_3} + d - 2)}}, \qquad \forall j \ge 2.
\end{equation}
Now the only bandwidth allowed to be arbitrarily small is $h_{k_1}$. In this case the convergence rate is $(\frac{1}{T})^{\frac{2{\revar\bar{\beta}_3}}{2 {\revar\bar{\beta}_3} + d - 2}}$. \\
Again, we need to check that conditions in \eqref{eq: def mathcal H} are respected. One can observe that, as $h_{k_3} = h_{k_2}$ (which coincides with $\beta_{k_3} = \beta_{k_2}$ because of \eqref{eq: choice h adaptive}), it is 
\begin{align}{\label{eq: cond 61.5}}
\min((\sum_{j=1}^d |\log h_j| \, h_{k_1} h_{k_2})^{\frac{1}{2}}, (h_{k_1})^{\frac{1}{2}} (h_{k_2} h_{k_3})^\frac{1}{4}) & = (h_{k_1})^{\frac{1}{2}} (h_{k_2} h_{k_3})^\frac{1}{4} \\
& = (h_{k_1} h_{k_2})^{\frac{1}{2}}. \nonumber
\end{align}
Then, we need to choose $h_{k_1}$ and $ h_{k_2}$ such that $ (h_{k_1} h_{k_2})^{\frac{1}{2}} \ge \frac{ (\log T)^{2 + a}}{\sqrt{T}}$ for some $a > 0$ arbitrarily small which, according to \eqref{eq: choice h adaptive}, leads us to the condition
$$h_{k_1}^{\frac{1}{2}} \ge \frac{ (\log T)^{2 + a}}{\sqrt{T}} T^{\frac{{\revar\bar{\beta}_3}}{2\,\beta_{k_2} (2 {\revar\bar{\beta}_3} + d - 2)}}.$$ 
This is equivalent to the constraint 
$$h_{k_1} \ge \frac{ (\log T)^{4 + 2a}}{{T}} T^{\frac{{\revar\bar{\beta}_3}}{\beta_{k_2} (2 {\revar\bar{\beta}_3} + d - 2)}}.$$
We choose $h_{k_1}$ realising the equivalence above, which implies the convergence rate to be 
$(\frac{1}{T})^{\frac{2{\revar\bar{\beta}_3}}{2 {\revar\bar{\beta}_3} + d - 2}}$, as the bias term associated to $h_{k_1}$ is negligible. Indeed, it is 
\begin{align*}
h_{k_1}^{2 \beta_{k_1}}& = (\log T)^{2 \beta_{k_1} (4 + 2a)} (\frac{1}{T})^{
	{\revar 2 \beta_{k_1}(1 - \frac{{\revar\bar{\beta}_3}}{\beta_{k_2}(2{\revar\bar{\beta}_3} + d - 2)})}}.
\end{align*}
This is negligible compared to $(\frac{1}{T})^{\frac{2{\revar\bar{\beta}_3}}{2{\revar\bar{\beta}_3} + d - 2}}$ using that $d \ge 3$ and $\beta_{k_1} \ge1$.
We deduce the optimal choice of the bandwidth with components $h_{k_l}$ as above satisfies the conditions in \eqref{eq: def mathcal H}. \\
\\
Then, in order to say that the chosen bandwidths belong to the set of candidate bandwidths $\mathcal{H}_T$ proposed in \eqref{eq: example HT} we should say they have the particular form $h_l = \frac{1}{z_l}$ for {\revar $z_l \in  \{1,... ,  \lfloor T \rfloor  \}$,}
that in general is not the case. However, we can replace $h_{k_j}$ in \eqref{eq: bandwidth 49.5} with 
$\tilde{h}_{k_j}(T) := \frac{1}{\lfloor T^{\frac{{\revar\bar{\beta}_3}}{\beta_{k_j} (2 {\revar\bar{\beta}_3} + d - 2)}}\rfloor }$, which is asymptotically equivalent and which leads to the same convergence rate. Moreover, $\tilde{h}_{k_j}(T)$ has the wanted form, which entails $\tilde{h}_{k_j}$ belongs to $\mathcal{H}_T$, defined as in \eqref{eq: example HT}. \\ 
We act similarly with $h_{k_1}$ and $h_{k_2}$ (or only on $h_{k_1}$, when $h_{k_2} = h_{k_3}$), to replace the rate optimal choice $h(T)$ with $\tilde{h}(T)$, which belongs to the set of candidate bandwidth $\mathcal{H}_T$, defined as in \eqref{eq: example HT}. From Theorem \ref{th: upper bound d ge 3} it follows that the $inf$ for $h \in \mathcal{H}_T$ of $B (h) + V (h)$ is clearly realized by $\tilde{h}_T$.
Thus, the bound stated in Theorem \ref{th: adaptive} is actually
\begin{equation*}
\mathbb{E}[\left \| \hat{\pi}_{\tilde{h},T} - \pi \right \|^2_A] \underset{\sim}{<}
\begin{cases}
(\frac{\log T}{ T})^ {\frac{2{\revar\bar{\beta}_3}}{2{\revar\bar{\beta}_3} + d - 2}} + e^{ - c_1 (\log T)^{c_2}} \qquad \mbox{if } \beta_{k_2} < \beta_{k_3} \\
(\frac{1}{ T})^ {\frac{2{\revar\bar{\beta}_3}}{2{\revar\bar{\beta}_3} + d - 2}} + e^{ - c_1 (\log T)^{c_2}} \qquad \mbox{if } \beta_{k_2} = \beta_{k_3},
\end{cases}
\end{equation*}
as we wanted.
\end{proof}

}

{\blue
\subsection{Proof of Theorem \ref{th: adaptive optimal point}}
\begin{proof}
The key point in this theorem consists in looking for the optimal choice of the bandwidth, that realizes the trade-off between the bias term $\sum_{j = 1}^d h_{k_j}^{\beta_{k_j}}$ and the variance \\
$\frac{\log T}{T} \, \min \Big(\frac{\sum_{j = 1}^d |\log h_j|}{\prod_{l \neq k_1, k_2} h_l}, \, \frac{1}{\sqrt{h_{k_2} h_{k_3}}\prod_{l \neq k_1, k_2, k_3} h_l} \Big)$.
Such balance will be different depending on the fact that the min in the variance is achieved by the first term or the second. \\
\\
$\bullet$ Case 1. \\
If $\min \Big(\frac{\sum_{j = 1}^d |\log h_j|}{\prod_{l \neq k_1, k_2} h_l}, \, \frac{1}{\sqrt{h_{k_2} h_{k_3}}\prod_{l \neq k_1, k_2, k_3} h_l} \Big) = \frac{\sum_{j = 1}^d |\log h_j|}{\prod_{l \neq k_1, k_2} h_l}$ then, as already said in previous theorem, it means that $h_{k_3} > h_{k_2}$. For $l \neq k_1, k_2$ let us introduce $h_l(T) := (\log T)^{a_l} (\frac{1}{T})^{b_l}$. We now look for an optimal choice for $a_l$ and $b_l$, for which the following identities are true:
\begin{align*}
(\log T)^{2 - \sum_{l \neq k_1, k_2} a_l}(\frac{1}{T})^{1 - \sum_{l \neq k_1, k_2} b_l}= (\log T)^{2 a_{k_3} \beta_{k_3}}(\frac{1}{T})^{2 b_{k_3} \beta_{k_3}} = ... = (\log T)^{2 a_{k_d} \beta_{k_d}}(\frac{1}{T})^{2 b_{k_d} \beta_{k_d}}. 
\end{align*}
It implies we want $a_{k_3}$, $b_{k_3}$, ... , $a_{k_d}$, $b_{k_d}$ satisfying the following system: 
\begin{equation*}
    \begin{cases}
     2 a_{k_3} \beta_{k_3} = ... = 2 a_{k_d} \beta_{k_d} = 2 - \sum_{l \neq k_1, k_2} a_l \\
 2 b_{k_3} \beta_{k_3} = ... = 2 b_{k_d} \beta_{k_d} = 1 - \sum_{l \neq k_1, k_2} b_l. 
    \end{cases}
\end{equation*}
It leads to $a_l = \frac{2 \bar{\beta}_3}{\beta_l(2 \bar{\beta}_3 + d - 2 )}$ and $b_l = \frac{ \bar{\beta}_3}{\beta_l(2 \bar{\beta}_3 + d - 2 )}$ for any $l \in \{ k_3, ... , k_d \}$. Then, the rate optimal choice for the bandwidth consists in taking 
\begin{equation}{\label{eq: hkj 1}}
h_{k_j} := (\frac{\log^2T}{T})^{\frac{ \bar{\beta}_3}{\beta_{k_j}(2 \bar{\beta}_3 + d - 2 )}}
\end{equation}
for any $j \ge 3$ and $h_{k_1}$, $h_{k_2}$ arbitrarily small. It yields the convergence rate $(\frac{\log^2T}{T})^{\frac{ 2\bar{\beta}_3}{2 \bar{\beta}_3 + d - 2 }}$. Now we want to check, similarly as in the proof of Theorem \ref{th: adaptive optimal}, that such bandwidth belongs to $\mathcal{H}_t^p$ as defined in \eqref{eq: example HT point}. We therefore want to ensure that the conditions in \eqref{eq: def mathcal Hp} are respected. It implies we have to choose $h_{k_1}$ and $h_{k_2}$ large enough to ensure that 
$$(\prod_{l = 1}^d h_l)^\frac{1}{2} \min((\sum_{j=1}^d |\log h_j| \, h_{k_1} h_{k_2})^{\frac{1}{2}}, (h_{k_1})^{\frac{1}{2}} (h_{k_2} h_{k_3})^\frac{1}{4}) \ge \frac{c (\log T)^{2 + a}}{\sqrt{T}} $$
for some $a$ arbitrarily small.
Remark that, as we are conducting our analysis in the case where $h_{k_3} > h_{k_2}$, it is
$$\min((\sum_{j=1}^d |\log h_j| \, h_{k_1} h_{k_2})^{\frac{1}{2}}, (h_{k_1})^{\frac{1}{2}} (h_{k_2} h_{k_3})^\frac{1}{4}) = (\sum_{j=1}^d |\log h_j| \, h_{k_1} h_{k_2})^{\frac{1}{2}}.$$ 
Then, we have to guarantee that $(\prod_{l \neq k_1, k_2} h_l)^\frac{1}{2} h_{k_1} h_{k_2} (\sum_{j=1}^d |\log h_j|)^\frac{1}{2} \ge \frac{c (\log T)^{2 + a}}{\sqrt{T}}$. According to the choice of $h_{k_j}$ for $j \ge 3$ gathered in \eqref{eq: hkj 1}, it holds true if and only if 
$[(\frac{\log^2T}{T})^{\frac{ d - 2}{2 \bar{\beta}_3 + d - 2 }} \log T]^\frac{1}{2} h_{k_1} h_{k_2} \ge \frac{c (\log T)^{2 + a}}{\sqrt{T}} $. It leads us to the choice 
\begin{equation}{\label{eq: hk1 hk2 2}}
h_{k_1} = h_{k_2} = (\log T)^{\frac{ \bar{\beta}_3 }{2 \bar{\beta}_3 + d - 2 } + 1 + a} (\frac{1}{T})^{\frac{ \bar{\beta}_3 }{2(2 \bar{\beta}_3 + d - 2) }}.
\end{equation}
Observe such choice will still imply the convergence rate to be $(\frac{\log^2T}{T})^{\frac{ 2 \bar{\beta}_3}{2 \bar{\beta}_3 + d - 2 }}$ as we have assumed $\beta_{k_2} \ge \beta_{k_1} > 2$ and so it is 
\begin{align*}
h_{k_1}^{2 \beta_{k_1}} + h_{k_2}^{2 \beta_{k_2}} & = (\log T)^{(\frac{ \bar{\beta}_3 }{2 \bar{\beta}_3 + d - 2 } + 1 + a) 2 \beta_{k_1}} (\frac{1}{T})^{\frac{ \bar{\beta}_3 \beta_{k_1} }{2 \bar{\beta}_3 + d - 2}} + (\log T)^{(\frac{ \bar{\beta}_3 }{2 \bar{\beta}_3 + d - 2 } + 1 + a) 2 \beta_{k_2}} (\frac{1}{T})^{\frac{ \bar{\beta}_3 \beta_{k_2} }{2 \bar{\beta}_3 + d - 2}} \\
& \le (\frac{\log^2T}{T})^{\frac{ 2 \bar{\beta}_3}{2 \bar{\beta}_3 + d - 2 }}.
\end{align*}
It implies the rate optimal choice as in \eqref{eq: hkj 1}, \eqref{eq: hk1 hk2 2} satisfies the constraints in \eqref{eq: def mathcal Hp}. \\
\\
$\bullet$ Case 2. \\
Assume now that $\min \Big(\frac{\sum_{j = 1}^d |\log h_j|}{\prod_{l \neq k_1, k_2} h_l}, \, \frac{1}{\sqrt{h_{k_2} h_{k_3}}\prod_{l \neq k_1, k_2, k_3} h_l} \Big) = \frac{1}{\sqrt{h_{k_2} h_{k_3}}\prod_{l \neq k_1, k_2, k_3} h_l}$, which means that $h_{k_2} = h_{k_3}$. Acting as in Case 1 it is easy to check that the rate optimal choice for the bandwidth consists in taking 
\begin{equation}{\label{eq: hkj 3}}
h_{k_j} := (\frac{\log T}{T})^{\frac{ \bar{\beta}_3}{\beta_{k_j}(2 \bar{\beta}_3 + d - 2 )}}
\end{equation}
for $j \ge 2$, leading to the convergence rate $(\frac{\log T}{T})^{\frac{ 2 \bar{\beta}_3}{2 \bar{\beta}_3 + d - 2}}$. \\
Because of \eqref{eq: cond 61.5}, we have to choose $h_{k_1}$ such that
$(\prod_{l = 1}^d h_l \, h_{k_1} \, h_{k_2} )^\frac{1}{2} \ge \frac{(\log T)^{2 + a}}{\sqrt{T}}$. Replacing for $j \ge 2$ $h_{k_j}$ as in \eqref{eq: hkj 3}, the constraint here above becomes 
$$(\frac{\log T}{T})^{\frac{1}{2} \frac{d - 2}{2 \bar{\beta}_3 + d - 2 }} (\frac{\log T}{T})^{\frac{\bar{\beta}_3}{ \beta_{k_2}(2 \bar{\beta}_3 + d - 2) }} h_{k_1} \ge \frac{(\log T)^{2 + a}}{\sqrt{T}} .$$ It leads us to the choice 
\begin{equation}{\label{eq: hk1 4}}
h_{k_1} = (\log T)^{2 + a - \frac{1}{2 \bar{\beta}_3 + d - 2}(\frac{d - 2}{2} + \frac{2 \bar{\beta}_3}{ \beta_{k_2}})} (\frac{1}{T})^{\frac{\bar{\beta}_3}{2 \bar{\beta}_3 + d - 2}(1 - \frac{1}{\beta_{k_2}})}.
\end{equation}
Then $h_{k_1}^{2 \beta_{k_1}}$ is negligible compared to $(\frac{\log T}{T})^{\frac{2 \bar{\beta}_3}{2 \bar{\beta}_3 + d - 2}}$ for $\beta_{k_1}(1 - \frac{1}{\beta_{k_2}}) > 1$, which holds true as $\beta_{k_2} \ge \beta_{k_1} > 2$ by hypothesis. Then, the rate optimal choice for the bandwidth as in \eqref{eq: hkj 3}, \eqref{eq: hk1 4} satisfies the constraints in \eqref{eq: def mathcal Hp}. \\
\\
The concluding argument is then the same as in the proof of Theorem \ref{th: adaptive optimal}. Indeed, to say that the bandwidths belong to \eqref{eq: example HT point}, they should have the form $h_l = \frac{1}{z_l}$ for some $z_l \in \{1, ... , \lfloor T \rfloor \}$, that is not true in general. However, up to replace $h_{k_j}$ with $\tilde{h}_{k_j}$, where the integer part of $T$ has come into play, we obtain asymptotically equivalent bandwidth achieving the same convergence rate and belonging to the set of candidate bandwidth $\mathcal{H}_T^p$. The proof is then concluded.
\end{proof}}

\section{Proof of the lower bounds stated in Section \ref{S: Lower bounds}}{\label{s: proof lower}}
This section is devoted to the proof of the lower bounds, as stated in Section \ref{S: Lower bounds}.

\subsection{Proof of Theorem \ref{th: lower bound}}
The proof of Theorem \ref{th: lower bound} is based on the two hypothesis method, as explained for example in Section 2.3 of Tsybakov \cite{Ts} and follows the standard scheme provided in Section 6 of \cite{Lower bound}. We start by making explicit link between the drift and the stationary measure. Then, we provide two priors depending on some calibration parameters and, to conclude, we find some conditions on the calibration such that it is possible to prove a lower bound for the minimax risk introduced in \eqref{eq: def minimax risk}.

{\modarn First, remark that {\rev it} is possible to restrict the class
	of coefficients $\Sigma(\beta, \mathcal{L},a_{\text{min}},b_0, a_0,a_1,b_1,\tilde{C},\tilde{\rho})$ to a class where the drift coefficient is bounded, by setting
	$$
	\tilde{\Sigma}=\tilde{\Sigma}(\beta, \mathcal{L},a_{\text{min}},b_0, a_0,a_1,b_1,\tilde{C},\tilde{\rho})
	:=\{(a,b) \in {\Sigma}(\beta, \mathcal{L},a_{\text{min}},b_0, a_0,a_1,b_1,\tilde{C},\tilde{\rho}) \mid \norm{b}_\infty \le b_0 \}, 
	$$
	and observing
	that it is sufficient to lower bound the minimax risk \eqref{eq: def minimax risk} with the class $\tilde{\Sigma}$
	to get the result on the larger class ${\Sigma}(\beta, \mathcal{L},a_{\text{min}},b_0, a_0,a_1,b_1,\tilde{C},\tilde{\rho})$.
	Now, it is possible to check that the  class of coefficients $\tilde{\Sigma}$ is invariant by translation in the following way.
For $h\in\mathbb{R}^d$, if $(a,b) \in \tilde{\Sigma}(\beta, \mathcal{L},a_{\text{min}},b_0, a_0,a_1,b_1,\tilde{C},\tilde{\rho})$, then
$(a(\cdot+h),b(\cdot+h))\in\tilde{\Sigma}(\beta, \mathcal{L},a_{\text{min}},b_0, a_0,a_1,b_1,\tilde{C/2},\tilde{\rho}')$ where
$\tilde{\rho}'=	\max(\tilde{\rho}+|h|, 2 |h|[\frac{\tilde{C}+b_0}{\tilde{C}}])$. As a consequence, it is sufficient to prove the theorem in the case where $x_0=(0,\dots,0)$, and the general case can be deduced by translation.
A second remark, is that {\rev one can replace} any $X$ solution of  \eqref{eq: model} with the process $t\mapsto a_{\text{min}}^{-1}X_t$. {\rev Then,} it is possible to assume without loss of generality that $a_{\text{min}}=1$ and $a_0 \ge 1$, in the statement of Theorem \ref{th: lower bound}.}

In the proof, we will lower bound the risk on the subclass of model \eqref{eq: model} given by the following simpler stochastic differential equation, for which the diffusion coefficient $a=\mathbb{I}_{d\times d}$ is constant :
\begin{equation}
dX_t=b(X_t)dt + dW_t.
\label{eq: model lower bound}
\end{equation}
and where $b$ is any drift function such that $(\mathbb{I}_{d \times d},b) \in {\modarn \tilde{\Sigma}}(\beta, \mathcal{L},1,b_0,a_0,a_1,b_1,\tilde{C},\tilde{\rho})$
{\revar for some fixed constants $b_0>0$, $a_0>1$, $a_1>0$, $b_1>0$ and with $\tilde{C}$, $\tilde{\rho}>0$ that will be specified later.}

\begin{proof}
$\bullet$ Explicit link between the drift and the stationary measure. \\
We first of all need to introduce $A$, the generator of the diffusion $X$ solution of \eqref{eq: model lower bound}:
\begin{equation}
A f (x) := \frac{1}{2} \sum_{i,j = 1}^d \frac{\partial^2}{\partial x_i \partial x_j } f(x) + \sum_{i = 1}^d b^i(x) \frac{\partial }{\partial x_i } f(x).
\label{eq: definition generator}
\end{equation}
We now introduce a class of {\rev functions} that will be useful in the sequel:
\begin{equation*}
\begin{split}
\mathcal{C} := & \left \{ f : \mathbb{R}^d \rightarrow \mathbb{R}, \,  f \in C^2(\mathbb{R}^d) \mbox{ such that } \forall i \in \left \{ 1, ..., d \right \} \, \lim_{x_i \rightarrow \pm \infty} f(x) = 0, \right.\\
&  \left. \lim_{x_i \rightarrow \pm \infty} \frac{\partial}{ \partial x_i} f(x) = 0 \mbox{  and } \int_{\mathbb{R}^d} f(x) dx < \infty  \right \}.
\end{split}
\end{equation*}
We denote furthermore as $A^*$ the adjoint operator of $A$ on ${\modar \mathbf{L}^2(\mathbb{R}^d,dx)}$ which is such that, for $f, g \in \mathcal{C}$,
$$\int_{\mathbb{R}^d} A f (x) g(x) dx = \int_{\mathbb{R}^d} f(x) A^* g (x) dx.$$
The form of $A^*$ is known (see for example Lemma 2 in \cite{Lower bound} remarking that, in our case, the discrete part of the generator $A_d$ is zero and so we do not have its adjoint in $A^*$):
$$ A^*g (x) = \frac{1}{2} \sum_{i = 1}^d \frac{\partial^2}{\partial x_i^2 } g (x) - (\sum_{i = 1}^d \frac{\partial \,  b^i}{\partial x_i } g (x) + b^i \frac{\partial \, g}{\partial x_i }  (x)).$$
If $g : \mathbb{R}^d \rightarrow \mathbb{R}$ is a probability density of class $\mathcal{C}^2$, solution of $A^* g = 0$, then it can be checked by Ito's formula it is an invariant density for the process we are considering. When the stationary distribution $\pi$ is unique, therefore, it can be computed as solution of the equation $A^* \pi = 0$. As proposed for example in \cite{Lower bound} and in \cite{DGY}, we
consider $\pi$ as fixed and $b$ as the unknown variable. We therefore want to compute a function $b = b_g$ solution to $A^* g = 0$.
For $g \in \mathcal{C}$ and $g > 0$, we introduce for all $x \in \mathbb{R}^d$ and for all $i \in \left \{1, ... , d \right \}$,
\begin{equation}
b^i_g (x) = \frac{1}{g (x)} \frac{1}{2}  \frac{\partial g}{\partial x_i} (x)
\label{eq: def b sol A*}
\end{equation}
and $b_g(x) = (b^1_g (x), ... , b^d_g(x))$. It is enough to remark that, for $b^i_g (x)$ defined as above it is
\begin{equation}
\frac{\partial \,  b_g^i}{\partial x_i } (x) = \frac{1}{g (x)} \frac{1}{2} \frac{\partial^2 g}{\partial x_i^2} (x) - \frac{b^i_g (x)}{g(x)} \frac{\partial \,  g}{\partial x_i } (x)
\label{eq: deriv b}
\end{equation}
to see that the function $b_g$ here above introduced is actually solution of $A^* g (x) = 0$, for any $x \in \R^d$. We know that $\pi$ is solution to $A^* \pi (x) = 0$ for $b = b_\pi$ and so it is a stationary measure for the process $X$ whose drift is $b_\pi$. However, if $b_\pi$ satisfies A1-A2 then, from Lemma \ref{lemma: ergodicity}, we know there exists a Lyapounov function and that the stationary measure of the equation with drift coefficient $b_\pi$ is unique. It follows it is equal to $\pi$.

Hence, we need $b_\pi$ to be a function satisfying A1-A2. We introduce some assumptions on $\pi$ for which the associated drift $b_\pi$ has the wanted properties.\\
\\
\textbf{A3:} Let $\pi : \mathbb{R}^d \rightarrow \mathbb{R}$ a probability density with regularity $\mathcal{C}^2$ such that, for any $x = (x_1, ... , x_d) \in \mathbb{R}^d$, $\pi (x) = c_n \prod_{j = 1}^d \pi_j(x_j) > 0$, where $c_n$ is a normalization constant. We suppose moreover that the following holds true for each $j \in \left \{ 1 , ... , d \right \}$:
\begin{enumerate}
	\item $\lim_{y \rightarrow \pm \infty} \pi_j (y) = 0$ and $\lim_{y \rightarrow \pm \infty} \pi'_j (y) = 0$.
	\item There exists $\tilde{\epsilon} > 0$ and ${\revar R_0} > 0$ such that, for any $y$ : $|y| > \frac{{\revar R_0}}{\sqrt{d}}$, 
	$$\frac{\pi'_j(y)}{\pi_j(y)} \le - \tilde{\epsilon}  \sgn(y). $$
	\item There exists a constant $c_1$ such that, for any $y \in \mathbb{R}$,
	$$|\frac{\pi'_j(y)}{\pi_j(y)}| \le c_1. $$
	\item There exists a constant $c_2$ such that, for any $x \in \R^d$, 
	$$|\frac{\partial^2 \pi}{\partial x_i^2} (x)| \le c_2 \pi(x).$$
\end{enumerate}
The properties listed here above have been introduced in order to make the associated drift function satisfying A1-A2, {\modarn and $b$ bounded}, so that, up to know that $\pi_{(a,b)}\in \mathcal{H}_d (\beta, 2 \mathcal{L})$, it would follow $(I_{d \times d}, b) \in {\modarn \tilde{\Sigma}}$. It is easy to see from the definition of $b_{\pi}$ given in \eqref{eq: def b sol A*} that, having $\pi$ in a multiplicative form, we get
$$|b^i_{\pi} (x)| = |\frac{1}{2} \frac{\pi'_i(x_i)}{\pi_i (x_i)}| \le \frac{c_1}{2} =: \tilde{b}_0, $$
where we have used the third point of A3. The drift function is also clearly {\modarn bounded}. Moreover from \eqref{eq: deriv b}, the third and the fourth points of A3 and the just proven boundedness of $b_{\pi}$ we have 
$$|\frac{\partial b^i_{\pi}}{\partial x_i}(x)| \le \frac{c_2}{2} + c_1 \tilde{b}_0 =: \tilde{b}_1.$$
In order to show that also A2 holds true we need to investigate the behaviour of $x_i b^i_\pi (x)$. From the second point of A3, which holds true for any $x_i$ such that $|x_i| > \frac{{\revar R_0}}{\sqrt{d}}$, it is 
$$x_i b^i_\pi (x) = \frac{x_i}{2} \frac{\pi'_i(x_i)}{\pi_i (x_i)} \le - \frac{x_i}{2} \tilde{\epsilon} \sgn(x_i) = - \frac{\tilde{\epsilon}}{2} |x_i|.$$
Using also the boundedness of $b^i_\pi$ showed before, it follows
$$x \cdot b_\pi(x) = \sum_{i = 1}^d x_i b^i_\pi (x) = \sum_{x_i : |x_i| > \frac{{\revar R_0}}{\sqrt{d}}}^d x_i b^i_\pi (x) + \sum_{x_i : |x_i| \le \frac{{\revar R_0}}{\sqrt{d}}}^d x_i b^i_\pi (x) $$
$$\le - {\revar\tilde{\epsilon}} \sum_{x_i : |x_i| > \frac{{\revar R_0}}{\sqrt{d}}}^d | x_i | + {\revar   d \times \frac{{\revar R_0}}{\sqrt{d}} \times \tilde{b}_0 }\le
{\revar -\frac{\tilde{\epsilon}}{\sqrt{d}} |x| + {R_0} \sqrt{d} \tilde{b}_0}
$$
where the last inequality is a consequence of the fact that, for $|x| > {\revar R_0}$, there has to be at least a component $x_i$ such that $|x_i| > \frac{R}{\sqrt{d}}$. Hence, we can use the sup norm and compare it with the euclidean one. 
{\revar Now, if $|x| \ge \tilde{\rho}:= 2 R_0 d \tilde{b}_0 / \tilde{\epsilon}$, we have
	\begin{equation*}
		x \cdot b_\pi(x) \le 	-\frac{\tilde{\epsilon}}{\sqrt{d}} |x| + {R_0} \sqrt{d} \tilde{b}_0 \le
		-\frac{\tilde{\epsilon}}{2\sqrt{d}} |x|. 
	\end{equation*}
}	
The proposed drift $b_\pi$ is therefore a bounded lipschitz function that satisfies A2, up to know that on the linked invariant density the properties gathered in A3 hold true. In the next step we propose two priors with the prescribed properties. \\
\\
$\bullet$ Construction of the priors. \\
We want to provide two drift functions belonging to {\revar $\tilde\Sigma(\beta, \mathcal{L},1,b_0,a_0,a_1,b_1,\tilde{C},\tilde{\rho})$ where $b_0>0$, $a_0>1$, $a_1>0$, $b_1>0$, and $\tilde{C},\tilde{\rho}$ will be specified later.} 
To do it, we introduce two probability densities defined on the purpose to make A3 hold true. We set $\pi^{(0)}(x) := c_\eta \prod_{k = 1}^d \pi_{k,0} (x_k),$ where $c_\eta$ is the constant that makes $\pi^{(0)}$ a probability measure. For any $y \in \mathbb{R}$ we define $\pi_{k,0} (y): = f(\eta |y|)$, where
\begin{equation*}
f (x) := \begin{cases} e^{- |x| } \qquad & \mbox{if } |x| \ge 1  \\
\in[1, e^{-1}] \qquad & \mbox{if } \frac{1}{2} < |x| < 1 \\
1 \qquad & \mbox{if } |x| \le \frac{1}{2}
\end{cases}
\end{equation*}
 and $\eta$ is a constant in $(0, \frac{1}{2})$ which plays the same role as $\tilde{\epsilon}$ did in A3, as it can be chosen as small as we want. In particular we choose $\eta$ small enough to get $\pi^{(0)} \in \mathcal{H}_d (\beta, \mathcal{L})$.
Moreover, we assume $f$ to be a $C^\infty$ function that satisfies $|f^{(k)}(x)| \le 2 e^{- |x|}$ for $k=1, 2$. \\
It is easy to see that $\pi_0$ satisfies A3, as it is clearly positive and in a multiplicative form. The first point holds true by construction. We observe then that, for $y$ such that $|y| > \frac{1}{\eta }$, we have $\pi_{j,0}'(y)= - \eta \sgn(y) \pi_{j,0}(y)$ for any $j \in \{ 1, ... , d \}$. Therefore, point 2 of A3 is satisfied for $y$ such that $|y| > \frac{{\revar R_0}}{\sqrt{d}}$, up to take {\revar $R_0:= \frac{\sqrt{d}}{\eta}$ and $\tilde{\epsilon}=\eta$.}\\
Regarding point 3 of A3, it clearly holds true for $y$ such that $|y| > \frac{1}{\eta }$ and $|y| < \frac{1}{2 \eta }$ for what said before and as the derivative is zero, respectively. When $\frac{1}{2 \eta } \le |y| \le  \frac{1}{\eta}$, point 3 of A3 is satisfied thanks to the condition $|f'(x)| \le 2 e^{- |x|}$
{\revar with the possible choice $c_1=2e \eta$.} An analogous reasoning can be applied to ensure the validity of the fourth point of A3 {\revar with $c_2=2e \eta^2$.} \\
\\
{\revar Choosing $\eta>0$ small enough, $\pi^{(0)}$ belongs to $\mathcal{H}_d (\beta,  \mathcal{L})\subset\mathcal{H}_d (\beta, 2 \mathcal{L})$ and is satisfying A3 with arbitrarily small constants $\tilde{\epsilon}$, $c_1$, $c_2$. Thus, we can choose $\eta$ such that the associated coefficients 
	satisfy 
	$(I_{d \times d}, b^{(0)}) \in {\tilde{\Sigma}}(\beta, \mathcal{L}, 1,\Tilde{b}_0, a_0, a_1, \tilde{b}_1, {\revar \tilde{C}, \tilde{\rho})}$ with $\tilde{b}_0 < b_0$, $\tilde{b}_1<b_1$ and
	$\tilde{C}=\eta/(2\sqrt{d})$, $\tilde{\rho}= \frac{2 R_0 d \tilde{b}_0}{\tilde{\epsilon}}=\frac{2 d^{3/2} \tilde{b}_0}{\eta^2}$ 
	are some constants depending on $\eta$.
	We deduce that  $ (I_{d \times d}, b^{(0)}) \in  \tilde{\Sigma}=\tilde{\Sigma}(\beta, \mathcal{L}, 1,b_0, a_0, a_1, b_1, \tilde{C}, \tilde{\rho})$.}
	 To provide the second hypothesis, we introduce the probability measure $\pi^{(1)}$. We are given it as $\pi^{(0)}$ to which we add a bump: let $K: \mathbb{R} \rightarrow \mathbb{R}$ be a $C^\infty$ function with support on $[-1, 1]$ and such that
\begin{equation}
K(0)= 1, \qquad \int_{-1}^1 K(z) dz = 0.
\label{eq: proprieta K}
\end{equation}
We set
\begin{equation}
\pi^{(1)} (x) := \pi^{(0)} (x) + \frac{1}{M_T} \prod_{l = 1}^d K(\frac{x_l - x_0^l}{h_l(T)}),
\label{eq: def pi1}
\end{equation}
where $x_0 = (x_0^1, ... , x_0^d) \in \mathbb{R}^d$ is the point in which we are evaluating the minimax risk, as defined in \eqref{eq: def minimax risk}, $M_T$ and $h_l(T)$ will be calibrated later and satisfy $M_T \rightarrow \infty$ and, $\forall l \in \left \{ 1, ... , d \right \}$, $h_l(T) \rightarrow 0$ as $T \rightarrow \infty$. From the properties of the kernel function given in \eqref{eq: proprieta K} we obtain
$$\int_{\mathbb{R}^d} \pi^{(1)}(x) dx = \int_{\mathbb{R}^d} \pi^{(0)}(x) dx = 1.$$
Moreover, as $\pi^{(0)} > 0$, $K$ has {\rev compact support} and $\frac{1}{M_T} \rightarrow 0$, for T big enough we can say that $\pi^{(1)} > 0$ as well. The key point consists in calibrating $M_T$ and $h_l(T)$ such that both the densities $\pi^{(0)}$ and $\pi^{(1)}$ belong to the anisotropic Holder class $\mathcal{H}_d (\beta, 2 \mathcal{L})$ (according with Definition \ref{def: insieme sigma v2}) and the laws $\mathbb{P}^{(0)}$ and $\mathbb{P}^{(1)}$ are close. \\
To do that, we need to evaluate the difference between the two proposed drifts. We introduce the following set of $\mathbb{R}^d$:
$$K_T := [x_0^1 - h_1 (T), x_0^1 + h_1 (T)] \times ... \times [x_0^d - h_d (T), x_0^d + h_d (T)].$$
Then, for $T$ large enough, 
\begin{enumerate}
    \item For any $x \in K_T^c$ and $\forall i \in \left \{ 1, ... , d \right \}$: $b^{(1)}_i(x) = b^{(0)}_i(x)$.
    \item For any $x \in K_T$ and $\forall i \in \left \{ 1, ... , d \right \}$: $|b^{(1)}_i(x) - b^{(0)}_i (x)| \le \frac{c}{M_T}\sum_{j = 1}^d \frac{1}{h_j(T)}$, where $c$ is a constant independent of $T$. 
\end{enumerate}
Before proceeding with the proof of these two points we introduce some notations: 
{\change
\begin{equation}\label{E: def bump sans log}
I^i[\pi ] (x) := \frac{1}{2} \frac{\partial \pi}{ \partial x_i} (x), \qquad  d_T (x) := \pi^{(1)} (x) - \pi^{(0)} (x) = \frac{1}{M_T} \prod_{l = 1}^d K(\frac{x_l - x_0^l}{h_l(T)}).
\end{equation}
}
As the support of $K$ is in $[-1, 1]$, for any $x\in K_T^c$ both $d_T(x)$ and its derivatives are $0$ and, in particular, $\pi^{(0)}(x) = \pi^{(1)}(x)$. We can therefore write, using the linearity of the operator $I^i$, 
$$b^{(1)}_i (x) =  \frac{1}{\pi^{(0)} (x)}\tilde{I}^i[\pi^{(0)} ] (x) + \frac{1}{\pi^{(0)} (x)}\tilde{I}^i[d_T ] (x) = b^{(0)}_i (x) + \frac{1}{\pi^{(0)} (x)}\tilde{I}^i[d_T ] (x)= b^{(0)}_i (x).$$
Regarding the second point here above, we observe that on $K_T$ we have
$$b^{(1)}_i - b^{(0)}_i = (\frac{1}{\pi^{(1)}} - \frac{1}{\pi^{(0)}}) \tilde{I}^i[\pi^{(0)}] + \frac{1}{\pi^{(1)}} \tilde{I}^i[d_T] = \frac{\pi^{(0)} - \pi^{(1)}}{\pi^{(1)}} \frac{1}{\pi^{(0)}} \tilde{I}^i[\pi^{(0)}] + \frac{1}{\pi^{(1)}} \tilde{I}^i[d_T] = \frac{d_T}{\pi^{(1)}} b^{(0)}_i + \frac{1}{\pi^{(1)}} \tilde{I}^i[d_T]. $$
 For how we have defined $\pi^{(1)} = \pi^{(0)} + d_T$, we see first of all it is lower bounded away from $0$. Moreover we know that $\pi^{(0)}$ satisfies Assumption A3 and so $b^{(0)}_i$ is bounded. Furthermore, we have the following controls on $d_T$:
 $$\left \| d_T \right \|_\infty \le \frac{c}{M_T}, \qquad \left \| \frac{\partial d_T}{\partial x_j} \right \|_\infty \le \frac{c}{M_T} \frac{1}{h_j (T)}.$$ It follows that, for any $x \in K_T$, 
$$|b^{(1)}_i - b^{(0)}_i| \le \frac{c}{M_T}(1 + \sum_{j= 1}^d \frac{1}{h_j (T)}) \le \frac{c}{M_T} \sum_{j= 1}^d \frac{1}{h_j (T)}, $$
where the last inequality is a consequence of the fact that, $\forall j \in \left \{ 1, ... , d \right \}$, $h_j(T) \rightarrow 0$ for $T \rightarrow \infty$ and so, if compared with the second term in the equation here above, all the other terms are negligible.\\
\\
Then, it is possible to show that also $(I_{d \times d}, b^{(1)})$ belongs to ${\modarn \tilde{\Sigma}}$, up to calibrate properly $M_T$ and $h_i(T)$, for $i \in \left \{ 1, ... , d \right \}$. 
We recall that we already know that $b^{(0)}$ satisfies A1-A2. Due to points 1 and 2 above also $b^{(1)}$ is bounded, up to ask that $\sum_{j = 1}^d \frac{1}{h_j(T)} = o(M_T)$. 
Remark also that we have $<b^{(1)}(x),x>\le - \tilde{C} |x|$ for $|x| \ge \tilde{\rho}$ for $T$ large enough, using that $b^{(1)}$ and $b^{(0)}$ coincide on $K_T^c$. Moreover, after some computations we have 
\begin{align*}
& \norm{\nabla b^{(1)}-\nabla b^{(0)}}_\infty \le \frac{c}{M_T} \Big[ \norm{d_T}_\infty+
\sum_{y \in \{x_1,\dots,x_d\}} \norm{d_T}_\infty\norm{\frac{\partial d_T}{\partial y}}_\infty 
\\
& +\frac{1}{M_T}
\sum_{y,y' \in \{x_1,\dots,x_d\}}
\norm{\frac{\partial d_T}{\partial y}}_\infty
\norm{\frac{\partial d_T}{\partial y'}}_\infty+
\sum_{y,y' \in \{x_1,\dots,x_d\}}
\norm{\frac{\partial^2 d_T}{\partial y \partial y'}}_\infty
\Big] \\
& \le \frac{1}{M_T} \sum_{i,j = 1}^d \frac{c}{h_i h_j}
\le \frac{1}{M_T} \frac{c}{h_1^2}.
\end{align*}
Hence, to get that $\norm{\nabla b^{(1)}} \le b_1$ it is sufficient that
\begin{equation}
\frac{c}{M_T} \frac{1}{h_1^2} \rightarrow 0
\label{eq: cond beta>2}
\end{equation}
for $T$ going to $\infty$.
Furthermore, requiring that
\begin{equation}
\frac{1}{M_T} \le \epsilon h_i(T)^{\beta_i} \qquad \forall i \in \left \{ 1, ... , d \right \},
\label{eq: cond Holder}
\end{equation}
it is easy to derive the Holder regularity of $\pi^{(1)}$ starting from the regularity of $\pi^{(0)}$, as proved for example in Lemma 3 of \cite{Lower bound}. It follows that, under condition \eqref{eq: cond Holder}, (which implies also $\sum_{j = 1}^d \frac{1}{h_j(T)} = o(M_T)$, as $\beta_j > 1$ for any $j$) and \eqref{eq: cond beta>2}, both $(I_{d \times d}, b^{(0)})$ and $(I_{d \times d},b^{(1)})$ belong to
${\revar \tilde{\Sigma}}$. \\
\\
$\bullet$ Choice of the calibration \\
{\modch Before we keep proceeding, we introduce some notations. We denote as $\mathbb{P}_0$ (respectively $\mathbb{P}_1$) the law of a stationary solution $(X_t)_{t \ge 0}$ of \eqref{eq: model lower bound} whose drift coefficient is $b^{(0)}$ (respectively $b^{(1)}$). Moreover we will note $\mathbb{P}_0^{(T)}$ the law of $(X_t)_{t \in [0, T]}$, solution of the same stochastic differential equation as here above. The corresponding expectation will be denoted as $\mathbb{E}_{(I_{d \times d}, b^{(0)})}^{(T)}$ (respectively $\mathbb{E}_{(I_{d \times d}, b^{(1)})}^{(T)}$).} \\
To find a lower bound for the risk we will need to use that there exist $C$ and $\lambda > 0$ such that, for all $T$ large enough,
\begin{equation}
\mathbb{P}^{(T)}_{0} ( Z^{(T)} \ge \frac{1}{\lambda}) \ge C,
\label{eq: Girsanov}
\end{equation}
where we have introduced the notation $Z^{(T)} := \frac{d \mathbb{P}^{(T)}_{1}}{d \mathbb{P}^{(T)}_{0}}$.
To ensure its validity it is enough to remark that the proof of Lemma 4 in \cite{Lower bound} is the same even in absence of jumps. Hence, we know \eqref{eq: Girsanov} holds true if 
$$ \sup_{T \ge 0} T \int_{\mathbb{R}^d} |b^{(1)} (x) - b^{(0)}(x)|^2 \pi^{(0)}(x) \, dx < \infty.$$
From points 1 and 2 above, it is equivalent to ask 
\begin{equation}
\sup_{T \ge 0} T \frac{c}{M_T^2} (\sum_{j = 1}^d \frac{1}{h_j^2(T)})  |K_T| = \sup_{T \ge 0} T \frac{c}{M_T^2} (\prod_{l = 1}^d h_l(T)) (\sum_{j = 1}^d \frac{1}{h_j^2(T)}) < \infty.
\label{eq: calib finale}
\end{equation}
Then, as $(I_{d \times d}, b^{(0)})$ and $(I_{d \times d},b^{(1)})$ belong to $ \Sigma$, we have 
\begin{align*}
R(\tilde{\pi}_T (x_0)) & \ge \frac{1}{2} \mathbb{E}_{(I_{d \times d}, b^{(1)})}^{(T)}[(\tilde{\pi}_T (x_0) - \pi^{(1)} (x_0))^2] + \frac{1}{2} \mathbb{E}_{(I_{d \times d}, b^{(0)})}^{(T)}[(\tilde{\pi}_T (x_0) - \pi^{(0)} (x_0))^2] \\
& \ge  \frac{1}{2} \mathbb{E}_{(I_{d \times d}, b^{(0)})}^{(T)}[(\tilde{\pi}_T (x_0) - \pi^{(1)} (x_0))^2 Z^{(T)}] + \frac{1}{2} \mathbb{E}_{(I_{d \times d}, b^{(0)})}^{(T)}[(\tilde{\pi}_T (x_0) - \pi^{(0)} (x_0))^2] \\
& \ge \frac{1}{2 \lambda} \mathbb{E}_{(I_{d \times d}, b^{(0)})}^{(T)}[(\tilde{\pi}_T (x_0) - \pi^{(1)} (x_0))^2 1_{\left \{ Z^{(T)} \ge \frac{1}{\lambda} \right \}}] + \frac{1}{2} \mathbb{E}_{(I_{d \times d}, b^{(0)})}^{(T)}[(\tilde{\pi}_T (x_0) - \pi^{(0)} (x_0))^2 1_{\left \{ Z^{(T)} \ge \frac{1}{\lambda} \right \}}] \\
&= \frac{1}{2 \lambda} \mathbb{E}_{(I_{d \times d}, b^{(0)})}^{(T)}\big[[(\tilde{\pi}_T (x_0) - \pi^{(1)} (x_0))^2+ (\tilde{\pi}_T (x_0) - \pi^{(0)} (x_0))^2] 1_{\left \{ Z^{(T)} \ge \frac{1}{\lambda} \right \}}\big],
\end{align*}
for all $\lambda> 1$. We remark it is 
$$(\tilde{\pi}_T (x_0) - \pi^{(1)} (x_0))^2+ (\tilde{\pi}_T (x_0) - \pi^{(0)} (x_0))^2 \ge (\frac{\pi^{(1)} (x_0) - \pi^{(0)}(x_0)}{2})^2$$
and so we obtain 
\begin{equation}
R(\tilde{\pi}_T (x_0)) \ge \frac{1}{8 \lambda}(\pi^{(1)} (x_0) - \pi^{(0)} (x_0))^2 \mathbb{P}_{0}^{(T)}(Z^{(T)} \ge \frac{1}{\lambda}) \ge \frac{c}{M_T^2},
\label{E: lower bound risk in proof wo log}
\end{equation}
where we have used \eqref{eq: Girsanov} and that, by construction,
$\pi^{(1)} (x_0) - \pi^{(0)} (x_0) = \frac{1}{M_T} \prod_{l = 1}^d K(0) = \frac{1}{M_T}$.
Hence, we have to find the largest choice for $\frac{1}{M_T^2}$, subject to the constraints \eqref{eq: cond Holder} and \eqref{eq: calib finale}. 
To do that, we suppose at the beginning to saturate \eqref{eq: cond Holder} for any $j \in \left \{ 1, ... , d \right \}$. From the order of $\beta$ we obtain 
\begin{equation}
h_1 (T) = h_2 (T)^{\frac{\beta_2}{\beta_1}} \le h_2(T) \le ... \le h_d(T). 
\label{eq: order h}
\end{equation}
We plug it in \eqref{eq: calib finale} and we observe that the biggest term in the sum is $ \frac{\prod_{l \neq 1} h_l(T)}{h_1(T)}$.
In order to make it as small as possible, we decide to increment $h_1(T)$ up to get $h_1(T) = h_2(T)$, remarking that it is not an improvement to take $h_1(T)$ also bigger than $h_2(T)$ because otherwise $ \frac{\prod_{l \neq 2} h_l(T)}{h_2(T)}$ would be the biggest term, and it would be larger than $ \frac{\prod_{l \neq 1} h_l(T)}{h_1(T)}$ for $h_1(T) = h_2(T)$. Therefore, we take $h_1(T) = h_2(T)$ and $h_l(T)= (\frac{1}{M_T})^{\frac{1}{\beta_l}}$ for $l \ge 2$. With this choice \eqref{eq: cond beta>2} is always satisfied as $\beta_2 > 2$. Moreover, we have
$$\prod_{l \ge 3} h_l(T) = (\frac{1}{M_T})^{\sum_{l \ge 3} \frac{1}{\beta_l}} =(\frac{1}{M_T})^{\frac{d-2}{ \barfix{\bar{\beta}_3}}}.$$
And so condition \eqref{eq: calib finale} turns out being
$$\sup_{T} T \frac{1}{M_T^2} \prod_{l \ge 3} h_l(T) = \sup_{T} T \frac{1}{M_T^2} (\frac{1}{M_T})^{\frac{d-2}{ \barfix{\bar{\beta}_3}}} \le c. $$
It leads us to the choice $M_T = T^{\frac{ \barfix{\bar{\beta}_3} }{2  \barfix{\bar{\beta}_3}  + d - 2  }}.$
It implies
$$R(\tilde{\pi}_T (x_0)) \ge  (\frac{1}{T})^{\frac{2 \,  \barfix{\bar{\beta}_3}}{2 \barfix{\bar{\beta}_3}  + d - 2  }},$$
as we wanted.

\end{proof}

 \subsection{Proof of Theorem \ref{th: lower bound avec log}}
  \begin{proof} 
 

  Following the same ideas as in the proof of Theorem \ref{th: lower bound} we construct a prior supported by two different models corresponding to two drift functions $b^{(0)}$ and $b^{(1)}$, and {\modar with} associated stationary {\modar probabilities} $\pi^{(0)}$ and $\pi^{(1)}$. {\change However, the shape of the bump used in the construction of $\pi^{(1)}$, given below by  \eqref{E: def bump log}, is different from the one used in the proof of Theorem \ref{th: lower bound} and defined in \eqref{E: def bump sans log}}.
 
 In the construction of the prior, the first two components are treated differently than the other ones. For this reason, we introduce cylindrical coordinates $(r,\theta,x_3,\dots,x_d)$ that translates into the Cartesian coordinates 
 $(r \cos(\theta),r \sin(\theta),x_3,\dots,x_d)$. 
 We recall 
  the expression of the Laplacian in cylindrical coordinates
 $\Delta \pi = \frac{1}{r} \frac{\partial}{ \partial r}\left( r \frac{\partial \pi}{\partial r} \right)
 + \frac{1}{r^2} \frac{\partial^2 \pi}{\partial \theta^2} + \sum_{k=3}^d
 \frac{\partial^2 \pi}{\partial x_k^2}$ and of the divergence operator
 $ \nabla \cdot \xi = \frac{1}{r} \frac{\partial(r \xi_r)}{\partial r} + 
 \frac{1}{r} \frac{\partial \xi_\theta}{\partial \theta} + \sum_{k=3}^d \frac{\partial \xi_k}{ \partial x_k}$
  where $\xi$ is the vector field $\xi=\xi_r \vec{e}_r + \xi_\theta \vec{e}_\theta + \sum_{k=3}^d \xi_k \vec{e}_k$ with
 $\vec{e}_r=(\cos(\theta), \sin(\theta),0,\dots,0)^T$, $\vec{e}_\theta=(-\sin(\theta), \cos(\theta),0,\dots,0)^T$, and $(\vec{e}_k)_{k=1,\dots,d}$ is the canonical Cartesian basis of $\mathbb{R}^d$.
 For $\pi$ smooth stationary probability of the diffusion $ dX_t=b(X_t)dt + d W_t$ the condition $A^*_b \pi=0$ can be written as 
 {\change 
 \begin{equation}\label{E: grad bpi eg delta pi}
 \frac{1}{2}\Delta \pi - \nabla \cdot (\pi b)=0
 \end{equation}
 }
 which yields
 in cylindrical coordinates to
 \begin{equation} \label{E:Astar pi zero cyl}
 	\frac{1}{r} \frac{\partial}{\partial r} \left( r \frac{\partial \pi}{\partial r} \right)
 	+ \frac{1}{r^2} \frac{\partial^2 \pi}{\partial \theta^2}
 	+ \sum_{k=3}^d \frac{\partial^2 \pi}{\partial x_k^2}
= 	\frac{1}{r} \frac{\partial \left(r \pi b_r \right)}{\partial r} +
\frac{1}{r}\frac{\partial (\pi b_\theta)}{\partial \theta}  + \sum_{k=3}^d
 \frac{\partial (\pi b_k)}{\partial x_k}. 
\end{equation}
 Given a stationary probability $\pi=\pi(r,x_3,\dots,x_d)>0$ independent of $\theta$, we see that the drift $b$ solution of \eqref{E:Astar pi zero cyl} is given by 
 $b=b_r \vec{e}_r+ b_\theta \vec{e}_\theta+ \sum_{k=3}^d b_k \vec{e}_k $ with
\begin{equation}\label{E: lien b pi cyl}
	 \begin{split}
 	b_r=\frac{1}{\pi} \frac{\partial \pi}{\partial r}, \quad b_\theta=0,\\
 	b_k=\frac{1}{\pi} \frac{\partial \pi}{\partial x_k}, ~ \forall k \ge 3.
 	\end{split}
 \end{equation}
$\bullet$ Construction of the {\modar priors.} As in the proof of Theorem \ref{th: lower bound}, the prior is based on two points $(\pi^{\modar (0)},\pi^{(T)})$. First, we define $\pi^{\modar (0)}$. Let $\psi: [0,\infty) \to [0,\infty)$ be a smooth function, vanishing on $[0,1/2]$ and satisfying $\psi(x)=x$ for $x \ge 1$. We let
$$
\pi^{(0)}(r,x_3,\dots,x_d)=c_\eta e^{-\eta \psi(r)} \prod_{k=3}^d e^{-\eta \psi(|x_k|)},
$$ 
 where $\eta>0$ and $c_\eta$ is such that $\int_{[0,\infty)\times[0,2\pi)\times\mathbb{R}^{d-2}} \pi^{(0)}(r,x_3,\dots,x_d) r ~drd\theta dx_3\dots dx_d=1$. Remark that for $\eta \to0$, we have $c_\eta=O(\eta^d)$. Using that $\psi$ has bounded derivatives and vanishes near $0$ one can check that $(x_1,\dots,x_d) \mapsto \pi^{(0)}(\sqrt{x_1^2+x_2^2},x_3,\dots,x_d)$ is a smooth function and
 that $\norm{\frac{\partial^l}{\partial x_k^l} \pi^{(0)}(\sqrt{x_1^2+x_2^2},x_3,\dots,x_d) }_\infty =O(c_\eta)=O(\eta^d)$ for all $1\le k \le d$ and $l\ge 0$. Hence, we can choose $\eta$ small enough such that $\pi^{(0)} \in \mathcal{H}_d(\beta,\mathcal{L})$. The drift function $b^{(0)}$ associated to $\pi^{(0)}$ given by \eqref{E: lien b pi cyl} is such that $b^{(0)}_r(r,x_3,\dots,x_d)=-\eta\psi'(r)=-\eta$ if $r>1$ and
 $b^{(0)}_k(r,x_3,\dots,x_d)= -\eta \sgn(x_k)$ for $|x_k|>1$. By computation analogous to the ones before, there exists $\tilde{C}>0$ and $\tilde{\rho}>0$ such that $<b^{(0)}(x),x>\le - \tilde{C} |x|$ for $x \ge \tilde{\rho}$. Moreover, $\norm{b^{(0)}}_{\infty}+\norm{\nabla b^{(0)}}_{\infty}=O(\eta)$. Hence, if $\eta$ is chosen small enough, $( \mathbb{I}_{d \times d},b^{(0)}) \in {\modarn\tilde\Sigma} (\beta, \mathcal{L},1,b_0,a_0,a_1,b_1/2,\tilde{C},\tilde{\rho})\subset {\modarn\tilde\Sigma} (\beta, \mathcal{L},1,b_0,a_0,a_1,b_1,\tilde{C},\tilde{\rho})$.

 The construction of {\modarn $\pi^{(1)}$} is more elaborate. We add a bump centered at $0$ to $\pi^{(0)}$. Let $K:\mathbb{R}\to\mathbb{R}$ be a smooth function with support on $[-1,1]$ and satisfying \eqref{eq: proprieta K}.   We set $\pi^{(1)}=\pi^{(0)}+\frac{1}{M_T} d_T$ where 
 \begin{equation} \label{E: def bump log}
 	d_T(r,x_3,\dots,x_d)= J_{r_\text{min}(T),r_\text{max}(T)}(r) \times \prod_{i=3}^d K(\frac{x_i}{h_i(T)}),
  \end{equation}
 where $0<r_\text{min}(T)<r_\text{max}(T)/4 < r_\text{max}(T) \le h_3(T) \le \dots  \le h_d(T)$ will be calibrated  later and goes to zero with a rate polynomial in $1/T$. 
 In the following we will suppress the dependence {\modarn on} $T$ of $r_\text{min}, r_\text{max}, (h_j)_{3 \le j \le d}$ in order to lighten the notations.  
 The function $J_{r_\text{min},r_\text{max}} : [0,\infty) \to \mathbb{R}$ is a smooth function with support on $[0,r_\text{max}]$ satisfying the following properties:
 \begin{align}  \label{E: cond J 0}
 	&J_{r_\text{min},r_\text{max}}
 	\text{is constant on $[0,r_\text{min}/4]$ and decreasing on $[r_\text{min}/4,r_{\text{max}}]$ with  $J_{r_\text{min},r_\text{max}}(0)\ge 1$}
 	\\ 	
 	\label{E: cond J 1}
 	&
 	0\le J_{r_\text{min},r_\text{max}}(r) \le c \left( \frac{\ln(r_\text{max}/r)+1}{\ln(r_\text{max}/r_\text{min})}\wedge 1 	\right) ,\\
 	&\abs{\frac{\partial J_{r_\text{min},r_\text{max}}(r)}{\partial r}}\le 
 	\frac{c}{(\ln(r_\text{max}/r_\text{min}))(r \wedge r_\text{min})} ,
 	\label{E: cond J 2}
 	\\
 	&  \abs{\frac{\partial^k J_{r_\text{min},r_\text{max}}(r)}{\partial r^k}}\le 
 	\frac{c(k)}{(\ln(r_\text{max}/r_\text{min}))r_\text{min}^k}, ~ \forall k \ge 1
 	{\modar ,}
 	\label{E: cond J 3}
 \end{align}
 {\modar where the constants $c$ and $c(k)$ are independent of $r_{\text{min}}$, $r_{\text{max}}$.}
 {\change
 The existence of such function $J_{r_\text{min},r_\text{max}}$ is postponed to Lemma \ref{L: existence J},  and we give in Remark \ref{rem : bosse log} some hints underneath its construction}. 
 Remark that $\pi^{(1)}$ is a probability measure for $T$ large enough, as the mean of $d_T$ is zero and ${\modar \pi^{(1)}}$ is positive if $T$ is large enough, as we assume $1/M_T \to 0$.
 
 We denote by $b^{(1)}$ the drift function associated to $\pi^{(1)}$ through the relations 
 \eqref{E: lien b pi cyl}.
  We now prove the following lemma.
 \begin{lemma} We have,
 	\begin{equation}\label{E: majo L2 diff b log}
 	\int_{\mathbb{R}^d} \abs{b^{(1)}(\sqrt{x_1^2+x_2^2},x_3,\dots,x_d)-b^{(0)}(\sqrt{x_1^2+x_2^2},x_3,\dots,x_d)}^2 dx_1\dots dx_d\le c
 	\frac{\prod_{i=3}^d h_i}{M_T^2\ln(\frac{r_\text{max}}{r_\text{min}})},
 	\end{equation}
 \begin{align}\label{E: majo Linfint diff b log}
 	\norm{b^{(1)}-b^{(0)}}_\infty & \le \frac{c}{M_T} \left[ \frac{1}{r_\text{min} \ln(\frac{r_\text{max}}{r_\text{max}})} + \frac{1}{h_3} \right],
 	\\ \label{E: majo Linfint diff grad b log}
 	\norm{\nabla b^{(1)}-\nabla b^{(0)}}_\infty &\le \frac{c}{M_T} \left[ \frac{1}{r_\text{min}^2 \ln(\frac{r_\text{max}}{r_\text{max}})} + \frac{1}{h_3^2} \right].
 \end{align}
 \label{L: maj norme L2 diff b log}
 \end{lemma}
 \begin{proof} 
 	 	$\bullet$ We start with the proof of \eqref{E: majo L2 diff b log}.  {\modar Recalling $\pi^{(1)}=\pi^{(0)}+\frac{d_T}{M_T}$ and using \eqref{E: lien b pi cyl}, we deduce
 		\begin{equation}\label{E:diff_b1_b0}
 			b^{(1)}-b^{(0)}= \frac{d_T}{M_T \pi^{(1)}}
 			b^{(0)}	+
 			\frac{1}{M_T \pi^{(1)}} \sum_{k=3}^d \frac{\partial d_T}{\partial x_k} \vec{e}_k+ \frac{1}{M_T \pi^{(1)}}	\frac{\partial d_T}{\partial r}\vec{e}_r.
 		\end{equation}
 	} 	
 	In the evaluation of the $\mathbf{L}^2$ norm of $b^{(1)}-b^{(0)}$, we start by the  contribution of the radial component.
 	{\modar From \eqref{E:diff_b1_b0}, we have
 		\begin{equation*} 
 			b^{(1)}_r-b^{(0)}_r 
 			=	
 			b^{(0)}_r \left(\frac{d_T}{M_T \pi^{(1)}}  \right)
 			+ \frac{1}{M_T \pi^{(1)}} \frac{\partial d_T}{\partial r} .
 	\end{equation*}}
Remarking that $\pi^{(1)}$ and $\pi^{(0)}$ coincides out of the compact set $K_T=\{r \le r_{\text{max}}\} \times [0,2\pi] \times \prod_{i=3}^d[-h_i,h_i]$, we deduce that $b^{(1)}$ and $b^{(0)}$ are equal on $K_T^c$.
Thus,
\begin{align}\nonumber
	\norm{b_r^{(1)}-b_r^{(0)}}_2^2&=\int_{K_T} |b_r^{(1)}(r,x_3,\dots,x_d)-b_r^{(0)}(r,x_3,\dots,x_d)|^2 rdr d\theta dx_3\dots dx_d
	\\ \nonumber 
	&\le \frac{c}{M_T^2}\int_{K_T} \left[ \abs{d_T(r,x_3,\dots,x_d)}^2+ \abs{\frac{\partial d_T(r,x_3,\dots,x_d)}{\partial r}}^2 \right] rdr d\theta dx_3\dots dx_d
	\\  \label{E: decoupe norme b radial}
	&=:\frac{c}{M_T^2}(I_1(T)+I_2(T)),
\end{align}
where {\modar we} used that $\pi^{(1)}$ is lower bounded independently of $T$ on the compact set $K_T$, for $T$ large enough, and the boundedness of $b^{(0)}$.
Using the definition of $d_T$ given by \eqref{E: def bump log}, and \eqref{E: cond J 1}, we have
\begin{align}\nonumber
	I_1(T) &= \int_0^{r_{\text{max}}} |J_{r_\text{min},r_\text{max}}(r)|^2 rdr \int_{\prod_{i=3}^d[-h_i,h_i]} 
	\prod_{i=3}^d K(\frac{x_i}{h_i})^2
	dx_3\dots dx_d
	\\ \nonumber
	&\le c \norm{K}_\infty^{2(d-3)} \int_0^{r_{\text{max}}}\left( \frac{\ln(r_\text{max}/r)+1}{\ln(r_\text{max}/r_\text{min})}\wedge 1 
	\right)^2 rdr \left( \prod_{i=3}^d h_i \right) 
	\\ \nonumber
	&\le c \norm{K}_\infty^{2(d-3)}\left(\prod_{i=3}^d h_i\right) \left[ \int_0^{r_{\text{max}}/\sqrt{\ln(r_\text{max}/r_\text{min})}}
	rdr
	+
	\int_{r_{\text{max}}/ \sqrt{\ln(r_\text{max}/r_\text{min})}}^{r_\text{max}} 
	\left(\frac{\ln(r_\text{max}/r )+1}{\ln(r_\text{max}/r_\text{min})}\right)^2
	rdr
	\right]
	\\ \nonumber
	&
	\le c \norm{K}_\infty^{2(d-3)}\left(\prod_{i=3}^d h_i\right)\left[
	\frac{r_\text{max}^2}{\ln(r_\text{max}/r_\text{min})}
	+
	\int_{r_{\text{max}}/ \sqrt{\ln(r_\text{max}/r_\text{min})}}^{r_\text{max}} 
	\left(\frac{\ln(\sqrt{\ln(r_\text{max}/r_\text{min})})+1}{\ln(r_\text{max}/r_\text{min})}\right)^2
	rdr
	\right]
	\\
	& \nonumber
	\le c \norm{K}_\infty^{2(d-3)}\left(\prod_{i=3}^d h_i\right)
	\left[
	\frac{r_\text{max}^2}{\ln(r_\text{max}/r_\text{min})}
	+
	\modar{ \int_{r_{\text{max}}/ \sqrt{\ln(r_\text{max}/r_\text{min})}}^{r_\text{max}} 
		\frac{c}{\ln(r_\text{max}/r_\text{min})}
		rdr}
	\right]
	\\
	& \label{E: controle I1 preuve borne inf log}
	\le c \norm{K}_\infty^{2(d-3)}\left(\prod_{i=3}^d h_i\right)
	\frac{r_\text{max}^2}{\ln(r_\text{max}/r_\text{min})}.
\end{align}
Using now \eqref{E: cond J 2},
\begin{align} \nonumber
	I_2(T) &= \int_0^{r_{\text{max}}} |\frac{\partial J_{r_\text{min},r_\text{max}}(r)}{\partial r}|^2 rdr \int_{\prod_{i=3}^d[-h_i,h_i]} 
	\prod_{i=3}^d K(\frac{x_l}{h_l(T)})^2
	dx_3\dots dx_d
	\\ \nonumber
	&\le c \norm{K}_\infty^{2(d-3)} \int_0^{r_{\text{max}}} \left( \frac{1}{(\ln(r_\text{max}/r_\text{min}))(r \wedge r_\text{min})} \right)^2 r dr \left(\prod_{i=3}^d h_i\right)
	\\ \nonumber
	&
	\le c \norm{K}_\infty^{2(d-3)} \left(\prod_{i=3}^d h_i\right)\frac{1}{(\ln(r_\text{max}/r_\text{min}))^2} \left[  
	\int_0^{r_{\text{min}}}\frac{1}{r_\text{min}^2} r dr
	+
	\int_{r_\text{min}}^{r_{\text{max}}}\frac{r}{r^2}dr
	\right]
	\\ \nonumber
	&
	\le c\norm{K}_\infty^{2(d-3)}\left(\prod_{i=3}^d h_i\right) \frac{1}{(\ln(r_\text{max}/r_\text{min}))^2} \left[  
	1+	\ln( \frac{r_{\text{max}}}{r_\text{min}})
	\right]
	\\&  \label{E: controle I2 preuve borne inf log}
	 \le c \norm{K}_\infty^{2(d-3)}\left(\prod_{i=3}^d h_i\right)\frac{1}{(\ln(r_\text{max}/r_\text{min}))}. 
\end{align}
 Collecting \eqref{E: decoupe norme b radial}, \eqref{E: controle I1 preuve borne inf log} and \eqref{E: controle I2 preuve borne inf log}, we deduce
 \begin{equation} \label{E: norme L2 radial bT b0}
 	\norm{b_r^{(1)}-b^{(0)}_r}_{2}^2  \le   c
 	\frac{\prod_{i=3}^d h_i}{M_T^2\ln(\frac{r_\text{max}}{r_\text{min}})}.
 \end{equation}
 	
We now compute the contribution of $b_k$ for $k \ge 3$ in the $\mathbf{L}^2$ norms of $ b^{{\modar(1)}}-b^{{\modar(0)}}$. {\modar Using \eqref{E:diff_b1_b0}, we have}
\begin{align}\nonumber
		\norm{b_k^{(1)}-b_k^{(0)}}_2^2 &\le \frac{c}{M_T^2}\int_{K_T} \left[ \abs{d_T(r,x_3,\dots,x_d)}^2+ \abs{\frac{\partial d_T(r,x_3,\dots,x_d)}{\partial x_k}}^2 \right] rdr d\theta dx_3\dots dx_d
		\\ \label{E: decoupe norme b k}
		&{\modarn =:}\frac{c}{M_T^2} (I_1(T)+I_3(T)).
\end{align}
 We have to upper bound $I_3(T)$ which is the new term. From the definition of $d_T$, we have 
 \begin{align*}
I_3 
&= \int_0^{r_{\text{max}}} |J_{r_\text{min},r_\text{max}}(r)|^2 rdr \int_{\prod_{i=3}^d[-h_i,h_i]} 
\left( 	\prod_{\stackrel{i=3}{i\neq k}}^d K(\frac{x_i}{h_i})^2 \right) |K'(\frac{x_k}{h_k})|^2 \frac{1}{h_k^2}
 	dx_3\dots dx_d
 \\ & \le \norm{K}_\infty^{2(d-4)}\norm{K'}_\infty^{2} \int_0^{r_{\text{max}}}\left( \frac{\ln(r_\text{max}/r)+1}{\ln(r_\text{max}/r_\text{min})}\wedge 1 
 \right)^2 rdr \left(\frac{ \prod_{i=3}^d h_i}{h_k^2} \right) 
 \end{align*}
 where we used \eqref{E: cond J 2}. Now by exactly the same computation yielding to {\modar \eqref{E: controle I1 preuve borne inf log},} we have,
 \begin{equation*}
I_3 \le c \norm{K}_\infty^{2(d-4)}\norm{K'}_\infty^{2} \frac{\prod_{i=3}^d h_i}{h_k^2}
\frac{r_\text{max}^2}{\ln(r_\text{max}/r_\text{min})}.
\end{equation*}
 Using that $r_\text{max}\le h_k$, we deduce
 \begin{equation} \label{E: controle I3 preuve borne inf log}
 I_3 \le c  \norm{K}_\infty^{2(d-4)}\norm{K'}_\infty^{2}  \frac{\prod_{i=3}^d h_i}{\ln(r_\text{max}/r_\text{min})}.
 \end{equation}
 Collecting \eqref{E: decoupe norme b k}, \eqref{E: controle I1 preuve borne inf log} and \eqref{E: controle I3 preuve borne inf log} we get
 \begin{equation} \label{E: norme L2 k bT b0}
 	\norm{b_k^{(1)}-b_k^{(0)}}_2^2 \le c \frac{\prod_{i=3}^d h_i}{M_T^2\ln(\frac{r_\text{max}}{r_\text{min}})}.
 \end{equation}
 Eventually, the upper bound \eqref{E: majo L2 diff b log} is a consequence of \eqref{E: norme L2 radial bT b0}, \eqref{E: norme L2 k bT b0} and $b_\theta^{(1)}=b_\theta^{(0)}=0$. 	
 
$\bullet$ We now prove \eqref{E: majo Linfint diff b log}.  Using again $\pi^{(1)}=\pi^{(0)}+\frac{d_T}{M_T}$ and \eqref{E: lien b pi cyl}, we have
$$
\norm{b^{(1)}-b^{(0)}}_\infty \le \frac{c}{M_T} \left[ \norm{d_T}_\infty+
 \norm{\frac{\partial d_T}{\partial r}}_\infty + \sum_{k=3}^d \norm{\frac{\partial d_T}{\partial x_k}}_\infty
\right].
$$
By the definition \eqref{E: def bump log} of $d_T$ with \eqref{E: cond J 0}, we have $\norm{d_T}_\infty \le c$ and $\norm{\frac{\partial d_T}{\partial x_k}}_\infty \le \frac{c}{h_k}$ for $k \ge 3$. Using \eqref{E: cond J 3}, we have $\norm{\frac{\partial d_T}{\partial r}}_\infty \le \frac{c}{r_\text{min} \ln(\frac{r_\text{max}}{r_\text{min}})}$. The upper bound \eqref{E: majo Linfint diff b log} follows.

$\bullet$ Finally, we  prove \eqref{E: majo Linfint diff grad b log}.  By \eqref{E: lien b pi cyl} and $\pi^{(1)}=\pi^{(0)}+\frac{d_T}{M_T}$, we have after some computations
\begin{multline*}
\norm{\nabla b^{(1)}-\nabla b^{(0)}}_\infty \le \frac{c}{M_T} \Big[ \norm{d_T}_\infty+
\sum_{y \in \{r,x_3,\dots,x_d\}} \norm{d_T}_\infty\norm{\frac{\partial d_T}{\partial y}}_\infty + 
\\
\frac{1}{M_T}
\sum_{y,y' \in \{r,x_3,\dots,x_d\}}
\norm{\frac{\partial d_T}{\partial y}}_\infty
\norm{\frac{\partial d_T}{\partial y'}}_\infty+
\sum_{y,y' \in \{r,x_3,\dots,x_d\}}
\norm{\frac{\partial^2 d_T}{\partial y \partial y'}}_\infty
\Big].
\end{multline*}
The upperbound \eqref{E: majo Linfint diff grad b log} follows from
$\norm{\frac{\partial^2 d_T}{\partial x_i \partial x_j}}_\infty \le c/(h_ih_j)$, 
$\norm{\frac{\partial^2 d_T}{\partial x_i \partial r}}_\infty \le c/(h_i r_\text{min} \ln(\frac{r_\text{max}}{r_\text{min}}))$ and
$\norm{\frac{\partial^2 d_T}{\partial r^2}}_\infty \le c/( r_\text{min}^2 \ln(\frac{r_\text{max}}{r_\text{min}}))$.
\end{proof}

 {\modar We can now discuss the conditions ensuring that $(\mathbb{I}_{d \times d}, b^{(1)} ) \in \Sigma(\beta, \mathcal{L},1,a_0,a_1,b_0,b_1,\tilde{C},\tilde{\rho})$.}
 Using that $J_{r_\text{min},r_\text{max}}$ is constant in a neighbourhood of zero, we deduce that $(x_1,\dots,x_d) \mapsto \pi^{(1)}(\sqrt{x_1^2+x_2^2},x_3,\dots,x_d)$ is a smooth function and by \eqref{E: def bump log} with \eqref{E: cond J 3} we have
 \begin{align*}
 	&\abs{	\frac{\partial^k }{\partial x_1^k} \pi^{(1)}(\sqrt{x_1^2+x_2^2},x_3,\dots,x_d)}
 	+
 	\abs{	\frac{\partial^k }{\partial x_2^k} \pi^{(1)}(\sqrt{x_1^2+x_2^2},x_3,\dots,x_d)}
 	\le \frac{c(k)}{M_T \ln(\frac{r_\text{max}}{r_\text{min}}) r_\text{min}^k},
 	\\
 	& \abs{	\frac{\partial^k }{\partial x_l^k} \pi^{(1)}(\sqrt{x_1^2+x_2^2},x_3,\dots,x_d)}
 	\le \frac{c(k)}{M_T h_l^k}, \quad \forall k \ge 1, \forall l \in \{3,\dots,d\}.
 \end{align*} 
 Then, the following condition is sufficient to ensure that $\pi^{(1)} \in \mathcal{H}(\beta,2\mathcal{L})$:
 \begin{equation} \label{E: cond regularite pi log}
 	\frac{1}{M_T}\le \varepsilon_0 r_\text{min}^{\beta_l} \ln(r_\text{max}/r_\text{min}), \text{for $l=1,2$ and,}\quad 	
 	\frac{1}{M_T}\le \varepsilon_0 h_l^{\beta_l}, ~ \forall l \in \{3,\dots,d\},
 \end{equation}
 for some constant $\varepsilon_0>0$ small enough. Remark also that we have $<b^{(1)}(x),x>\le - \tilde{C} |x|$ for $|x| \ge \tilde{\rho}$ for $T$ large enough, using that $b^{(1)}$ and $b^{(0)}$ coincide on $K_T^c$. Moreover, to get that for $T$ large enough,
 $\norm{b^{{\modar (1)}}}_\infty \le b_0$ and $\norm{\nabla b^{(1)}} \le b_1$ it is sufficient that
 \begin{equation}
 \label{E: cond b1 dans classe}
 \frac{c}{M_T} \left[ \frac{1}{r_\text{min}^2 \ln(\frac{r_\text{max}}{r_\text{max}})} + \frac{1}{h_3^2} \right] \xrightarrow{ T \to \infty} 0.
 \end{equation}
  Thus, we have $(\mathbb{I}_{d \times d}, b^{(1)} ) \in \Sigma(\beta, \mathcal{L},1,b_0,a_0,a_1,b_1,\tilde{C},\tilde{\rho})$ as soon as \eqref{E: cond regularite pi log} and  \eqref{E: cond b1 dans classe} are valid.  
 
 $\bullet$ Choice of the calibration and proof of \eqref{E: lower bound log}.
 
 Repeating the proof of \eqref{E: lower bound risk in proof wo log} we have 
 
\begin{equation*}
		\mathcal{R}_T (\beta, \mathcal{L},1,b_0,a_0,a_1,b_1,\tilde{C},\tilde{\rho}) \gtrsim\frac{1}{M_T^2},
 \end{equation*}
 as soon as $\sup_{T \ge 0} T \int_{\mathbb{R}^d}  |b^{(1)} (x) - b^{(0)}(x)|^2 \pi^{\modar(0)}(x) \, dx < \infty$. From \eqref{E: majo L2 diff b log} this condition is implied by
 \begin{equation} \label{E: cond Girsanov log}
 	\sup_{T\ge0} T \frac{\prod_{i=3}^d h_i}{M_T^2\ln(\frac{r_\text{max}}{r_\text{min}})} < \infty.
 \end{equation}
 Now, we search $1/M_T \to 0$ tending to zero as slow as possible subject to the existence of $0<r_\text{min}<r_\text{max}/4 < r_\text{max} \le h_3 \le \dots  \le h_d$ satisfying \eqref{E: cond regularite pi log}--\eqref{E: cond b1 dans classe} and \eqref{E: cond Girsanov log}. For $l \ge 3$, we set $h_l=(\frac{1}{M_T\varepsilon_0})^{1/\beta_l}$, $r_\text{max}=h_3$ and 
  $r_\text{min}=(\frac{1}{M_T\varepsilon_0})^{1/\beta_2}$. We have $\log(r_\text{max}/r_\text{min})\sim_{T \to \infty} (\frac{1}{\beta_2}-\frac{1}{\beta_3}) \ln (M_T) \to \infty$ using  $\beta_2<\beta_3$. Since $\beta_1\le\beta_2$, we deduce that the conditions $ \frac{1}{M_T}\le \varepsilon_0 r_\text{min}^{\beta_l} \ln(r_\text{max}/r_\text{min})$, for $l=1,2$ hold true when $T$ is large enough. Hence, \eqref{E: cond regularite pi log} is satisfied, and with these choices \eqref{E: cond b1 dans classe} becomes a consequence of $\beta_3>\beta_2 > 2$.
 Replacing $h_i$, $r_\text{min}$, $r_\text{max}$ by their expression in function of $M_T$, the condition \eqref{E: cond Girsanov log} writes,
 \begin{equation*}
 	\sup_{T\ge0} T \frac{1}{ M_T^{2 + (d-2)/{\overline{\beta}_3}}\ln(M_T)} < \infty,
 \end{equation*}
 and is satisfied for the choice $M_T=\left( T/\ln(T) \right)^{\frac{\overline{\beta}_3}{2\overline{\beta}_3+(d-2)}}$. This yields to
 \begin{equation*}
 	\mathcal{R}_T (\beta, \mathcal{L},1,b_0,a_0,a_1,b_1,\tilde{C},\tilde{\rho})
 	\gtrsim\left( T/\ln(T) \right)^{-\frac{2\overline{\beta}_3}{2\overline{\beta}_3+(d-2)}}.
 \end{equation*}

To complete the proof Theorem \ref{th: lower bound avec log}, it remains to show the existence of the function $J_{r_\text{min},r_\text{max}}$. 
This is done in the next lemma.
 \end{proof}
\begin{lemma}There exists a smooth function $J_{r_\text{min},r_\text{max}}: \mathbb{R} \to [0,\infty)$ with  compact support on $[0,r_{\max}]$ such that \eqref{E: cond J 0}--\eqref{E: cond J 3} hold true.	
	\label{L: existence J}
\end{lemma}
\begin{proof}
	We let $\varphi : [0,\infty) \to [0,1]$ be a smooth function with $\varphi{\scriptscriptstyle{\vert [0,1/4]}}=0$ and $\varphi{\scriptscriptstyle{\vert [1/2,\infty)}}=1$.
	We define piecewise the function $J_{r_\text{min},r_\text{max}}$ as follows,
	\begin{align*}
		&J_{r_\text{min},r_\text{max}}(r)=0, \quad \text{ for $r \ge r_\text{max}$,}
		\\
		&J_{r_\text{min},r_\text{max}}(r)=
	\frac{1}{ \ln(r_\text{max} /r_\text{min})}
		\int_r^{r_\text{max}} 
		\varphi(1-{\modarn s}/r_{\text{max}}) {\modarn \frac{ds}{s},} \quad \text{ for $r_\text{max}/2 \le r \le r_\text{max}$,}
		\\&
		J_{r_\text{min},r_\text{max}}(r)=J_{r_\text{min},r_\text{max}}(r_\text{max}/2)+
			\frac{1}{ \ln(r_\text{max} /r_\text{min})}
		\int_r^{r_\text{max}/2}{\modarn \frac{ds}{s},}  \quad \text{ for $r_\text{min}/2 \le r \le r_\text{max}/2$,}		
		\\&
		J_{r_\text{min},r_\text{max}}(r)=J_{r_\text{min},r_\text{max}}(r_\text{min}/2)+
		\frac{1}{ \ln(r_\text{max} /r_\text{min})}
		\int_r^{r_\text{min}/2} \varphi({\modarn s}/r_\text{min}) {\modarn \frac{ds}{s},} \quad \text{ for $ 0\le r \le r_\text{min}/2$.}		
	\end{align*}
Using the definition of $\varphi$, one can check that the derivative of the function $J$ is smooth and that $J$ is decreasing with $J'(r)=0$ for $|r|\le r_{\text{min}}/4$. 
By simple computation we have,
\begin{align} \label{E: J explicite morceau1}
&J_{r_\text{min},r_\text{max}}(r)=
\frac{\Phi_1(r/r_\text{max})}{ \ln(r_\text{max} /r_\text{min})} , \quad \text{ for $r_\text{max}/2 \le r \le r_\text{max}$,}
\\&\label{E: J explicite morceau2}
J_{r_\text{min},r_\text{max}}(r)=\frac{\Phi_1(1/2) +
	\ln(\frac{r_{\max}}{2r})}{\ln(r_\text{max} /r_\text{min})},
\quad \text{ for $r_\text{min}/2 \le r \le r_\text{max}/2$,}			
\\&\label{E: J explicite morceau3}
J_{r_\text{min},r_\text{max}}(r)=1+
\frac{\Phi_1(1/2) + \Phi_2(r/r_\text{min})}{ \ln(r_\text{max} /r_\text{min})}, \quad \text{ for $ 0\le r \le r_\text{min}/2$,}		
\end{align}
where $\Phi_1(u)=\int_u^1 \varphi(1-s) \frac{ds}{s}$ and 
$\Phi_2(u)=\int_u^{1/2} \varphi(s) \frac{ds}{s}$ are smooth functions on $[0,1]$.
%
%

Thus, using monotonocity of $J_{r_\text{min},r_\text{max}}$ and $\varphi \ge 0$, we have $J_{r_\text{min},r_\text{max}}(0) \ge J_{r_\text{min},r_\text{max}}(r_\text{min}/2)\ge 1$. We deduce that \eqref{E: cond J 0} is true.

Using \eqref{E: J explicite morceau1}--\eqref{E: J explicite morceau3} and that $\Phi_1$ and $\Phi_2$ are bounded on $[0,1]$, we deduce {\modar \eqref{E: cond J 1}.}  The condition \eqref{E: cond J 2} follows again from  \eqref{E: J explicite morceau1}--\eqref{E: J explicite morceau3} and the boundedness of $\Phi_1'$ and $\Phi_2'$. The condition \eqref{E: cond J 3} is shown by differentiating $k$ times the representations \eqref{E: J explicite morceau1}--\eqref{E: J explicite morceau3} and using $r_\text{min}\le r_\text{max}$.
\end{proof}
{\change 
\begin{remark}\label{rem : bosse log}
The idea underneath the construction of the bump $(r,\theta) \mapsto J_{r_\text{min},r_\text{max}}(r)$ in Lemma \ref{L: existence J} is the following. Having in mind \eqref{E: grad bpi eg delta pi}, 
we have constructed a smooth bump, 
with radius $r_\text{max}$, height greater than $1$, and solution of $\Delta J_{r_\text{min},r_\text{max}} =0$ on the torus $r\in [r_\text{min}/2,r_\text{max}/2]$.
\end{remark}
}
\subsection{Proof of Theorem \ref{th: lower bound d=2}}
\begin{proof}
The proof of Theorem \ref{th: lower bound d=2} heavily relies on the proof of Theorem \ref{th: lower bound avec log}. 
As done above, we prove the theorem in the case where $x_0 = (0, 0)$ and $a_{min} = 1$. We will lower bound the risk on the subclass of model \eqref{eq: model} given by
$$dX_t=b(X_t)dt + dW_t,$$
 where $b$ is any drift function such that $(\mathbb{I}_{d \times d},b) \in {\modarn\tilde\Sigma}(\beta, \mathcal{L},1,b_0,a_0,a_1,b_1,\tilde{C},\tilde{\rho})$.
Following the proof of Theorem \ref{th: lower bound avec log} we introduce cylindrical coordinates $(r, \theta)$ that translates into the Cartesian coordinates $(r \cos(\theta),r \sin(\theta))$. 
In two dimensions \eqref{E:Astar pi zero cyl} is written without the sum over $k$ and the drift $b$ solution of \eqref{E:Astar pi zero cyl} is given by 
$b=b_r \vec{e}_r+ b_\theta \vec{e}_\theta$ with $\vec{e}_r=(\cos(\theta), \sin(\theta))^T$, $\vec{e}_\theta=(-\sin(\theta), \cos(\theta),)^T$ and  
$$	b_r=\frac{1}{\pi} \frac{\partial \pi}{\partial r}, \quad b_\theta=0. $$

$\bullet$ Construction of the priors.  Let $\psi: [0,\infty) \to [0,\infty)$ be a smooth function, vanishing on $[0,1/2]$ and satisfying $\psi(x)=x$ for $x \ge 1$ as in the proof of Theorem \ref{th: lower bound avec log}. We let
$$\pi^{(0)}(r)=c_\eta e^{-\eta \psi(r)},
$$ 
 where $\eta>0$ and $c_\eta$ is such that $\int_{[0,\infty)\times[0,2\pi)} \pi^{(0)}(r) r ~drd\theta=1$.
 Again, for $\eta$ small enough we have $\pi^{(0)} \in \mathcal{H}_2(\beta,\mathcal{L})$. Moreover, acting as in the proof of Theorem \ref{th: lower bound avec log}, it is easy to see that $( \mathbb{I}_{d \times d},b^{(0)}) \in {\modarn\tilde\Sigma} (\beta, \mathcal{L},1,b_0,a_0,a_1,b_1/2,\tilde{C},\tilde{\rho})\subset {\modarn\tilde\Sigma} (\beta, \mathcal{L},1,b_0,a_0,a_1,b_1,\tilde{C},\tilde{\rho})$.

The construction of $\pi^{(1)}$ is more elaborate, as before we add a bump centered at $0$ to $\pi^{(0)}$. We set 
$$\pi^{(1)}=\pi^{(0)}+\frac{1}{M_T} J_{r_\text{min}(T),r_\text{max}(T)}(r),$$
where $J_{r_\text{min}(T),r_\text{max}(T)}$ is the function introduced in the proof of Theorem \ref{th: lower bound avec log} which satisfies \eqref{E: cond J 0}, \eqref{E: cond J 1}, \eqref{E: cond J 2} and \eqref{E: cond J 3} and whose existence has been proven in Lemma \ref{L: existence J}.
We recall that $0<r_\text{min}(T)<r_\text{max}(T)/4 < r_\text{max}(T)$ will be calibrated later and go to zero with a rate polynomial in $1/T$. We denote by $b^{(1)}$ the drift function associated to $\pi^{(1)}$. \\
Following the proof of Lemma \ref{L: maj norme L2 diff b log} in the bi-dimensional context it is easy to see that the following bounds hold true.
	\begin{equation*}
 	\int_{\mathbb{R}^d} \abs{b^{(1)}(\sqrt{x_1^2+x_2^2})-b^{(0)}(\sqrt{x_1^2+x_2^2})}^2 dx_1 dx_2\le c
 	\frac{1}{M_T^2\ln(\frac{r_\text{max}}{r_\text{min}})},
 	\end{equation*}
 \begin{align*}
 	\norm{b^{(1)}-b^{(0)}}_\infty & \le \frac{c}{M_T} \frac{1}{r_\text{min} \ln(\frac{r_\text{max}}{r_\text{max}})} ,
 	\\ 
 	\norm{\nabla b^{(1)}-\nabla b^{(0)}}_\infty &\le \frac{c}{M_T} \frac{1}{r_\text{min}^2 \ln(\frac{r_\text{max}}{r_\text{max}})}.
 \end{align*}
Moreover we still have, for any $k \ge 1$, 
 \begin{align*}
 	&\abs{	\frac{\partial^k }{\partial x_1^k} \pi^{(1)}(\sqrt{x_1^2+x_2^2})}	+
 	\abs{	\frac{\partial^k }{\partial x_2^k} \pi^{(1)}(\sqrt{x_1^2+x_2^2})}
 	\le \frac{c(k)}{M_T \ln(\frac{r_\text{max}}{r_\text{min}}) r_\text{min}^k}.
 \end{align*} 
 Then, the following condition is sufficient to ensure that $\pi^{(1)} \in \mathcal{H}_2(\beta,2\mathcal{L})$:
 \begin{equation} \label{E: cond regularite pi log d = 2}
 	\frac{1}{M_T}\le \varepsilon_0 r_\text{min}^{\beta_l} \ln(r_\text{max}/r_\text{min}), \text{for $l=1,2$ and }
 \end{equation}
 for some constant $\varepsilon_0>0$ small enough. As before, if \eqref{E: cond regularite pi log d = 2} holds true it is sufficient that 
 \begin{equation}
 \frac{c}{M_T} \frac{1}{r_\text{min}^2 \ln(\frac{r_\text{max}}{r_\text{max}})}  \xrightarrow{ T \to \infty} 0
 \label{E: cond b1 dans classe d=2}
 \end{equation}
 to obtain $(\mathbb{I}_{d \times d}, b^{(1)} ) \in \Sigma(\beta, \mathcal{L},1,b_0,a_0,a_1,b_1,\tilde{C},\tilde{\rho})$.  
 
 $\bullet$ Choice of the calibration and proof of \eqref{E: lower bound log}.
 
 Repeating the proof of \eqref{E: lower bound risk in proof wo log} we clearly have 
 $$	\mathcal{R}_T (\beta, \mathcal{L},1,b_0,a_0,a_1,b_1,\tilde{C},\tilde{\rho}) \gtrsim\frac{1}{M_T^2},$$
 as soon as $\sup_{T \ge 0} T \int_{\mathbb{R}^2}  |b^{(1)} (x) - b^{(0)}(x)|^2 \pi^{\modar (0)}(x) \, dx < \infty$, which is implied by
 \begin{equation} \label{E: cond Girsanov log d=2}
 	\sup_{T\ge0} T \frac{1}{M_T^2\ln(\frac{r_\text{max}}{r_\text{min}})} < \infty.
 \end{equation}
We look for $1/M_T$ tending to zero as slow as possible subject to the existence of $0<r_\text{min}<r_\text{max}/4 < r_\text{max} $ satisfying \eqref{E: cond regularite pi log d = 2}--\eqref{E: cond b1 dans classe d=2} and \eqref{E: cond Girsanov log d=2}. We set
$r_\text{min}=(\frac{1}{M_T\varepsilon_0})^{1/\beta_2}$ and  $r_\text{max}=(\frac{1}{M_T\varepsilon_0})^{1/\beta_2 - \gamma}$ for any arbitrary $\gamma> 0$.
It follows $\log(r_\text{max}/r_\text{min})\sim_{T \to \infty} (\frac{1}{\beta_2}-(\frac{1}{\beta_2}- \gamma)) \ln (M_T) \to \infty$ as $\gamma > 0$. Since $\beta_1\le\beta_2$, we clearly have from the definition of $r_\text{min}$ and $r_\text{max}$ that the conditions $ \frac{1}{M_T}\le \varepsilon_0 r_\text{min}^{\beta_l} \ln(r_\text{max}/r_\text{min})$, for $l=1,2$ hold true when $T$ is large enough. Hence, \eqref{E: cond regularite pi log d = 2} is satisfied. The same can be said about \eqref{E: cond b1 dans classe d=2}, recalling we have assumed $\beta_2 > 2$.
 Replacing $r_\text{min}$ and $r_\text{max}$ by their expression in function of $M_T$, condition \eqref{E: cond Girsanov log d=2} writes
 \begin{equation*}
 	\sup_{T\ge0} T \frac{1}{ M_T^{2}\ln(M_T)} < \infty,
 \end{equation*}
 and is satisfied for the choice $M_T=(\frac{T}{\log T})^{\frac{1}{2}}$. It provides
 \begin{equation*}
 	\mathcal{R}_T (\beta, \mathcal{L},1,b_0,a_0,a_1,b_1,\tilde{C},\tilde{\rho})
 	\gtrsim \frac{\log T}{T},
 \end{equation*}
as we wanted.

\end{proof}
\section*{Acknowledgement}
The authors are very grateful to the referees for their careful reading, their suggestions and remarks which improved the paper.

\appendix

\section{Appendix: proof of technical results}{\label{s: technical}}
This section is devoted to the proof of the results which are more technical and for which some preliminaries are needed.

{\modchi \subsection{Proof of Corollary \ref{lemma: bound transition density}}
\begin{proof}
We start proving the upper bound. Clearly we have {\modarn $|x-y|^2 \le 2|\theta_{t,s}(x)-y|^2+2|\theta_{t,s}(x)-x|^2$ and thus
$$- |\theta_{t,s}(x) - y|^2 \le - \frac{1}{2}|x - y|^2 + |\theta_{t,s}(x) - x|^2.$$}
Then, the proof consists in finding a good bound for $|\theta_{t,s}(x) - x|$, together with the straightforward application of Lemma \ref{l: bound with flow}. We start by noticing that, when $|\theta_{t,s}(x)| > 1$, $|\theta_{t,s}(x)|^2$ is a decreasing function in $t$. Indeed, because of the definition of the flow, its derivative with respect to time is
$${\modarn 2 <\theta_{t,s}(x), \dot{\theta}_{t,s}(x)>} = 2 < \theta_{t,s}(x) , b(\theta_{t,s}(x)) > \, < 0, $$
where the last is a consequence of A2. We remark that, if $| \theta_{t,s}(x)| > 1$, it is not possible to have $|x| \le 1$, as in this case the flow could increase only until it reaches the value 1. Hence, for $| \theta_{t,s}(x)| > 1$ we have also $|x| > 1$ and it is $|\theta_{t,s}(x)| \le
{\modarn |\theta_{s,s}(x)| = |x| }$ for any $t \ge 0$ {\modarn since $t\mapsto |\theta_{t,s}(x)|$ is decreasing until it possibly reaches the value $1$.}
If instead $|\theta_{t,s}(x)| \le 1$, then
we have 
\begin{equation}
|\theta_{t,s}(x) - x| \le 1 + |x|. 
\label{eq: theta smaller than 1}
\end{equation}
The definition of the flow together with the mean value theorem provides 
\begin{align*}
|\theta_{t,s}(x) - x| &= |\int_{s}^t b(\theta_{r,s}(x)) dr| \\
& \le \int_s^t |b(0)| dr + \int_s^t \left \| b' \right \|_\infty |\theta_{r,s}(x)| dr \\
& \le b_0|t -s| + b_1 |t -s| (|x| + 1).
\end{align*}
It concludes the proof of the upper bound on the transition density. \\
On the other side we have 
{\modarn
$$-|\theta_{t,s}(x) - y|^2 \ge - 2|x - y|^2  -2|\theta_{t,s}(x) - x|^2.$$}
Hence, we need to find a lower bound for {\modarn $-|\theta_{t,s}(x) - x|^2$.} What we have already showed, together with the application of Lemma \ref{l: bound with flow}, concludes the proof.
\end{proof}

}

\subsection{Proof of Lemma \ref{L:hyp sigma stronger}}\label{Ss:proof_unif_mixing}

Recalling Lemma \ref{lemma: ergodicity}, the main novelty is to prove that the control \eqref{E: resu mixing reg non reg} is uniform on the class of coefficients $(a,b) \in \Sigma$.
To simplify the proof, we first strengthen the hypothesis on the coefficients of the diffusion by assuming that they are infinitely differentiable with bounded derivatives or any order, together with the condition $(a,b) \in \Sigma$. 

The proof of \eqref{E: resu mixing reg non reg} relies on the theory of Lyapunov-Poincaré {\modar inequalities} introduced in \cite{Bak_et_al08}. We recall some definitions related to {\modar these} functional {\modar inequalities}. First, if $f \in\mathcal{C}^2(\mathbb{R}^d,\mathbb{R})$ we define $\Gamma(f)=A(f^2)-2fA(f)$ where $A$ is the generator of the diffusion, and we have $\Gamma(f)=(\nabla f)^T \tilde{a} \nabla f=|a^T \nabla f|^2$.
{\modarn We let $P^{*,\pi}_t$ be the adjoint in $\mathbf{L}^2({\modar \R^d}, \pi)$ of the operator $P_t$, which is given by the expression
	\begin{equation}\label{E:expression Pstar}
		P^{*,\pi}_t(f)(x)=\frac{1}{\pi(x)}\int_{\mathbb{R}^d} f(y) \pi(y) p_t(y,x) dy,
	\end{equation}
	and $A^{*,\pi}$ the generator of the semi group $(P^{*,\pi}_t)_{t \ge 0}$ (details can be found e.g.\ in \cite{Hairer_SPDE_Lecture09}). If $f$ is $\mathcal{C}^2$ then
	\begin{equation}\label{E:expr_A_star_pi}
		A^{*,\pi} f = \frac{1}{2} \sum_{1\le i,j \le d}\tilde{a}_{i,j}\frac{\partial^2f}{\partial x_i \partial x_j}
		+ \sum_{i=1}^d b^*_i \frac{\partial f}{\partial x_i}  
	\end{equation}
	with $b^*_i=-b_i+\sum_{j=1}^d \frac{\partial \tilde{a}_{j,i} }{\partial x_j} +
	\frac{1}{\pi}\sum_{i=1}^d (\frac{\partial \pi}{\partial x_j} )\tilde{a}_{j,i}$
	as soon as the expression in the right hand side of \eqref{E:expr_A_star_pi} belongs to $\mathbf{L}^2(\pi)$. 	
}
\begin{definition}{(Lyapunov-Poincaré Inequality \cite{Bak_et_al08})} 
	\label{def: LPI}
	Let $W$ be a $\mathcal{C}^2(\mathbb{R}^d,\mathbb{R})$ 
	function such that $W \ge 1$, $W \in \mathbf{L}^2(\pi)$, 
	$A W \in \mathbf{L}^2(\pi)$. The probability $\pi$ satisfies a 
	$W$-Lyapunov-Poincaré inequality if there exists $C_{\text{LP}}>0$ such that for all 
	$f\in\mathcal{C}^2(\mathbb{R}^d,\mathbb{R})$, 
	{\modarn  bounded function
		in the domain of $A^{*,\pi}$, 
		with $\Gamma(f)W \in \mathbf{L}^1(\pi)$ and }
	$\int_{\mathbb{R}^d}f(x)\pi(x)dx=0$ we have
	\begin{equation} \label{E:WPI def}
		\int_{\mathbb{R}^d} f^2(x)W(x) \pi(x) dx
		\le C_{\text{LP}} \int_{\mathbb{R}^d}
		[W(x)\Gamma(f)(x)-f(x)^2 AW(x)] \pi(x) dx.
	\end{equation} 
\end{definition}
{\revar We can now state a result of \cite{Bak_et_al08} which shows that a W-Lyapounov-Poincaré inequality entails a mixing property of the semi-group. The mixing rate depends explicitly on the constant $C_{LP}$ and consequently controlling $C_{LP}$ is the central point to get the uniform mixing in Lemma \ref{L:hyp sigma stronger}.}    For the sake of completeness, we  give a {\modar detailed} proof adapted to our context of this result.
\begin{lemma}[Bakry et al. \cite{Bak_et_al08}]\label{L:lemme de Bakry}
	Assume that the $W$-Lyapunov-Poincaré inequality holds true with some constant $C_{\text{LP}}>0$.
	Then,
	for any $\varphi: \mathbb{R}^d \to \mathbb{R}$ bounded function with
	$\int_{\mathbb{R}^d} \varphi(x) \pi(x)dx=0$, we have
	\begin{equation}\label{E:conclusion_LPI}
		\int_{\mathbb{R}^d} [P_t(\varphi)(x)]^2 \pi(x) dx \le  e^{-t/C_{LP}} \int_{\mathbb{R}^d} \varphi^2(x) W(x)
		\pi(x) dx,
	\end{equation}
	where $(P_t)_t$ is the semi-group associated to the process $X$.
\end{lemma}
\begin{proof}
	$\bullet$ First, we establish some properties on the stationary density $\pi$. 
	{\modarn Using Theorem 1.2 in \cite{MenPes}, we know that
		\begin{equation}\label{E:control_MPZ}
			\abs{\frac{\partial^{j} }{\partial y^j} p_t(x,y)} \le \frac{c}{t^{(d+j)/2}} e^{-\frac{\abs{\theta_{t,0}(x)-y}^2}{ct}}
		\end{equation}		
		for $j\in\{0,1,2\}$.
	}
	%
		From the invariance of $\pi$, $\pi(y)=\int_{\mathbb{R}^d} \pi(x) p_{t}(x,y)dx$, we deduce that {\modar $\pi \in \mathcal{C}^2(\mathbb{R}^d,\mathbb{R})$, and that  $\pi$ is bounded, with bounded derivatives.}
		We also deduce from Lemma \ref{l: bound with flow} that $\pi(x)>0$ for all $x \in \mathbb{R}^d$.
		
		$\bullet$  We now prove  \eqref{E:conclusion_LPI}. Remark that by a density argument we can assume that
		$\varphi$ is $\mathcal{C}^2$ and compactly supported.
		We know that if $f$ is in the domain of $A^{*,\pi}$ then $t\mapsto P_t^{*,\pi}(f)$ is differentiable in $\mathbf{L}^2(\pi)$ and
		$\frac{\partial}{\partial t} P_t^{*,\pi}(f)= P_t^{*,\pi} A^{*,\pi}(f)=A^{*,\pi}P_t^{*,\pi}(f)$.

As in \cite{Bak_et_al08}, we define $I_t=\int_{\mathbb{R}^d} (P_t^{*,\pi}\varphi(x))^2 W(x) \pi(x) dx$. By formal differentiation, we get for $t>0$,
\begin{equation} \label{E:diff It}
\frac{\partial I_t}{\partial t} =2 
\int_{\mathbb{R}^d} P_t^{*,\pi} \varphi(x) A^{*,\pi} P_t^{*,\pi}\varphi(x) W(x) \pi(x) dx.
\end{equation}
This differentiation can be justified by proving that $\frac{\partial}{\partial t} (P_t^{*,\pi}(\varphi))^2=2 P_t^{*,\pi} \varphi \times A^{*,\pi} P_t^{*,\pi}\varphi$ is dominated by a function in $\mathbf{L}^1(W\pi)$.
Using \eqref{E:expression Pstar} and that $\pi$ is invariant we get $|P_t^{*,\pi}(\varphi)|\le \norm{\varphi}_\infty$. Also,
$ |A^{*,\pi} P_t^{*,\pi}\varphi|=|P_t^{*,\pi}  A^{*,\pi} \varphi|$ is bounded by $\norm{A^{*,\pi} \varphi}_\infty<\infty$, using \eqref{E:expr_A_star_pi} and the fact that $\varphi$ is compactly supported. Since $W \in \mathbf{L}^1(\pi)$, we deduce \eqref{E:diff It}. Then, using that $A^{*,\pi}$ is a differential operator by \eqref{E:expr_A_star_pi}, we have 
$A^{*,\pi}(  (P_t^{*,\pi}(\varphi))^2 )=2   P_t^{*,\pi}(\varphi) A^{*,\pi}(P_t^{*,\pi}(\varphi)  ) + \Gamma(P_t^{*,\pi}(\varphi))$.
Hence,
$$
\frac{\partial I_t}{\partial t} = 
\int_{\mathbb{R}^d} [ A^{*,\pi}(  (P_t^{*,\pi}\varphi)^2 )(x)-\Gamma(P_t^{*,\pi}(\varphi) )  ]W(x) \pi(x) dx.
$$
We now prove that $\Gamma(P_t^{*,\pi}\varphi) W \in \mathbf{L}^1(\pi)$. Let $\chi_M$ a sequence of $\mathcal{C}^2(\mathbb{R},\mathbb{R})$ functions bounded with bounded derivative and approximating $x\mapsto x^2$  in the sense that  $|\chi_M(x)| \le x^2$, $|\chi_M'(x)| \le 2|x|$, $\chi_M''(x)\ge 0$, $\chi_M''(x)  \xrightarrow{n\to\infty} 2$ for all $x \in \mathbb{R}$. We have,
\begin{equation*}
	A^{*,\pi}(\chi_M(P_t^{*,\pi}\varphi))=
	\chi_M'(P_t^{*,\pi}\varphi) A^{*,\pi}(P_t^{*,\pi}\varphi)+
	\frac{1}{2}\chi_M{''}(P_t^{*,\pi}\varphi) \Gamma(P_t^{*,\pi}\varphi)	
\end{equation*}
Multiplying by $W$ and integrating with respect to $\pi$, it yields to
\begin{multline*}
\int_{\mathbb{R}^d}
\frac{1}{2}\chi_M{''}(P_t^{*,\pi}\varphi(x)) 
\Gamma(P_t^{*,\pi}\varphi)(x)W(x)\pi(x)dx\le
\abs{
\int_{\mathbb{R}^d}	
	A^{*,\pi}(\chi_M(P_t^{*,\pi}\varphi))(x)W(x)\pi(x)dx}\\
+\abs{\int_{\mathbb{R}^d}	
	\chi_M'(P_t^{*,\pi}\varphi(x)) A^{*,\pi}(P_t^{*,\pi}\varphi)(x)W(x) \pi(x)
 dx}\\
= \abs{
	\int_{\mathbb{R}^d}	
	\chi_M(P_t^{*,\pi}\varphi)(x)AW(x)\pi(x)dx}
+\abs{\int_{\mathbb{R}^d}	
	\chi_M'(P_t^{*,\pi}\varphi(x)) A^{*,\pi}(P_t^{*,\pi}\varphi)(x)W(x) \pi(x)
	dx}
\\ \le
\norm{(P_t^{*,\pi}\varphi)^2}_{\mathbf{L^2(\pi)}}
\norm{AW}_{\mathbf{L^2(\pi)}}
+
\norm{P_t^{*,\pi}\varphi}_{\mathbf{L^4(\pi)}}
\norm{A^{*,\pi}(P_t^{*,\pi}\varphi)}_{\mathbf{L^4(\pi)}}
\norm{W}_{\mathbf{L^2(\pi)}}.
\end{multline*}
Letting $M\to\infty$ and using Fatou's lemma, we deduce $\int_{\mathbb{R}^d}
\Gamma(P_t^{*,\pi}\varphi)(x)W(x)\pi(x)dx<\infty$.
Hence, we can write
\begin{align*}
\frac{\partial I_t}{\partial t} &= 
\int_{\mathbb{R}^d}  A^{*,\pi}(  (P_t^{*,\pi}\varphi)^2 )(x) W(x) \pi(x) dx
-
\int_{\mathbb{R}^d} \Gamma(P_t^{*,\pi}(\varphi) ) W(x) \pi(x) dx
\\
&=
\int_{\mathbb{R}^d}    (P_t^{*,\pi}\varphi)^2 (x) {\modar AW}(x) \pi(x) dx
-
\int_{\mathbb{R}^d} \Gamma(P_t^{*,\pi}(\varphi) ) W(x) \pi(x) dx
\\
&
=
\int_{\mathbb{R}^d}  [(P_t^{*,\pi}\varphi)^2 (x) AW(x)- \Gamma(P_t^{*,\pi}(\varphi) ) W(x)] \pi(x)dx.
\end{align*}
Now we use the Lyapunov-Poincaré inequality with $f=P_t^{*,\pi}(\varphi)$ and get
$\frac{\partial I_t}{\partial t}  \le -(1/C_{LP}) I_t$. From Gronwall's lemma it yields 
$I_t \le e^{-t/C_{LP}} I_0$ and since $W\ge1$,
$$
\int_{\mathbb{R}^d}(P_t^{*,\pi}(\varphi)(x))^2 \pi(x) dx\le I_t \le e^{-t/C_{LP}} I_0= e^{-t/C_{LP}}\int_{\mathbb{R}^d} \varphi(x)^2 W(x)\pi(x)dx.
$$
By duality between $P_t$ and $P_t^{*,\pi}$ in $\mathbf{L}^2(\pi)$, we deduce \eqref{E:conclusion_LPI}.
\end{proof}

\begin{lemma}\label{L: maj cov}
	Assume that $a$ and $b$ are $\mathcal{C}^2$ with bounded derivatives and $(a,b) \in \Sigma$. Then, for $\varphi$ bounded with $\int_{\mathbb{R}^d} \varphi(x) \pi(x) dx=0$, and all $t>0$,
	\begin{equation} \label{E:eq control L2 Pt}
		\int_{\mathbb{R}^d} [P_t(\varphi)(x)]^2 \pi(x) dx \le c e^{-t/c} \norm{\varphi}_\infty^2,
	\end{equation}
	where the constant $c>0$ is uniform over the class $\Sigma$. 
\end{lemma}
\begin{proof}
The main idea of the proof is that it is possible to show that for all $(a,b) \in \Sigma$, the stationary probability  $\pi$ satisfies a $W$-Lyapunov-Poincaré inequality with the same function $W$ and	same constant $C_{LP}$.
Following \cite{Bak_et_al08}, the existence of $W$-Lyapunov-Poincaré inequality is related to the existence of classical Lyapunov functions.   

$\bullet$ First, we construct a Lyapunov function $V$ independent of $(a,b)\in \Sigma$. 
We let $\chi\in \mathcal{C}^{\infty}(\mathbb{R}^d,\mathbb{R})$ such that $0\le\chi(x)\le|x|$ and $\chi(x)=|x|$ for $|x| \ge 1$ and we set
$V(x)=e^{\varepsilon_0 \chi(x)}$ for some $\varepsilon_0>0$ which will be calibrated later.
We have for $|x| \ge 1$, $\nabla V(x)=\varepsilon_0\frac{x}{|x|} V(x)$ and thus,
\begin{align*}
AV(x)=\frac{1}{2}\sum_{1\le i,j \le d} \tilde{a}_{i,j}(x) \frac{\partial^2 V}{\partial x_i \partial x_j}(x)+\epsilon_0 V(x) <\frac{x}{|x|},b(x)>.
\end{align*}
Using that for  $|x| \ge 1$,  $|\frac{\partial^2 V}{\partial x_i \partial x_j}|=|e^{\varepsilon_0 |x|} 
\{\varepsilon_0^2\frac{x_ix_j}{|x|^2} -\varepsilon_0\frac{x_ix_j}{|x|^3} + \frac{\varepsilon_0}{|x|}
1_{\{i=j\}}  \}|\le e^{\varepsilon_0|x|} \{\varepsilon_0^2+\frac{2\varepsilon_0}{|x|}\}$, and that $<x,b(x)>\le 
-{\modar \tilde{C}}|x|$ for $|x| \ge {\modar \tilde{\rho}}$ by Assumption {\bf A2}, we get,
$$
AV(x) \le -{\modar \tilde{C}}\varepsilon_0 V(x)+\frac{\varepsilon_0}{2} V(x) \sum_{i,j} |\tilde{a}_{i,j}(x)| \{\varepsilon_0+\frac{2}{|x|}\},
$$
for all $|x| \ge {\modar \tilde{\rho}} \vee 1$. Using $|\tilde{a}_{i,j}(x)|=|(aa^T)_{i,j}(x)|\le |a(x)|^2 \le a_0^2$ by 
Assumption {\bf A1}, we deduce
$$
AV(x) \le -{\modar \tilde{C}}\varepsilon_0 V(x) [1- \frac{d^2 a_0^2}{2 {\modar \tilde{C}} }
\{\varepsilon_0+\frac{2}{|x|}\} ] .
$$
Now, we set {\modarn $\varepsilon_0$ as any constant with $0<\varepsilon_0 < \frac{{\modarn \tilde{C}}}{2d^2 a_0^2}$,} and deduce for $|x| \ge{\modar \tilde{\rho}} \vee 1 \vee \frac{4 d^2 a_0^2}{{\modar \tilde{C}}}  $:
$$
AV(x) \le -\frac{{\modarn \tilde{C}}\varepsilon_0}{2} V(x) .
$$
We then define $\alpha={\modar \tilde{C}}\varepsilon_0/2$, $R={\modar \tilde{\rho}} \vee 1 \vee \frac{4 d^2 a_0^2}{{\modar \tilde{C}}} $ and $\beta= 
e^{\varepsilon_0 R}[ \frac{1}{2}d^2 a_0^2 \varepsilon_0^2 \norm{\nabla\chi}_\infty^2 +  \frac{1}{2} d^2 a_0^2 
\varepsilon_0 \times
\sup_{1\le i,j\le d}
\norm{\partial_{x_i,x_j}^2\chi}_\infty + {\modarn (b_0+b_1R)}  \varepsilon_0 \norm{\nabla\chi}_\infty + \alpha ]$ 
{\modarn where the notations 
	$b_0$, $b_1$ are introduced in Assumption {\bf A1}. Using that for $|x|\le R$, $|b(x)|\le b_0+b_1R$, we } can check that $|AV(x)+\alpha V(x)| \le \beta$ for $|x| \le R$, and in turn $V$ is a Lyapunov function :
\begin{equation} \label{E: vrai Lyapunov}
AV(x) \le -\alpha  V(x)+ \beta 1_{\{|x| \le R\}}, \quad \forall x \in \mathbb{R}^d.
\end{equation}

$\bullet$ We now prove that \eqref{E:WPI def} holds true with the same function $W$ and constant $C_{LP}$ for all coefficients $(a,b) \in \Sigma$. Using Proposition 3.6 in \cite{Bak_et_al08} with
\eqref{E: vrai Lyapunov}, it is sufficient to find $R'>0$ large enough such that
\begin{align} \label{E: construct W cond1}
	& \{V \le 2\beta/\alpha\} \subset B(0,R')
	\\\label{E: construct W cond2}
	&\pi(B(0,R')) > 1/2
\end{align}
and the local Poincaré inequality is valid on $B(0,R')$ with some constant $\kappa_{R'}$:
for all $f \in \mathcal{C}^1$ with bounded derivative,
\begin{equation}\label{E: construct W cond3}
	\int_{B(0,R')} f^2(x)\pi(x)dx
	\le \kappa_{R'}\int_{\mathbb{R}^d} \Gamma(f) \pi(x)dx + \frac{1}{\pi(B(0,R'))}
	\left(\int_{B(0,R')} f(x) \pi(x)dx \right)^2.
\end{equation}
If the conditions \eqref{E: construct W cond1}--\eqref{E: construct W cond3} are valid, then by Proposition 3.6 in \cite{Bak_et_al08} we {\modar deduce the Lyapunov-Poincar\'e Inequality}
\eqref{E:WPI def} with $W=V+(\beta \kappa_{R'}-1)_+$ and 
$1/C_{LP}=2\alpha(1-\frac{\pi(B(0,R')^c)}{\pi(B(0,R'))})\times(1+(\beta \kappa_{R'}-1)_+ )^{-1}$.

To get \eqref{E: construct W cond1}, using the expression of $V$ it is sufficient to take $R'\ge (\frac{1}{\varepsilon_0} \ln (2\beta/\alpha) )\vee 1$. Considering \eqref{E: construct W cond2}, we take the expectation with respect to $\pi$ in 
\eqref{E: vrai Lyapunov} and obtain 
$$
0=\int_{\mathbb{R}^d} A V(x) \pi(x) dx \le -\alpha \int_{\mathbb{R}^d} V(x) \pi(x) dx
+ \beta \int_{\mathbb{R}^d} \pi(x) dx,
$$
{\modar and thus
	\begin{equation} \label{E: borne unif L1 Lyapunov}
		\int_{\mathbb{R}^d} V(x) \pi(x) dx \le \frac{\beta}{\alpha}.
	\end{equation}
	We deduce 
{\modarn 
	\begin{equation}\label{E: moment expo pi}
		\int_{|x| \ge 1} e^{\varepsilon_0|x|} \pi(x) dx \le \int_{\mathbb{R}^d} V(x) \pi(x) dx\le \frac{\beta}{\alpha}.		
\end{equation}}
	 Using the Markov inequality, this yields, for any $R' \ge 1$, $\pi(\{x \mid |x|\ge R'\}) \le e^{-\varepsilon_0 R'} \frac{\beta}{\alpha}$.}
{\modar It entails that $R'>\frac{1}{\varepsilon_0} \ln (2\beta/\alpha) $ is sufficient for the condition \eqref{E: construct W cond2} to hold true.} 
We set $R'=(\frac{1}{\varepsilon_0} \ln (2\beta/\alpha)) \vee 1 $, and now we have to check  the condition \eqref{E: construct W cond3}. From the Poincaré inequality on the ball $B(0,R')$ endowed with the 
Lebesgue measure (see e.g. Theorem 4.9 in \cite{Evans_Gariepy_book}) and the Proposition 4.2.7 in \cite{Bakry_et_al_book}, we know that if 
\begin{equation}\label{E:borne pi preuve mixing}
	1/c_\pi \le \pi(x) \le c_\pi, \quad \forall x \in B(0,R'), 
\end{equation}
for some $c_\pi>0$ 
then,
$$
	\int_{B(0,R')} f^2(x)\pi(x)dx
\le C c_\pi^3 R' \int_{B(0,R')} |\nabla f|^2 \pi(x)dx + \frac{1}{\pi(B(0,R'))}
\left(\int_{B(0,R')} f(x) \pi(x) dx\right)^2,
$$
where $C$ is some universal constant. As the matrix $\tilde{a}$ is lower bounded by Assumption {\bf A1}, we have $  |\nabla f|^2 \le {a_{min}^{-1} \Gamma(f)}$, and we deduce that \eqref{E: construct W cond3} holds true with
$\kappa_R'=C R' c_\pi^3$. Consequently, by Proposition 3.6 in \cite{Bak_et_al08},
the Lyapunov-Poincaré inequality holds true. Moreover, the constant $C_{LP}$ in \eqref{E:WPI def} is independent of $(a,b) \in \Sigma$, as soon as
we can find a constant $c_\pi$ in \eqref{E:borne pi preuve mixing} independent of $(a,b)$. 
{\modarn Using the invariance of $\pi$ and
	\eqref{eq: bound Pesce} with $t=1$ and $s=0$, gives for $x \in B(0,R')$,
	\begin{align}\nonumber
		C_0^{-1} \int_{B(0,R')} \pi(y) e^{- \lambda_0^{-1}|\theta_1(x)-y|^2} dy &\le \pi(x)\le C_0 \int_{\mathbb{R}^d} \pi(y)  dy 
		\\ \nonumber
		{\modch C_0^{-1}} \int_{B(0,R')} \pi(y) e^{- 4\lambda_0^{-1} R'^2} dy &\le \pi(x)\le C_0 \int_{\mathbb{R}^d} \pi(y) dy
		\\ \label{E:mino majo pi}
		{\modch \frac{C_0^{-1}}{2}} e^{- 4\lambda_0^{-1} R'^2} &\le \pi(x)\le C_0 ,
	\end{align}
	where in the second line we used $|\theta_1(x)| \le |x|$ for $|x|\ge 1$, and in the 
	last line we used \eqref{E: construct W cond2}. 
	
	Hence, we have proved the W-Lyapunov-Poincaré inequality \eqref{E:WPI def} with the same function $W$ and constant $C_{LP}$ for all $(a,b) \in \Sigma$. To see that we are in the scope of the Definition \ref{def: LPI} we need to check that $W$ and $AW$ belong to $\mathbf{L}^2(\pi)$. From
	\eqref{E: moment expo pi} we know that the stationary measure integrates the exponential function $x\mapsto e^{\epsilon_0 |x|}$ for any constant
	$\varepsilon_0 \in (0,\frac{\tilde{C}}{2d^2a_0^2})$. As the constant $\varepsilon_0$ can be chosen arbitrarily small, it suffices to choose
	$\varepsilon_0 < \frac{\tilde{C}}{6d^2a_0^2}$ to deduce that $W\in \mathbf{L}^2(\pi)$. As $b$ has at most linear growth, we also deduce $AW \in \mathbf{L}^2(\pi)$.
}

$\bullet$ Eventually, we deduce \eqref{E:eq control L2 Pt} by applying Lemma \ref{L:lemme de Bakry}, {\modar together with the upper bound 
	\begin{equation*}
		\int_{\mathbb{R}^d} W(x) \pi(x) dx \le 
		\int_{\mathbb{R}^d} V(x) \pi(x) dx  + (\beta \kappa_{R'}-1)_+\le 
		\beta/\alpha+(\beta \kappa_{R'}-1)_+
	\end{equation*}
	where in the last inequality we used \eqref{E: borne unif L1 Lyapunov}.}
\end{proof}
%
{\revar 
	\begin{remark}
		Remark that the theory of W-Lyapounov-Poincaré inequality could be applied to more general processes than continuous diffusions. Indeed, following the proof of Lemma \ref{L: maj cov} we see that the two main ingredients for getting Lyapounov-Poincaré inequality are the existence of a classical Lyapounov function, and some lower bound on compact sets for the {\em carré du champ} of some operator $\Gamma(f)$ by $c |\nabla f|^2$. Both ingredients are possible to get, for instance, in the context of jump-diffusion processes. Adaptation of Lemma \ref{L:lemme de Bakry} seems also possible to the situation of jump-diffusion processes, taking into account that, in such a context, the {\em carré du champ} of the adjoint operator $A^{*,\pi}$ is not necessary equal to the {\em carré du champ} of the operator $A$.
	\end{remark}
	
	We now prove Lemma \ref{L:hyp sigma stronger} as a consequence of Lemma \ref{L: maj cov}.}
\begin{proof}
	{\modar First, remark that by a density argument,} we can assume in the proof that $\varphi$ is a smooth function supported on $K$.
	{\modar The inequality \eqref{E: resu mixing reg non reg} is a consequence of \eqref{E:eq control L2 Pt} in Lemma \ref{L:hyp sigma stronger}, but the latter requires that the coefficients of the S.D.E. are of class $\mathcal{C}^2$. Hence, an approximation of the initial S.D.E. by one with smoother coefficients is required.}
	For $(a,b)$ in $\Sigma=\Sigma(a_0,a_1,b_0,b_1,a_\text{min},\tilde{C}_{\text{b}},\tilde{\rho}_{\text{b}})$, we introduce {\modar the following} smooth approximations of $a$ and $b$. Let $\eta$ be a smooth function supported on the unit ball of $\mathbb{R}^d$ with $\int_\mathbb{R^d}\eta(x)dx=1$, and we set
	$a_n=a\star \eta_n$, $b_n=b\star \eta_n$, where $\star$ is the convolution operator and  $\eta_n(\cdot)=n\eta(n\cdot)$. As $(a,b) \in \Sigma$, we deduce that 
	\begin{equation}
	\label{E:approx_coeff}
	\norm{a-a_n}_\infty+\norm{b-b_n}_\infty \le \frac{c}{n}.
\end{equation} 
Moreover, for all $n$ large enough, $(a_n,b_n) \in \Sigma'=\Sigma(2a_0,2a_1,2b_0,2b_1,a_\text{min}/2,\tilde{C}_{\text{b}}/2,2\tilde{\rho'}_{\text{b}})$.

We denote by  $X^{n}$ the solution of the S.D.E. \eqref{eq: model} with the coefficients $(a_n,b_n)$  in place of $(a,b)$, and by $\pi_n$ the unique stationary distribution of $X^n$.
{\modar In the sequel of the proof, we also emphasize the dependence on the initial condition of the process \eqref{eq: model}  by denoting as $\E_{x}$ (resp. $\E_\pi$) the expectation computed when the process starts with the initial condition $X_0=x$ (resp.  $X_0\overset{\text{law}}{=}\pi$).}
Since $a_n$ and $b_n$ are smooth coefficients, by Lemma \ref{L: maj cov}, we have
\begin{equation}\label{E:control mixing sur Xn}
\int_{\mathbb{R}^d} \left[ {\modar \E_x[\varphi(X_t^{n})]} - \pi_n(\varphi) \right]^2 \pi_n(x) dx
	 \le ce^{-t/c}\norm{\varphi}_\infty^2
\end{equation}
	where the constant $c=c_{\Sigma'}$  is independent of $n$ and $(a,b) \in \Sigma$.
 Using classical estimates for solutions of stochastic differential equations with \eqref{E:approx_coeff}, it is possible to show
\begin{equation}\label{E:control X Xn}
\sup_{x \in \mathbb{R}^d}
{\modar \E_x\left[\abs{X^{n}_t-X_t}\right]} \le \frac{c' e^{c't}}{n}, \forall t>0, \forall n\ge 1,
\end{equation}
and with some constant $c'>0$.

Now, we write
\begin{align*}
{\modar \E_{\pi}\left[\varphi(X_t)\varphi(X_0) \right]}&=\int_{\mathbb{R}^d} 
	P_t(\varphi)(x) \varphi(x) \pi(x)dx	
	\\
	&= \int_{\mathbb{R}^d} {\modar E_x[ \varphi(X^{n}_t)]} \varphi(x) \pi(x)dx	+
	O( \frac{c'e^{c't}}{n} \norm{\varphi}_\infty \norm{\varphi'}_\infty ),
\end{align*}
where we used \eqref{E:control X Xn}. 
On the compact $K$ we can find a constant $c_K$ such that
$c_K^{-1}\le \pi \le c_K$,  $c_K^{-1}\le \pi_n \le c_K$ (in the same way as we obtained \eqref{E:mino majo pi}), using \eqref{E:control mixing sur Xn}, we deduce 
{\modar
	\begin{equation} \label{E: ergo L2 vers pi_n}
		\int_K  \abs{ {\modar \E_x[ \varphi( X_t^{n} ) ]} - \pi_n(\varphi)}^2 \pi(x) dx \le c_K^2 c e^{-t/c} \norm{\varphi}_\infty^2.
\end{equation} }
{\modar It follows}
\begin{align}
	\nonumber
	{\modar \E_{\pi}\left[\varphi(X_t)\varphi(X_0) \right]}&=
	\int_{K}  \pi_n(\varphi) \varphi(x) \pi(x)dx	
	\\ & \nonumber
	+ O \left( \int_K  | {\modar \E_x[ \varphi(X^{n}_t)]}  - \pi_n(\varphi)| \varphi(x)\pi(x)dx\right)
	+	O( \frac{c'e^{c't}}{n} \norm{\varphi}_\infty \norm{\varphi'}_\infty )
	\\ \label{E:covariance varphi smooth}
	&= \pi_n(\varphi) \pi(\varphi)
	+ O( c_K \norm{\varphi}_\infty^2 \sqrt{c} e^{-t/(2c)})
	+	O( \frac{c'e^{c't}}{n} \norm{\varphi}_\infty \norm{\varphi'}_\infty ){\modar ,}
\end{align}
{\modar where in the last line we used \eqref{E: ergo L2 vers pi_n}.}
It remains to find some upper bound on $\abs{\pi_n(\varphi)-\pi(\varphi)}$. We write for $t>0$
$$
\pi_n(\varphi)-\pi(\varphi)=\pi_n(\varphi)-{\modar \E_x(\varphi(X_t^{n}))} + {\modar \E_x(\varphi(X_t^{n}))}-P_t(\varphi)(x)+
P_t(\varphi)(x)-\pi(\varphi)
$$
and integrate on the compact set $K$ with respect to the Lebesgue  measure  to find
\begin{multline*}
	\text{vol}(K) | \pi_n(\varphi)-\pi(\varphi)| 
	\le c_K \norm{
		\pi_n(\varphi)- {\modar 
			\E_x(\varphi(X_t^{n}))}}_{\mathbf{L}^1(\pi_n)} 
		+
	\\
	\int_K \abs{ {\modar \E_x(\varphi(X_t^{n}))}-P_t(\varphi)(x)} dx +
	c_K \norm{P_t(\varphi)(x)-\pi(\varphi)}_{\mathbf{L}^1(\pi)},
\end{multline*}
where we used that on $K$, $\pi$ and $\pi_n$ are lower bounded by $1/c_K$.
From \eqref{E:control X Xn}, we deduce
\begin{multline*}
	\text{vol}(K) | \pi_n(\varphi)-\pi(\varphi)| 
	\le c_K \norm{\pi_n(\varphi)-{\modar \E_x(\varphi(X_t^{n}))}}_{\mathbf{L}^1(\pi_n)}+
	\\
	\text{vol}(K) \frac{c' e^{c't}}{n} +
	c_K \norm{P_t(\varphi)(x)-\pi(\varphi)}_{\mathbf{L}^1(\pi)},
\end{multline*}
In the last equation, we specify $t=\sqrt{\log(n)}$. The first term on the right hand side goes to zero by \eqref{E:control mixing sur Xn}, the second one goes to zero {\modar immediately,} while the last one goes to zero by the mixing property of the process $X$ (see Lemma \ref{lemma: ergodicity}). We deduce that
\begin{equation} \label{E: control pi pi_n}
	| \pi_n(\varphi)-\pi(\varphi)| \le \varepsilon_n(\varphi,K,a,b),
\end{equation}
for some sequence $\varepsilon_n(\varphi,K,a,b) \xrightarrow{n \to \infty} 0$ (let us stress that
this convergence is not uniform with respect {\modar to} $(a,b) \in \Sigma$, $K$, or the function $\varphi$).

{\modar Gathering} \eqref{E:covariance varphi smooth} {\modar and} \eqref{E: control pi pi_n}, we have
	\begin{equation*}
\abs{{\modar \E_{\pi}}\left[\varphi(X_t)\varphi(X_0) \right]
	- \pi(\varphi)^2} \le
	c_K^2 \norm{\varphi}_\infty^2 \sqrt{c} e^{-ct/2}+
	\frac{c'e^{c't}}{n} \norm{\varphi}_\infty \norm{\varphi'}_\infty
	+\norm{\varphi}_\infty \varepsilon_n(\varphi,K,a,b)	
	\end{equation*}
	where the constant $c$ does not depend on $n$.
	Letting $n \to\infty$, we deduce \eqref{E: resu mixing reg non reg}.
	
\end{proof}

\end{document}